\documentclass{article}
\usepackage[british]{babel}
\usepackage{soul}
\usepackage[all]{xy}
\SelectTips{cm}{}

\newdir{|>}{!/2.3pt/@{|}*:(.5,.05)@^{>}*:(.5,-.05)@_{>}}
\newdir{ >}{{}*!/-5pt/\dir{>}}

\usepackage{pdflscape}
\usepackage{amsthm,amssymb,mathrsfs,xspace,graphicx,xcolor,amsmath,enumerate}
\usepackage[mathscr]{euscript}

\edef\cdrestoreat{
\noexpand\catcode\lq\noexpand\@=\the\catcode\lq\@}\catcode\lq\@=11
\newcommand\eatit[1]{\relax}
\newcommand\xstop{\@ifnextchar{.}{\eatit}{\mbox{}\xspace}}
\renewenvironment{proof}[1][\proofname]{\par
  \pushQED{\qed}%
  \normalfont \topsep6\p@\@plus6\p@\relax
  \trivlist
  \item[\hskip\labelsep
        \sc
    #1\@addpunct{.}]\ignorespaces
}{%
  \popQED\endtrivlist\@endpefalse
}

\newlength{\gnCornerHgt}
\newlength{\gnArgHgt}
\newcommand\os[1]{\ensuremath{\mathinner{\mathpalette\Godelnum@{#1}}}}
\newcommand\Godelnum@[2]{%
 \begingroup
 \settoheight{\gnCornerHgt}{$\m@th#1\ulcorner$}%
 \settoheight{\gnArgHgt}{$\m@th#1#2$}%
 \ifdim\gnArgHgt<.9\gnCornerHgt
  \setlength{\gnArgHgt}{0pt}%
 \else
  \addtolength{\gnArgHgt}{-\gnCornerHgt}%
  \addtolength{\gnArgHgt}{1.5pt}%
 \fi 
 \raisebox{\gnArgHgt}{$\m@th#1\ulcorner$}%
\mkern-4mu#2\mkern-2mu%
 \raisebox{\gnArgHgt}{$\m@th#1\urcorner$}%
 \endgroup
}
\newcommand\unos[1]{\ensuremath{\lfloor#1\rfloor}}
\cdrestoreat

\mathcode`\:="603A 
\mathchardef\colon="303A 

\newcommand\stndop[1]{\ensuremath{\mathrm{\normalfont#1}}\xspace}
\DeclareFontFamily{U}{min}{}
\DeclareFontShape{U}{min}{m}{n}{<-> udmj30}{}
\newcommand\Yo{\text{\!\usefont{U}{min}{m}{n}\symbol{'210}\!}}
\newcommand\yo{\stndop{y}}

\DeclareFontFamily{OT1}{pzc}{}
\DeclareFontShape{OT1}{pzc}{m}{it}{<->s*[1.14]pzcmi7t}{}
\DeclareMathAlphabet{\mathpzc}{OT1}{pzc}{m}{it}

\newtheorem{thm}{Theorem}[section]
\newtheorem{lem}[thm]{Lemma}
\newtheorem{prp}[thm]{Proposition}
\newtheorem{cor}[thm]{Corollary}
\theoremstyle{definition}
\newtheorem{dfnt}[thm]{Definition}
\newtheorem{rem}[thm]{Remark}
\newtheorem{ntn}[thm]{Notation}
\newtheorem{exm}[thm]{Example}
\newtheorem{exms}[thm]{Examples}

\newcommand\entail[2]{\noindent\ensuremath
 {\mbox{#1}\mathord{\Longrightarrow}\mbox{#2}}}
\newcommand\nsc[2]{\noindent\ensuremath
 {\mbox{#1}\mathord{\Longleftrightarrow}\mbox{#2}}}

\newcommand\nameit[1]{#1~}
\newcommand\thx{\nameit{Theorem}}
\newcommand\lmx{\nameit{Lemma}}
\newcommand\prx{\nameit{Proposition}}
\newcommand\crx{\nameit{Corollary}}
\newcommand\dfx{\nameit{Definition}}

\newcommand\rmx{\nameit{Remark}}

\newcommand\xsx{\nameit{Examples}}

\newcommand\dfn[1]{{\bfseries\itshape #1\/}}
\newcommand\ie{i.e.\ }
\newcommand\eg{e.g.\ }

\newcommand\dom{\stndop{dom}}

\newcommand\identity{\stndop{id}}
\newcommand\id[1]{\ensuremath{\identity_{#1}}}
\newcommand\Identity{\stndop{Id}}
\newcommand\Id[1]{\ensuremath{\Identity_{#1}}}
\newcommand\op{{}^{\textrm{\normalfont\scriptsize op}}}

\newcommand\nat{\ensuremath{\mathbb{N}}\xspace}

\newcommand\des[1]{\ensuremath{\descentdata({#1})}}
\newcommand\descentdata{\stndop{Des}}

\newcommand\ct[1]{\ensuremath{\mathpzc{#1}}\xspace}
\newcommand\Ct[1]{\ensuremath{{\normalfont\textsf{\bfseries #1}}}\xspace}
\newcommand\Pred[1]{{\ensuremath{{\mathpzc{Prd}\kern-.4ex_{{#1}}}}}}
\newcommand\TZ{\ensuremath{\ct{Top}_{\kern-.2ex0}}\xspace}

\newcommand\Fam[1]{\ensuremath{\stndop{Fam}(#1)}}

\newcommand\RB[1]{\mathchoice
 {\rotatebox[origin=c]{180}{$#1$}}
 {\rotatebox[origin=c]{180}{$#1$}}
 {\rotatebox[origin=c]{180}{$\scriptstyle#1$}}
 {\rotatebox[origin=c]{180}{$\scriptscriptstyle#1$}}}
\newcommand\D{\RB{E}\kern-.3ex}
\newcommand\B{\RB{A}\kern-.6ex}
\newcommand\Forall[1]{\forall\kern-.2ex_{#1}}
\newcommand\Exists[1]{\exists\kern-.1ex_{#1}}

\newcommand\thr[1]{\ensuremath{\mathcal{#1}}}
\newcommand\lowstndop[1]{_{\stndop{#1}}}

\newcommand\exl{\lowstndop{ex/lex}}

\newcommand\exw{\lowstndop{ex/wl}}
\newcommand\rgw{\lowstndop{reg/wl}}
\newcommand\cmpr{\lowstndop{c}}
\newcommand\extcmpr{\lowstndop{cx}}
\newcommand\ext{\lowstndop{x}}

\newcommand\pw[1]{\mathscr{P}
 {\ifx\relax#1\relax\else{(#1)}\fi}}

\newcommand\udl[1]{\ensuremath{\left|{#1}\right|}}
\newcommand\postc[2]{\ensuremath{\Sigma_{#1}
 {\ifx\relax#2\relax\else{#2}\fi}}}
\newcommand\qd[1]{\ensuremath{\widehat
 {\ifx\relax#1\relax(\blank)\else{#1}\fi}}} 
\newcommand\qb[1]{\ensuremath{\ct{Q\mkern1mu}
 _{\ifx\relax#1\relax(\blank)\else{#1}\fi}}}
\newcommand\ecd[1]{\ensuremath{\widehat
 {\ifx\relax#1\relax(\blank)\else{#1}\fi}}} 
\newcommand\ecb[1]{\ensuremath{\ct{T\mkern-3mu}
 _{\ifx\relax#1\relax(\blank)\else{#1}\fi}}}
\newcommand\fb[1]{\ensuremath{\ct{F\mkern-4mu}
 _{\ifx\relax#1\relax(\blank)\else{#1}\fi}}}
\newcommand\fd[1]{\ensuremath{{\ifx\relax#1\relax(\blank)\else
 {#1}\fi}^{+}}}
\newcommand\kf[1]{\ensuremath{K
 _{\ifx\relax#1\relax(\blank)\else{#1}\fi}}}
\newcommand\ec[1]{\ensuremath{\mathord{\left[{#1}\right]}}}
\newcommand\blank{\mathrm{-}}

\newcommand\ple[1]{\ensuremath{\langle #1\rangle}}
\newcommand\pow[1]{\ensuremath{\nf{P}({#1})}}
\newcommand\projection{\stndop{pr}}
\newcommand\pr[1]{\ensuremath{\projection_{#1}}}
\newcommand\injection{\stndop{in}}
\newcommand\inj[1]{\ensuremath{\injection_{#1}}}

\newcommand\nf[1]{\ensuremath{{\mathfrak{#1}}}}

\newcommand\tot[1]{\ensuremath{\mathord{!}_{#1}}}

\newcommand\Land{\wedge}
\newcommand\Lor{\vee}
\newcommand\Implies{\Rightarrow}
\newcommand\Iff{\Leftrightarrow}

\newcommand\heps{\ensuremath{\epsilon}\xspace}
\newcommand\epso{\heps-}

\renewcommand\to{\xymatrix@1@C=1.3em{{}\ar[r]&{}}}
\newcommand\tos{\xymatrix@1@C=.6em{{}\ar[r]&{}}}
\newcommand\emar{\ar@<-.2ex>@{^{(}->}}
\newcommand\emarb{\ar@<.2ex>@{_{(}->}}
\newcommand\emto{\xymatrix@1@C=1.3em{{}\ \emar[r]&{}}}
\newcommand\ntar{\ar|>{}|-*=0[@]{{\raisebox{.8ex}{\scalebox{.4}{$\bullet$}}}}}

\newcommand\ntto{\xymatrix@1@C=1.3em{{}\ntar[r]&{}}}
\newcommand\isar{\ar|>{}|-*=0[@]{\widetilde{\kern1ex}}}
\newcommand\isto{\xymatrix@1@C=1.3em{{}\isar[r]&{}}}

\newcommand\onar{\ar@{-{|>}}}
\newcommand\onto{\xymatrix@1@C=1.3em{{}\onar[r]&{}}}

\newcommand\suar{\ar@{->>}}

\newcommand\moar{\ar@{ >->}}
\newcommand\moto{\xymatrix@1@C=1.3em{{}\moar[r]&{}}}

\newcommand\pmar{\ar@{-^{`}}}
\newcommand\pmto{\xymatrix@1@C=1.3em{{}\pmar[r]&{}}}

\newcommand\fuar{\ar@<-.2ex>}
\newcommand\futo{\xymatrix@1@C=1.6em{{}\fuar[r]&{}}}

\newcommand\CTT{\ensuremath{\mathrm{CIC}}\xspace}

\newcommand\proj[1]{\ensuremath{\ct{Proj}\mkern-1mu
 _{\ifx\relax#1\relax(\blank)\else{#1}\fi}}}
\newcommand\perb[1]{\ensuremath{\ct{PER}
 _{\ifx\relax#1\relax(\blank)\else{#1}\fi}}}
\newcommand\perd[1]{\ensuremath{\qd{{#1}\extcmpr}}}
\newcommand\crb[1]{\ensuremath{\ct{Hc}
 _{\ifx\relax#1\relax(\blank)\else{#1}\fi}}}
\newcommand\crd[1]{\ensuremath{{\ifx\relax#1\relax(\blank)\else{#1}\fi}^{-}}}
\newcommand\Ob[1]{\ensuremath{{#1}_0}}

\newcommand\Sbm[2]{\ensuremath{{#1}
 \ifx\relax#2\relax\else_{#2}\fi}}
\newcommand\Sub{\stndop{Sub}}
\newcommand\Sb{\Sbm{\Sub}}
\newcommand\WeakSub{\Psi}
\newcommand\Wsb{\Sbm{\WeakSub}}
\newcommand\Strongsub{\stndop{St}}
\newcommand\Stg{\Sbm{\Strongsub}}

\newcommand\Karoubi{K}
\newcommand\K[1]{\ensuremath{\Karoubi
 {\ifx\relax#1\relax\else{(#1)}\fi}}}
\newcommand\Exp{\stndop{Exp}}

\newcommand\ev{\stndop{ev}}

\newcommand\ruc{{\normalfont(RUC)}\xspace}
\newcommand\rc{{\normalfont(RC)}\xspace}
\newcommand\dis{{\normalfont(nul)}\xspace}
\newcommand\dtc{{\normalfont(dise)}\xspace}
\newcommand\exh{{\normalfont(disj)}\xspace}
\newcommand\lcc{{LCC}\xspace}
\newcommand\eqc{{e.q.c.}\xstop}
\newcommand\ed{{e.d.}\xstop}
\newcommand\eed{{e.e.d.}\xstop}
\newcommand\dd{{d.d.}\xstop}
\newcommand\fo{{f.o.}\xstop}

\newcommand\pai[2]{\ensuremath{\langle #1\rangle_{#2}}}
\newcommand\uni[1]{\ensuremath{\mathrm{o}^{#1}}\xspace}
\newcommand\type{\mathbin{:}}

\newcommand\seq[1]{\mathrel{\mathord{=}_{\scriptscriptstyle #1}}}
\newcommand\Seq[3]{{#2}\seq{#1}{#3}}
\newcommand\Frac[2]{\ensuremath{{\displaystyle\frac{{#1}}{{#2}}}}}
\newcommand\cmp[1]{\ensuremath{\{\kern-2.5pt|
 {\ifx\relax#1\relax\blank\else{#1}\fi}|\kern-2.5pt\}}}
\newcommand\comc[2]{\ensuremath{(#1\mathbin{\downarrow}#2)}}
\newcommand\ogp{\ensuremath{\Xi}\xspace}
\newcommand\gp{\ensuremath{\mathsf{\bf Pr}}\xspace}
\newcommand\pgp{\ensuremath{\Leftrightarrow}\xspace}
\newcommand\sng[1]{\ensuremath{S_{#1}}}
\newcommand\mbs[1]{\ensuremath{\ni_{#1}}\xspace}
\newcommand\rmbs[1]{\mathrel{\mbs{#1}}}
\newcommand\cng{\mathbin{\mathchoice
    {\textsf{\small\&}}
    {\textsf{\small\&}}
    {\textsf{\scriptsize\&}}
    {\textsf{\tiny\&}}}}
\newcommand\tru{\mathord{{\mathsf{tt}}}}

\newcommand\eqv{\ensuremath{\mathbin{\leftrightarrow}}\xspace}

\renewcommand\phi{\varphi}
\newcommand\glb{\stndop{glb}}
\newcommand\ler{\stndop{ler}}
\newcommand\ier{\stndop{ER}}

\newcommand\ciwe[1]{{#1} produces enough relative weak
 exponentials\xspace} 

\newcommand\set[2]{\ensuremath{\left\{#1\left|\,\strut{#2}\right.\right\}}}
\newcommand\restrict[2]{\ensuremath{#1\mathord{\restriction}_{#2}}}

\newcommand\propo{\kern.5ex\mbox{\normalfont\it prop}\kern.5ex}
\newcommand\propos{\kern.5ex\mbox{\normalfont\it prop}_s\kern.5ex}
\date{}

\begin{document}

\title{Quasitoposes as elementary quotient completions}
\author{Maria Emilia Maietti\thanks{%
Dipartimento di Matematica ``Tullio Levi Civita'',
Universit\`a di Padova,
via Trieste 63, 35121 Padova, Italy,
email:~\texttt{maietti@math.unipd.it}}
\and
Fabio Pasquali\thanks{%
Dipartimento di Matematica ``Tullio Levi Civita'',
Universit\`a di Padova,
via Trieste 63, 35121 Padova, Italy,
email:~\texttt{pasquali@dima.unige.it}}
\and
Giuseppe Rosolini\thanks{%
DIMA, Universit\`a di Genova,
via Dodecaneso 35, 16146 Genova, Italy,\hfill\mbox{}
email:~\texttt{rosolini@unige.it}}}

\maketitle
\begin{abstract} 
We characterise when the elementary quotient completion of an elementary
doctrine, introduced in previous work by the first and third authors, is a
quasitopos. Since the elementary quotient completion generalises the notion
of exact completion of a cartesian category with weak equalisers to the
setting of elementary doctrines, as proposed by F.W.~Lawvere, the
characterisation established in this paper, generalises previous results on
exact completions, which include A.~Carboni and E.~Vitale’s characterisation
in terms of projective objects, Carboni and G.~Rosolini’s characterisation
of locally cartesian closed exact completions, as revised by J.~Emmenegger,
and M.~Menni’s characterisation of exact completions that are toposes.
\end{abstract}
\section{Introduction}
The present paper is the culmination of a series of works by the authors on
quotient completions of elementary doctrines, within the category-theoretical
framework to logic proposed by F.W.~Lawvere in
\cite{LawvereF:adjif,LawvereF:equhcs}, highlighting the expressive power of
this concept.

The study of constructions that freely adjoin quotients to a category
is a central topic in both mathematics and computer science. In category
theory, closely related notions include the exact completion of a category
with finite, possibly weak, limits and the notion of a regular category; see
\cite{CarboniA:freecl,CarboniA:regec}. Those notions have been extensively
studied and applied; see
\cite{Lack2001203,Johnstone200187,Hofstra2004129,Gran20141771,vandenBerg20183137,Frey2023,Awodey2025},
to mention just a few examples. 

In \cite{MaiettiME:quofcm}, the first and third authors generalised the
notion of exact completion of a cartesian category with weak equalisers to
the notion of elementary quotient completion of an elementary doctrine,
as in Lawvere's \cite{LawvereF:adjif,LawvereF:equhcs}. That construction
provides a universal way to complete an elementary doctrine with respect to
a canonical notion of quotient.

The exact completion of a category \ct{C} with finite products and weak
equalisers, presented in \cite{CarboniA:freecl,CarboniA:regec}, is an
instance of this construction: the doctrine of subobjects of the exact
completion is the elementary quotient completion of the doctrine of
variations of \ct{C}---the notion of variations, proposed in
\cite{LawvereF:adjiab}, was formally introduced in \cite{GrandisM:weasec}.

An analogous generalisation of the notion of regular completion of a
cartesian category with weak equalisers, formulated in terms of elementary
and existential doctrines, was studied in \cite{MaiettiME:triech}.

Further developments and applications of the quotient completions can be
found in
\cite{MaiettiME:triech,MaiettiME:eleqcc,quarit,EmmeneggerJ:eledac,DagninoP26,trip}.

One of the main results of the paper in \S8, is the characterisation in
\thx\ref{quasitop} of those elementary doctrines whose elementary quotient
completion yields a quasitopos, as well as those whose completion yields an
arithmetic quasitopos, that is, a quasitopos equipped with a natural numbers
object. The characterisations are formulated in terms of specific
triposes, introduced following \cite{PittsA:tritir}.

Further consequences of the characterisation \thx\ref{quasitop} include
\thx\ref{quapers} and \thx\ref{singt} with its corollaries, which provide
criteria for determining when the elementary quotient completion yields an
elementary topos. The latter results, presented in \S9, are analogous to
those in \cite{HylandJ:efft,RosoliniG:colcat} and clarify the relationship
between elementary quotient completions and the well-known tripos-to-topos
construction of \cite{HylandJ:trit}, as analysed in \cite{FreyJ:triqat}.

Establishing those characterisations requires extending several known
results on exact completions, which include
\begin{itemize}
\item Carboni and Vitale's
characterisation of exact completions of weakly finitely complete categories
in terms of projective objects in \cite{CarboniA:regec}, given in
\thx\ref{PPO} in \S3;
\item Carboni and Rosolini's characterisation of locally cartesian closed
exact completions \cite{RosoliniG:loccce}, presented in \thx\ref{fstform} in
\S6; and
\item Menni's characterisation of exact completions that are toposes
\cite{MenniM:chalec}, as \thx\ref{tclassifier} in \S7.
\end{itemize}
In particular, the characterisation of local cartesian closure in \S6
provides an alternative proof of the result established in
\cite{MaiettiME:loccce}. The new proof presented here relies only
on the property of cartesian closedness of the base and the existence of a
projective cover for it, and uses a descent argument. It also 
makes use of the notion of a topology between doctrines, introduced in \S3.

The remaining sections are organised as follows. In \S2, we recall the
notion of elementary quotient completion and related concepts. In \S5,
we characterise when an elementary quotient completion is distributive and
disjunctive, a result that plays a key role in establishing the main
theorems in \S8. 

Notable examples of arithmetic quasitoposes arising as non-exact elementary
quotient completions include the category of equilogical spaces
\cite{ScottD:dattl,ScottD:newcds,BauerA:equs}, the category of assemblies
over a partial combinatory algebra \cite{HylandJ:efft,OostenJ:reaait}, and
the category of total setoids, in the style of E.~Bishop, over the dependent
type theory known as the Calculus of Inductive Constructions of Th.~Coquand,
G.~Huet, and C.~Paulin-Mohring, see \cite{CoquandT:calc,CP89}.

It is worth noting, before closing the introduction, that both proof
assistants Rocq and Lean, are based on the Calculus of Inductive
Constructions, but differ in their treatment of quotients: Rocq relies on
setoids, Lean provides built-in quotient types. In future work, we aim to
apply elementary quotient completion and the results of this paper to relate
Rocq and Lean through their categorical semantics.

\section{The elementary quotient completion}

As explained in the introduction, we shall concentrate on the properties
of the elementary quotient completion of an elementary doctrine---a
construction introduced in \cite{MaiettiME:quofcm}.

For ease of reference, we recall this construction together with a few
preliminary concepts concerning certain indexed posets, which we call
{\it doctrines} (modified by suitable adjectives), in analogy with Lawvere's
notion in \cite{LawvereF:adjif,LawvereF:equhcs}. They provide an alternative
presentation of certain faithful fibrations in the sense of
\cite{JacobsB:catltt}, via an evident essential equivalence of 2-categories
that directly extends the usual correspondence between indexed families and
functions.

A \dfn{primary doctrine} is an indexed poset $P:\ct{C}\op\futo\Ct{Pos}$
on a category \ct{C} with finite products such that each fibre
$P(B)$ is an inf-semilattice and the families of operations
$\xymatrix{P(B)\times P(B)\ar[r]^-{\Land_B}&P(B)&1\ar[l]_-{\top_B}}$
are natural in \ct{B}.

The category $\ct{C}$ is often called \dfn{base} of the doctrine. We
say that $\alpha$ is \dfn{over} $B$ when $\alpha$ is an element of
$P(B)$.

\begin{ntn}
At times, we write the action of the functor $P:\ct{B}\op\futo\Ct{Pos}$
on an arrow $f$ as $P_f$ instead of $P(f)$. We may also drop the
subscripts in the operations on the fibres.
\end{ntn}

We refer the reader to
\cite{MaiettiME:eleqc,MaiettiME:quofcm,MaiettiME:relqcv,
PasqualiF:cofced} for examples of primary doctrines. Here we limit ourselves
to recall three examples which shall recur in the following.

The first example---possibily motivating large part of the studies---is the
doctrine $\Wsb{\ct{C}}:\ct{C}\op\futo\Ct{Pos}$ determined on a
category \ct{C} with finite products and weak pullbacks. The base
category is \ct{C}, and the fibre on the object $C$ is the poset
reflection of the slice category \comc{\ct{C}}{C}. Reindexing is
obtained by weak pullback. It is an instance of variational doctrine,
see \cite{LawvereF:adjiab,MaiettiME:triech}. The term ``variation''
was introduced in \cite{GrandisM:weasec} for an element in the fibre. 

The second is probably the most common: if a category \ct{C} with finite
products is endowed with a collection $\mathcal{M}$ of monos closed under
identities and pullback, then this determines a primary doctrine 
$\Sbm{\mathcal{M}}{\ct{C}}:\ct{C}\op\futo\ct{Pos}$ on the base category
\ct{C} where the fibre on the object $X$ consists of (the poset reflection
of) $\Sbm{\mathcal{M}}{\ct{C}}(X)$ of the monos in $\mathcal{M}$ with
codomain $C$, with reindexing given by pulling back. There is, in
particular, the primary doctrine $\Sb{\ct{C}}:\ct{C}\op\futo\Ct{Pos}$
of \emph{all} subobjects of \ct{C}, obtained with the collection
$\mathcal{T}$ of all monos when \ct{C} has pullbacks. Other instances
are obtained using the class of monos of a stable factorisation
system, see \cite{FreydP:catcf}.

\begin{rem}\label{logshand}
Note that, since the base category \ct{B} has finite products, there is an
algebraic theory associated to it, as shown in
\cite{AdamekJ:algt}. And there is a logical theory that handles the
(in)equalities in the fibres, as shown in \cite{PittsAM:catl}  where
the notion of primary doctrine goes under the name of ``prop-category
with finite meets''. We briefly recall those, mainly to introduce the
notation we shall employ.

The logical theory $\thr{T}_{P}$ is with sorts: one sort \os{B}
for each object $B$ of \ct{B}. Terms are constructed from variables
$x\colon B$, and operations symbols: a unary operation symbol \os{f} for
each arrow $f:B\to D$ in \ct{B} which can be applied to terms $t$ of sort
\os{B} to produce a term $\os{f}(t)$ of sort \os{D}; a constant symbol
\uni{\mathbf{t}} of sort \os{T} for each terminal object $T$; and a binary
operation symbol \pai{\blank,\blank}{\mathbf{p}} for each product diagram 
$\mathbf{p}\colon=\xymatrix@1@C=4em{B_1&B
\ar[l]|-{\pr1}\ar[r]|-{\pr2}&B_2}$ in \ct{B}, which can be applied to
a pair of terms, the first of sort \os{B_1} and the second of sort
\os{B_2}, to produce a term of sort \os{B_2}.
Atomic formulas of $\thr{T}_{P}$ are built from terms and relation symbols:
a unary relation symbol \os{\alpha} for each element $\alpha$ in a fibre 
$P(B)$ which takes terms of sort $B$.

The theory $\thr{T}_{P}$ is written in the $\mathord{\Land}$-fragment of
intuitionistic logic, using the entailment relation with contexts and a
single formula both in the antecedent and in the consequent.

The axioms of $\thr{T}_{P}$ can be presented in a finitary form. But there
is also an extremely simple, direct presentation. By induction on the
construction of formulas, define the \dfn{unquoting} of $\phi$ as follows.
Write $\Gamma$ for the context
$x_1\type\os{C_1},\ldots,x_1\type\os{C_n}$. 

First unquoting for terms:
\[\begin{array}{l}
\unos{\Gamma\mid x_i: C_i}=\pr{i}:C_1\times\ldots\times C_n\to C_i\\
\unos{\Gamma\mid \os{f}(t):D}=f\circ\unos{\Gamma\mid t: D}:
C_1\times\ldots\times C_n\to D\\
\unos{\Gamma\mid \uni{T}: T}=\tot{C_1\times\ldots\times C_n}:
C_1\times\ldots\times C_n\to T\\
\unos{\Gamma\mid \pai{t_1,t_2}{\mathbf{p}}: B}=
\ple{\unos{\Gamma\mid t_1},\unos{\Gamma\mid t_2}}:
C_1\times\ldots\times C_n\to B.\\
\end{array}\]\
Next for formulas:
\[\begin{array}{c@{\quad}c@{\quad}c}
\unos{\Gamma\mid \os{\alpha}}=\alpha&
\unos{\Gamma\mid \phi\Land\psi}=\unos{\Gamma\mid \phi}\Land_{B}
\unos{\Gamma\mid \psi}&
\unos{\Gamma\mid \phi(t)}=P_{\unos{\Gamma\mid t}}(\unos{\Gamma\mid \phi}).\\
\end{array}\]
Finally, an entailment in context
\[x_1\type\os{B_1},\ldots,x_1\type\os{B_n}\mid
\phi(t_1,\ldots,t_m)\vdash\psi(t_1,\ldots,t_m)\]
is an axiom of the theory $\thr{T}_{P}$  
if the inequality
\[\unos{\Gamma\mid \phi}\leq_{B_1\times\ldots\times B_n}\unos{\Gamma\mid \psi}\]
holds in the fibre over $B_1\times\ldots\times B_n$.
\end{rem}

In line with the logical shorthand, we often refer to an element of
the fibre over $B$ as a \dfn{predicate} of $B$, and to an element of the
fibre over the product $B_1\times B_2$ as a \dfn{relation} from $B_1$ to
$B_2$.

We are ready to recall the construction of the elementary quotient
completion: for a primary doctrine
$P:\ct{B}\op\futo\Ct{Pos}$ consider the doctrine 
$\qd{P}:\qb{P}\op\futo\Ct{Pos}$ whose base consists of the following data:
\begin{description}
\item[Objects] which are $P$-\dfn{equivalence relations} $(b,\rho)$, \ie a
pair of an object $B$ in \ct{B} and a relation on $B$ which is
\begin{description}
\item[reflexive:] $x\type B\mid\top\vdash\rho(x,x)$
\item[symmetric:] $x,x'\type B\mid\rho(x,x')\vdash\rho(x',x)$
\item[transitive:] 
$x,x',x''\type B\mid\rho(x,x')\Land\rho(x,x')\vdash\rho(x',x'')$
\end{description}
\item[An arrow] $\ec{f}:(B,\rho)\to(C,\sigma)$
is an equivalence class of relation-preserving arrows $f:B\to C$ in \ct{B},
\ie $x,x'\type B\mid\rho(x,x')\vdash\sigma(f(x),f(x'))$, where two
representatives $f$ and $f'$ are equivalent when
$x\type B\mid\top\vdash\sigma(f(x),f'(x))$.
\item[Composition] is given by that of \ct{B} on representatives.
\end{description}

\begin{rem}
In the original posetal notation of the doctrine $P$,
the three conditions for a $P$-equivalence relation are written 
\begin{itemize}
\item $\top_{B}\le_{B} P_{\ple{\pr1,\pr1}}(\rho)$;
\item $\rho\le_{B} P_{\ple{\pr2,\pr1}}(\rho)$;
\item $P_{\ple{\pr1,\pr2}}(\rho)\Land P_{\ple{\pr2,\pr3}}(\rho)\le_{B}
P_{\ple{\pr1,\pr3}}(\rho)$;
\end{itemize}
where \pr{i}, $i=1,2,3$ are appropriate projections from a product of
one, two and three copies of $B$, respectively.
\end{rem}

The fibre of \qd{P} over the object $(B,\rho)$ is a subposet
$\des{\rho}\subseteq P(B)$ of the \dfn{descent data} with respect to $\rho$,
\ie those predicates $\alpha$ of $B$ which are
\begin{description}
\item[\dfn{stable\/}:]
$x,x'\type B\mid\alpha(x)\Land\rho(x,x')\vdash\alpha(x')$. 
\end{description}
One sees immediately that descent data with respect to $\rho$ are closed
under finite meets, and that reindexing a descent datum $\beta$ with
respect to $\sigma$ along any representative of an arrow
$\ec{f}:(B,\rho)\to(C,\sigma)$ produces the same descent datum $P_f(\beta)$
with respect to $\rho$. One easily checks that those determine a primary
doctrine $\qd{P}:\qb{P}\op\futo\Ct{Pos}$, see
\cite{MaiettiME:quofcm,MaiettiME:eleqc} for more details.

Moreover, the doctrine \qd{P} enjoys further properties. To start with,
it is elementary.

A primary doctrine $E:\ct{B}\op\futo\Ct{Pos}$ is \dfn{elementary} (usually
called just an elementary doctrine, \ed for short) if every
object $B$ has a relation $\delta_{B}$ such that for every arrow $e$ of the
form $\ple{\pr1,\pr2,\pr2}:C\times B\to C\times B\times B$ in \ct{B}, the
assignment
\[\D_{e}(\alpha)\colon=
E_{\ple{\pr1,\pr2}}(\alpha)\Land_{c\times B\times b}E_{\ple{\pr2,\pr3}}(\delta_{B})\]
for $\alpha$ in $E(c\times b)$ determines a left adjoint to
\[E_{e}:E(C\times B\times B)\to E(C\times B).\]

In  \cite[\prx2.5]{EmmeneggerJ:eledac} there is a 
useful characterisation of elementary doctrines: a primary doctrine
$P:\ct{B}\op\futo\Ct{Pos}$ is elementary if and only if there is a family
$(\seq{B})_{B\in\Ob{\ct{B}}}$, indexed by the objects of the base
\ct{B}, which consists of relations $\seq{B}\in P(B\times B)$ for
every $B$ in \ct{B}, and the family satisfies the following conditions,
where we write the binary relation infix:
\begin{enumerate}[(1)]
\item $x\type B\mid\top\vdash\Seq{B}{x}{x}$
\item $x,x'\type B\mid\alpha(x)\Land\Seq{B}{x}{x'}\vdash\alpha(x')$
\item $x_1,x_1'\type B_1,x_2,x_2'\type B_2\mid
\Seq{B_1}{x_1}{x_1'}\Land\Seq{B_2}{x_2}{x_2'}\vdash
\Seq{B_1\times B_2}{\ple{x_1,x_2}}{\ple{x_1',x_2'}}$.
\end{enumerate}
One proves also that the families $(\delta_{B})$ and $(\seq{B})$
coincide---we shall use the second denotation only in formulas of the
logical shorthand introduced in \rmx\ref{logshand}.

It is easy to see that the doctrine
$\Wsb{\ct{C}}:\ct{C}\op\futo\Ct{Pos}$ of variations is elementary.
In case the category \ct{C} has finite limits,
the doctrine $\Sb{\ct{C}}:\ct{C}\op\futo\ct{Pos}$ of all subobjects
of \ct{C} is elementary.

Another property of the doctrine \qd{P} is that every \qd{P}-equivalence
relation has a quotient, see
\cite{MaiettiME:quofcm,MaiettiME:eleqc,MaiettiME:relqcv}.

Recall that, in an \ed $E:\ct{B}\op\futo\Ct{Pos}$,
a \dfn{quotient} of a $E$-equivalence relation $\rho$ on the object $B$ is
an arrow $q:B\to B/\rho$ in \ct{B} which is universal among the arrows
$f:B\to A$ out of $B$, such that
\begin{equation}\label{upquo}
x,x'\type B\mid\rho(x,x')\vdash\Seq{A}{f(x)}{f(x')}.
\end{equation}
So $q:B\to B/\rho$ coequalises $\rho$-related arguments---\ie satisfies
\eqref{upquo}---, and for every $f:B\to A$ which satisfies \eqref{upquo}
there is a unique arrow filling in the commutative diagram
\[\xymatrix{B\ar[r]^{q}\ar[rd]_-{f}&B/\rho\ar@{-->}[d]\\
&A}\]

\begin{rem}\label{sakoc}
For sake of completeness, we review how the quotient of the
\qd{P}-equivalence relation $\tau$ on the object $(B,\rho)$ is determined:

As a relation on an object in \qb{P}, the relation $\tau$ is a descent datum
on $\rho$, so
\[x_1,x_1',x_2,x_2'\type B\mid 
\tau(x_1,x_2)\Land\rho(x_1,x_1')\Land\rho(x_2,x_2')\vdash\tau(x_1',x_2').\]
In particular, 
$x_1,x_2\type B\mid 
\rho(x_1,x_2)\vdash\tau(x_1,x_2)$ since $\tau$ is reflexive. So \id{B}
represents an arrow $\ec{\id{B}}:(B,\rho)\to(B,\tau)$ which is indeed the
requested quotient, see \cite{MaiettiME:eleqc} for more details.
\end{rem}

The doctrine $\qd{P}:\qb{P}\op\futo\Ct{Pos}$ also has comprehensive
diagonals, see \cite{MaiettiME:quofcm}. Recall that an \ed
$E:\ct{B}\op\futo\Ct{Pos}$ has \dfn{comprehensive diagonals} when every two
parallel arrows
\[\xymatrix{B\ar@<.5ex>[r]^-{f}\ar@<-.5ex>[r]_-{g}&B'}\]
such that $x\type B\mid\top\vdash\Seq{B'}{f(x)}{g(x)}$, are already equal in
\ct{B}. Clearly, \qd{P} has comprehensive diagonals simply because the
equivalence relation imposed on arrows is the one that forces the property,
see \cite{MaiettiME:eleqc}.

An \dfn{elementary quotient completion} (\eqc for short) is an \ed of the
form $\qd{E}:\qb{E}\op\futo\Ct{Pos}$ for an \ed
$E:\ct{B}\op\futo\Ct{Pos}$. Since $E$ is elementary, for each
object $B$ in \ct{B} there is an $E$-equivalence relation on it. This
induces a full embedding 
\[\xymatrix@C=5em@R=1ex{
  \ct{B}\op\ar[rd]^(.4){E}_(.4){}="P"\ar[dd]_{\Yo\op}&\\
&\Ct{Pos}\\
\qb{E}\op\ar[ru]_(.4){\qd{E}}^(.4){}="R"&\ntar"P";"R"_{\yo}}\]
of the doctrine $E$ into \qd{E}
which is completely determined by a universal property, see
\cite{MaiettiME:eleqc}. On the base 
categories it acts by sending an object $B$ in \ct{B} to the $P$-equivalence
relation $\Yo(B)\colon=(B,\delta_B)$ and acts identically on arrows like
$\yo$ on the predicates since $\des{\delta_B}=P(B)$.

The full embedding of an \ed $E$ into its \eqc \qd{E} is the unit of the
2-adjunction between \ed{}s $E$ and \ed{}s $Q$ with comprehensive diagonals
and descent quotients of $Q$-equivalence relations.

It is interesting to note how all the logical operations migrate from the
original primary doctrine to its \eqc: we begin with the case of existential
doctrines.

A primary doctrine $D:\ct{B}\op\futo\Ct{Pos}$ is \dfn{existential} if, for
every product diagram 
$\xymatrix@1@C=4em{B_1&B\ar[l]|-{\pr1}\ar[r]|-{\pr2}&B_2}$ in \ct{B}, the
functor $D_{\pr{1}}:D(B_1)\to D(B)$ has a left adjoint
$\D_{\pr{1}}$---hence also $D_{\pr2}$ has a left adjoint.
Moreover, the left adjoints satisfy the following conditions:
\begin{description}
\item[Beck--Chevalley Condition:] for any pullback diagram
\[\xymatrix{A\ar[r]^{\pr{}'}\ar[d]_{f'}&A_1\ar[d]^f\\
B\ar[r]^{\pr{}}&B_1}\]
with \pr{} a projection (hence also $\pr{}'$ a projection), for any
$\beta$ in $D(B)$, the natural inequality
$\D_{\pr{}'}D_{f'}(\beta)\leq D_{f}\D_{\pr{}}(\beta)$ in $D(A_1)$ is 
an identity;
\item[Frobenius Reciprocity:] for $\pr{}:B\to B_1$ a projection,
$\alpha$ in $D(B_1)$, $\beta$ in $D(B)$, the natural inequality
$\D_{\pr{}}(D_{\pr{}}(\alpha)\Land\beta)\leq\alpha\Land\D_{\pr{}}(\beta)$ in
$D(B_1)$ is an identity.
\end{description}

A doctrine $\Wsb{\ct{C}}:\ct{C}\op\futo\Ct{Pos}$ of variations is always
existential because post-composition with \pr1 determines a left adjoint
\[\ec{\postc{\pr1}{}}:\Wsb{\ct{C}}(C_1\times C_2)\to\Wsb{\ct{C}}(C_1)\]
to reindexing along \pr1 by taking the equivalence class of the composition
$\postc{\pr1}{f}\colon=\pr1\circ f$ of any representative in
\ec{f:A\to C_1\times C_2}. 

In case the collection $\mathcal{M}$ is part of a 
stable factorisation system $(\mathcal{E},\mathcal{M})$, then the
doctrine $\Sbm{\mathcal{M}}{\ct{C}}:\ct{C}\op\futo\ct{Pos}$ is
existential. So, in particular, the doctrine
$\Sb{\ct{C}}:\ct{C}\op\futo\ct{Pos}$
of all subobjects is existential if and only if \ct{C} is regular, see
\cite[\thx4.4.4]{JacobsB:catltt}.

\begin{rem}\label{eqc-implo}
It is easy to check that, if a primary doctrine $P:\ct{B}\op\futo\Ct{Pos}$
is existential, 
then the doctrine $\qd{P}:\qb{P}\op\futo\Ct{Pos}$ is existential. Indeed,
if a predicate $\alpha$ of $B$ is stable with respect to the $P$-equivalence
relation $\rho$, then one proves easily that
\[x_2,x_2'\type B_2\mid
\Exists{x_1\type B_1}\alpha(x_1,x_2)\Land\rho(x_2,x_2')\vdash
\Exists{x_1\type B_1}\alpha(x_1,x_2')\]
\end{rem}

The situation for the other connectives and quantifiers is similar. As the
structure of a primary doctrine offers the basic tools to deal with each
situation, one says that a primary doctrine $P:\ct{B}\op\futo\Ct{Pos}$ is
\dfn{disjunctive} if each fibre is a (bounded) distributive lattice and each
reindexing is a lattice homomorphism; and it is \dfn{Boolean} if every fibre
is a Boolean lattice. It is \dfn{implicational} if, for
every predicate $\alpha$, the functor $\alpha\Land_B\blank:P(B)\to P(B)$ has
a right adjoint. It is \dfn{universal} if, for every product diagram
$\xymatrix@1@C=4em{B_1&B\ar[l]|-{\pr1}\ar[r]|-{\pr2}&B_2}$ in \ct{B}, the
functor $D_{\pr1}:D(B_1)\to D(B_1\times B_2)$ has a right adjoint
$\B_{\pr{1}}$ which satisfies the Beck--Chevalley Condition.

\begin{rem}\label{eqc-impl}
Similar to \rmx\ref{eqc-implo}, if a primary doctrine
$P:\ct{B}\op\futo\Ct{Pos}$ is any among disjunctive, implicational, or
universal, then so is the doctrine $\qd{P}:\qb{P}\op\futo\Ct{Pos}$.
\end{rem}

Summing all up, a primary doctrine is \dfn{first-order} (\fo for short) if
it is elementary, existential, disjunctive, implicational, and
universal. They appear already in \cite{PittsA:tritir} with the name ``first
order hyperdoctrine''.

In the logical shortand of \rmx\ref{logshand}, the rules corresponding to
each logical construct are standard and reflect the algebraic conditions of
each definition.

Finally, if a primary doctrine $P:\ct{B}\op\futo\Ct{Pos}$ is disjunctive,
then the doctrine $\qd{P}:\qb{P}\op\futo\Ct{Pos}$ is disjunctive; it is
implicational if $P$ is implicational; and it is universal if $P$ is
universal.

In case $P:\ct{B}\op\futo\Ct{Pos}$ is elementary, then each implication can
be reversed, \eg if \qd{P} is existential, then $P$ is existential. The
simple proof goes via the full embedding of $P$ into \qd{P}.

Before embarking in the main subject of the paper, it is appropriate to
review some points about the notion of doctrines admitting comprehension, in
the sense of \cite{LawvereF:equhcs}. Recall that a primary doctrine
$P:\ct{B}\op\futo\Ct{Pos}$ \dfn{admits comprehension} if, for every
predicate $\alpha$ on $B$, there is an arrow $\cmp{\alpha}:A\to B$ in \ct{C}
universal among those along which $\alpha$ is reindexed to $\top$, \ie
$P_{\cmp{\alpha}}(\alpha)=\top_A$, and every arrow $f:C\to B$ such that
$P_f(\alpha)=\top_C$ has a unique factorisation through
\cmp{\alpha}. Uniqueness of the factorisation ensures that the arrow
\cmp{\alpha} is monic in \ct{B}. We call such an arrow in \ct{B} a
\dfn{comprehending} arrow.

Moreover, it is easy to see that the commutative diagram involving the
arrows \cmp{\alpha} and \cmp{P_f(\alpha)}
\[\xymatrix{A'\ar[d]_-{\cmp{P_f(\alpha)}}\ar[r]&A\ar[d]^-{\cmp{\alpha}}\\
B'\ar[r]^-{f}&B}\]
is a pullback. So the comprehending arrows form a class $\mathcal{H}$ closed
under identities and pullbacks, and one sees easily that there is a natural
embedding $c:P\ntto\Sbm{\mathcal{H}}{\ct{C}}$ into the
primary doctrine $\Sbm{\mathcal{H}}{\ct{C}}:\ct{C}\op\futo\ct{Pos}$.

Comprehension is \dfn{full} when that embedding is full in each fibre, \ie
whenever for any pair of predicates $\alpha$ and $\beta$ on $B$ there
is a commutative diagram 
\[\xymatrix@C=2em@R=2ex{C\ar[rr]\ar[rd]_(.4){\cmp{\alpha}}&&
D\ar[ld]^(.4){\cmp{\beta}}\\
&B}\]
then $\alpha\leq_B\beta$ in $P(B)$.

There is also the case when the previous universal condition is weakened to
a simple existence, possibly with no uniqueness, 
inspired by \cite{LawvereF:adjiab,MenniM:chalec} and extensively
used in \cite{MaiettiME:quofcm,MaiettiME:triech,MaiettiME:eleqcc}.
In details, one says that the primary doctrine $P:\ct{B}\op\futo\Ct{Pos}$
\dfn{admits weak comprehension} if, for every predicate 
$\alpha$ on $B$ there is an arrow $\cmp{\alpha}:A\to B$ in \ct{C} such that
$P_{\cmp{\alpha}}(\alpha)=\top_A$, and every arrow $f:C\to B$ such that
$P_f(\alpha)=\top_C$ has a factorisation through \cmp{\alpha}. We call such
an arrow in \ct{B} a \dfn{weak comprehending} arrow.

The absence of uniqueness prevents imitating the previous case. This time,
under the assumption that the base category \ct{B} has weak 
pullbacks, weak comprehending arrows determine a natural embedding
$w:P\ntto\Wsb{\ct{B}}$ into the doctrine
$\Wsb{\ct{B}}:\ct{B}\op\futo\Ct{Pos}$ of
variations of \ct{B}. Like before, $P$ is said to 
\dfn{admit full weak comprehension} if the embedding is full in each fibre.

For elementary doctrines admitting comprehension, even just weakly, it
is a very relevant case that each relation $\delta_B$ has the diagonal
$\ple{\id{B},\id{B}}:B\to B\times B$ as a comprehending arrow.
This guarantees that any two morphisms $f,g: A\to B$ are equal in the
base $\ct{B}$ if and only if $x\colon A\mid\top\vdash f(x)=g(x)$.

An elementary doctrine which admits full weak comprehension and has
comprehensive diagonals is called \dfn{variational}. Similarly, an
elementary doctrine which admits full comprehension and has comprehensive
diagonals is called \dfn{m-variational}.

\begin{exm}
Just to give an idea of what fails for full comprehension, consider the
example of a category \ct{B} with finite limits and a collection 
${\mathcal M}$ of monos which is closed under identities, composition and
pullbacks---for instance, the collection of subspace inclusions in the
category \ct{Top} of topological spaces and continuous functions.

Given two monos $x:X\moto B$ and $x':X'\moto B$, say
that $x\leq_{\mathcal M}x'$ if there is a commutative diagram
\[\xymatrix@C=2em@R=2ex{X\moar[dr]_-{x}\ar[rr]^-{m}&&X'\moar[dl]^(.4){x'}\\
&B&}\]
with $m$ in ${\mathcal M}$. The closure properties of ${\mathcal M}$ make
sure precisely that the relation $\leq_{\mathcal M}$ is a preorder on monos
with a fixed codomain. Taking the poset reflection of such preorders for
each $B$ gives a primary doctrine 
$\Sbm{\mathcal M}{\ct{C}}:\ct{B}\op\futo\Ct{Pos}$ admitting comprehension,
which is full when ${\mathcal M}$ is the largest collection of arrows of
\ct{B}.
\end{exm}

We conclude this section by recalling the following important theorem:
\begin{thm}\label{comprequo}
Let $P:\ct{B}\op\futo\Ct{Pos}$ be an elementary doctrine with comprehensive
diagonals. Then the following are equivalent:
\begin{enumerate}[{\rm(i)}]
\item \qd{P} is m-variational.
\item $P$ is variational.
\end{enumerate}
\end{thm}
\begin{proof}
\entail{(i)}{(ii)} Immediate since $P$ coincides with \qd{P} on
objects of the form $(B,\delta_B)$ and $P$ has comprehensive diagonals.

\entail{(ii)}{(i)} See \cite[\lmx5.3]{MaiettiME:quofcm}.
\end{proof}

\section{A characterisation for elementary quotient completions}

In this section we concentrate on a generalisation to \eqc{}s
$\qd{P}:\qb{P}\op\futo\Ct{Pos}$ of the renowned characterisation of exact
completions of \cite{CarboniA:freecl}, see also \cite{CarboniA:regec}.
We first concentrate on the properties of a doctrine 
$\qd{P}:\qb{P}\op\futo\Ct{Pos}$ which lead to the characterisation of such
doctrines.

We already noticed that the doctrine \qd{P} has comprehensive diagonals and
quotients. The \qd{P}-quotients are also of effective descent and every
object of \qb{P} appears as the \qd{P}-quotient of a \qd{P}-projective, as
we check in the following \prx\ref{intch}.

So, before stating the proposition, for sake of completeness recall from
\cite{MaiettiME:eleqc} that for a doctrine $Q:\ct{C}\op\futo\Ct{Pos}$ with
$Q$-quotients, a quotient $q:C\to C/\rho$ is
\dfn{stable under pullback} if in every pullback diagram 
\[\xymatrix{C\ar[d]_-{f'}\ar[r]^-{q'}&C'\ar[d]^-{f}\\
B\ar[r]^{q}&B/\rho}\]
the arrow $q'$ is the quotient of the $E$-equivalence relation
$P_{f'\times f'}(\rho)$ on $C$. The quotient $q$ is 
\dfn{effective} if
$\rho=P_{q\times q}(\delta_{B/\rho})$, \ie
\[x,x'\type B\mid\Seq{B/\rho}{q(x)}{q(x')}\vdash\rho(x,x').\]
It is \dfn{descent} if the factor through \des{\rho}
\begin{equation}\label{factdesc}
\vcenter{\xymatrix@R=1ex@C=1em{
Q(C/\rho)\ar[rr]^-{Q(q)}\ar[rrd]&&Q(C)\\
\kern6em&
&\des{\rho}\ar@{{}{}}|>{}|-*=0[@]{\subseteq}[u]}}
\end{equation}
is full (and faithful). And $q$ is
\dfn{of effective descent} if the factor through \des{\rho}
in \eqref{factdesc} is iso.

\begin{rem}
When considered in an \ed of the form $\Sb{\ct{C}}:\ct{C}\op\futo\Ct{Pos}$
on a category \ct{C} with finite limits and coequalisers of equivalence
relations, the condition of stability under pullback is plain.

The condition of effectivity reflects into an equivalence relation being
effective (\ie the kernel pair of its coequaliser). So such a coequaliser is
the coequaliser of its kernel pair.

The third condition of effective descent characterises the subobjects over
the codomain of a quotient by particular subobjects of the domain.

So, if the first two conditions hold in \ct{C} for all coequalisers of
equivalence relations, then \ct{C} is exact. In that case, the third
condition follows from the first two, see
\cite[Appendix~A]{RosoliniG:loccce}.
\end{rem}

To get a grasp with the notion of $Q$-quotients of effective descent,
in case of an \eed (we use the same abbreviation ``\eed'' for elementary
existential doctrine as \cite{LawvereF:equhcs}), it is interesting to see
that that condition corresponds to $D$-surjectivity. Recall that, in an \eed
$D$, an arrow $f:X\to Y$ is \dfn{$D$-surjective} if it satisfies
$\top_{Y}=\exists_f\top_X$, \ie
\[y\type Y\mid\top\vdash\Exists{x\type X}(y\seq{Y}f(x)).\]

\begin{rem}
In case $D$ has comprehensive diagonals, it is easy to check that, a
$D$-surjective arrow $f:X\to Y$ is epic.

Indeed, suppose $fe=ge:X\to Z$. Then
\[y\type Y,x\type X\mid y\seq{Y}e(x)\vdash f(y)\seq{Z}g(y).\]
Then
$y\type Y\mid\top\vdash f(y)\seq{Z}g(y)$ by $D$-surjectivity. So $f=g$
because diagonals are comprehensive in $D$.
\end{rem}

\begin{prp}\label{qusu}
Let $D:\ct{B}\op\futo\Ct{Pos}$ be an \eed, let $q:A\onto A/\rho$ be an
effective $D$-quotient of the $D$-equivalence relation $\rho$. Then the
following are equivalent:
\begin{enumerate}[{\rm(i)}]
\item $q:A\onto A/\rho$ is of effective descent.
\item $q:A\onto A/\rho$ is $D$-surjective.
\end{enumerate}
\end{prp}  
  
\begin{proof}
Note first of all that
\[x,x'\type A\mid \rho(x,x')\dashv\vdash q(x)\seq{A/\rho}q(x')\qquad \qquad (1)\] 
since the quotient $q:A\rightarrow A/\rho$ is effective by hypothesis.
Hence a predicate $\alpha$ on $A$ is descent data for $\rho$ if and only if
$\alpha=P_q \exists_q\alpha$.

\entail{(i)}{(ii)}
Since $Q$ is of effective descent, $P_q:P(A/\rho)\to P(A)$ maps bijectively
into \des{\rho} with inverse $\exists_q: \des{\rho}\to  P(A/\rho)$. 
Thus $\exists_q\top_A= \exists_q P_q\top_{A/\rho}=\top_{A/\rho}$.

\entail{(ii)}{(i)}
As $q$ is $D$-surjective, by Frobenius reciprocity
$\beta=\exists_q\top_A\wedge \beta=\exists_q P_q\beta$ for any predicate
$\beta$ on $A/\rho$.
\end{proof}

The following instrumental lemma allows to determine when a doctrine has
effective descent quotients.

\begin{lem}\label{lemmavardescent}
Let $D:\ct{B}\op\futo\Ct{Pos}$ be a m-variational \eed . If $D$ has
stable effective $D$-quotients, they are of effective descent.
\end{lem}
\begin{proof}
It follows from \ref{qusu} since $q$ is $D$-surjective. To prove this,
recall that $q$ factors through the comprehension $\eta: X\to A/\rho$ of
$\exists_q\top_{A/\rho}$ $q =\eta\cdot g$ for some $g: A\to X$ by the
property of comprehension. 
Since $\eta$ is $D$-injective it follows from the property $(1)$ above
that 
\[x,x'\type A\mid \rho(x,x')\dashv\vdash g(x)\seq{A/\rho}g(x')\]
and by the universal property of $D$-quotients there exists $t: A/\rho\to X$ such that $g=t\cdot q$.
From $q=\eta\cdot g=\eta\cdot (t\cdot q)$ since $q$ is epi we conclude $\eta\cdot t=\id{A}$ and that
 $\eta$ is an isomorphism. Thus  $q$ is $D$-surjective.
\end{proof}
An object $B$ of \ct{C} is $Q$-\dfn{projective} if, for every arrow
$f:B\to C/\rho$ into the codomain of a quotient $q:C\to C/\rho$ of a
$Q$-equivalence relation, there is a commutative diagram
\[\xymatrix@R=2em@C=3em{&B\ar[d]^-{f}\ar[dl]_-{g}\\
C\ar[r]_-{q}&C/\rho}.\]

\begin{prp}\label{intch}
In the \eqc $\qd{E}:\qb{E}\op\futo\Ct{Pos}$ of an \ed
$E:\ct{B}\op\futo\Ct{Pos}$, the following hold: 
\begin{enumerate}[{\rm(i)}]
\item every \qd{E}-quotient is of effective descent;
\item an object of the form $(B,\delta_{B})$ is \qd{E}-projective.
\end{enumerate}
\end{prp}

\begin{proof}
(i) In \qb{E} a \qd{E}-quotient is $\ec{\id{B}}:(B,\rho)\to(B,\tau)$ as seen
in \rmx\ref{sakoc}. So
$\qd{E}(B,\tau)=\des{\tau}\subseteq\des{\rho}=\qd{E}(B,\rho)$.

(ii) It follows immediately from the description of \qd{E}-quotients in
\rmx\ref{sakoc} by definition of arrow in \qb{E}.
\end{proof}

\begin{thm}\label{PPO}
Let $Q:\ct{C}\op\futo\Ct{Pos}$ be an \ed. Then the following
are equivalent:
\begin{enumerate}[{\rm(i)}]
\item $Q:\ct{C}\op\futo\Ct{Pos}$ is an \eqc.
\item $Q:\ct{C}\op\futo\Ct{Pos}$ satisfies the following properties:
\begin{enumerate}[(a)]
\item $Q$ has comprehensive diagonals;
\item $Q$ has quotients of effective descent;
\item for every object $C$ in \ct{C} there is a $Q$-equivalence relation
$\rho$ on a $Q$-projective $B$ such that $C\isto B/\rho$;
\item the collection of $Q$-projectives is closed under finite products.
\end{enumerate}
\end{enumerate}
\end{thm}

\begin{proof} \entail{(i)}{(ii)} Follows immediately from \prx\ref{intch}
after noticing that every \qd{P}-projective must be a retract of an object
of the form $(B,\delta_{B})$.

\entail{(ii)}{(i)}
Let \ct{B} be the full subcategory of \ct{C} on the $Q$-projective objects
and let \restrict{Q}{\ct{B}} be the restriction of $Q$ to the subcategory
\ct{B}. By (d), $\ct{B}\subseteq\ct{C}$ is closed under
finite products; so \restrict{Q}{\ct{B}} is elementary. And 
by (a), \restrict{Q}{\ct{B}} has comprehensive diagonals. So the inclusion
$J:\ct{B}\emto\ct{C}$ extends to a homomorphism of \ed{}s
$(J,\id{J}):\restrict{Q}{\ct{B}}\emto Q$ since the
\restrict{Q}{\ct{B}}-fibre over $B$ coincides with 
the $Q$-fibre over $B$. By the universal property of the \eqc, it extends to
a homomorphism of \ed{}s with quotients
$\overline{(J,\id{J})}:\qd{\restrict{Q}{\ct{B}}}\futo Q$.
We intend to show that $\overline{(J,\id{J})}$ is an equivalence.

Note first that, by (c), every arrow $f:C\to C'$ in \ct{C} appears in a
commutative diagram
\[\xymatrix{B\ar[d]_-{q}\ar[r]^-{g}&B'\ar[d]^-{q'}\\
C\ar[r]^-{f}&C'}\]
where $B$ and $B'$ are $Q$-projectives. This gives immediately that the
first component of $\overline{(J,\id{J})}$ is full and essentially
surjective on objects. Moreover, faithfulness follows from (a).

Finally, condition (b) ensures that the second component of
$\overline{(J,\id{J})}$ is a family of isomorphisms.
\end{proof}

Together \prx\ref{intch} and \thx\ref{PPO} indicate that the information
about an \ed $Q:\ct{C}\op\futo\Ct{Pos}$ which is an \eqc, is completely
determined by the restriction of \restrict{Q}{\ct{B}} along the full embedding
$\ct{B}\emto\ct{C}$ on a ``sufficiently large'' collection of
$Q$-projectives. The required extension of the collection is determined by
conditions (c) and (d) in \thx\ref{PPO}: the collection must consist of
enough cartesian $Q$-projectives, meaning that it is closed under finite
products, and each object $C$ in \ct{C} appears as a the codomain of a
$Q$-quotient of an object $B$ in \ct{B}. Related characterizations of quotient completions can be found in \cite{dagnino2026projectivecoversdoctrinesalgebras}.

From now on, in an \ed $Q:\ct{C}\op\futo\Ct{Pos}$ with comprehensive
diagonals and quotients of effective descent, we say that
$\ct{B}\subseteq\ct{C}$ is a \dfn{projective cover of $Q$} when
\begin{enumerate}[1.]
\item all objects of \ct{B} are $Q$-projectives;
\item each object in \ct{C} appears as the codomain of a $Q$-quotient of
an object in \ct{B}.
\end{enumerate}
A projective cover \ct{B} is \dfn{cartesian} if \ct{B} is closed under
finite products.

If there is a cartesian projective cover of $Q$, then there is clearly a
canonical choice of a cartesian projective cover of $Q$.

\begin{lem}\label{canchp}
Let $Q:\ct{C}\op\futo\Ct{Pos}$ be an \ed with comprehensive
diagonals and quotients of effective descent.
If $Q$ has a cartesian projective cover, then the full subcategory
$\proj{Q}\emto\ct{C}$ on all the $Q$-projectives is a cartesian projective
cover of $Q$.
\end{lem}

\begin{proof}
Immediate.
\end{proof}

\begin{rem}
Constructing the \eqc $\qd{E}:\qb{E}\op\futo\Ct{Pos}$ of an \ed
$E:\ct{B}\op\futo\Ct{Pos}$, the full functor $\Yo:\ct{B}\futo\qb{E}$ factors 
through the embedding $\proj{\qd{E}}\emto\ct{C}$. In that sense, the
category \proj{\qd{E}} is the best choice for generating the \eqc. Note
that, since in \qb{E} idempotents split, \proj{\qd{E}} coincide with its
Karoubi envelope \K{\proj{\qd{E}}}---which is the completion of \ct{B} with
absolute colimits, see \cite{PareR:onac}. The category \ct{B}, which
originated \qd{E}, gives \proj{\qd{E}} directly as \K{\ct{B}\ext}.
\end{rem}

It was shown in \cite[\thx5.9]{MaiettiME:quofcm} that, when the
\ed $E:\ct{B}\op\futo\Ct{Pos}$ admits full weak
comprehension, then the \eqc $\qd{E}:\qb{E}\op\futo\Ct{Pos}$ admits full
comprehension. Applying \thx\ref{PPO} to variational \eqc{}s gives
a kind of converse to \cite[\thx5.9]{MaiettiME:quofcm}.

\begin{cor}\label{conv}
If $Q:\ct{C}\op\futo\Ct{Pos}$ is a variational \eqc, then $Q$ is equivalent
to the \eqc $\qd{E}:\qb{E}\op\futo\Ct{Pos}$ obtained from a variational \ed
$E:\ct{B}\op\futo\Ct{Pos}$. Hence, $Q$ is m-variational.

Moreover, for each object $C$ in \ct{C} there is a coequaliser diagram
\[\xymatrix@C=2em{J(A)
\ar@<.5ex>[rr]^-{J(r_1)}\ar@<-.5ex>[rr]_-{J(r_2)}&&
J(B)\onar[r]&C}\]
where $J:\ct{E}\futo\qb{E}\equiv\ct{C}$ is the full embedding into the \eqc,
and $\xymatrix@1{A\ar@<.5ex>[r]^-{r_1}\ar@<-.5ex>[r]_-{r_2}&B}$ is a
pseudo-equivalence relation in \ct{B}.
\end{cor}

\begin{proof}
Let $E:\proj{Q}\futo\Ct{Pos}$ be the restriction of $Q$ to the full
subcategory of $Q$-projectives. Clearly $E$ admits full weak comprehension.

For the final part, let $(B,\rho)$ an object of \qd{E} which correspond to
$C$ via the categorical equivalence $\qb{E}\equiv\ct{C}$. Let
$\cmp{\rho}:A\to B\times B$ be a weak comprehending arrow of $\rho$, and
take $r_i\colon=\pr{i}\cmp{\rho}:A\to B$. The conclusion is
a straightforward calculation.
\end{proof}

Note that an \eqc has quotients which are automatically effective and stable
under pullback. Also, recognising a logical operator on an \eqc is
equivalent to determining the same operator on the generating \ed by
restriction along the embedding $(\Yo,\yo):P\to\qd{P}$.

\crx\ref{conv} bears a strong similarity with the characterisations results
about regular completions and exact completions of categories with weak
finite limits in \cite{CarboniA:regec}. In fact, the exact completion of category $\ct{C}$ with weak finite limits is an
instance of \crx\ref{conv} when  finite products
are actually strong. And the regular completion of $\ct{C}$  is contained in any \eqc of
a variational \ed. We shall review those briefly.

First, also to get a more definite idea of what kind of quotients are
determined by an \eqc, we recall some results about m-variational doctrines
from \cite{MaiettiME:quofcm}.

\begin{lem}\label{bascp}
Let $P:\ct{B}\op\futo\Ct{Pos}$ be an \ed.
Suppose the diagram
\begin{equation}\label{ppb}
\vcenter{\xymatrix{X'\ar[d]_-{g'}\ar[r]^-{f'}&X\ar[d]^-{g}\\
  B'\ar[r]^-{f}&B}}
\end{equation}
commutes in \ct{B}. If $\alpha$ is a predicate on
the object $B$, then the following hold:
\begin{enumerate}[{\rm(i)}]
\item If $g:X\to B$ is a weak comprehending arrow for $\alpha$, and
the diagram \eqref{ppb} is a weak pullback, then
$g':X'\to B'$ is a weak comprehending arrow for $P_{f}(\alpha)$.
\item If $g:X\to B$ is a comprehending arrow for $\alpha$, then
the diagram \eqref{ppb} is a $($weak$)$ pullback if and only if
$g':X'\to B'$ is a $($weak$)$ comprehending arrow for $P_{f}(\alpha)$.
\end{enumerate}
\end{lem}

\begin{proof}
(i) Note first that
$P_{g'}(P_{f}(\alpha))=P_{f'}(P_{g}(\alpha))=\top$ since diagram
\eqref{ppb} is commutative and $g$ is a comprehending arrow for $\alpha$.. 
Consider an arrow $h:A\to B'$ such that 
$x\type A\mid\top\vdash\alpha(f(h(x)))$. 
So the commutative diagram in \eqref{ppb} can be completed to a diagram
\begin{equation}\label{justone}
\xymatrix@C=2ex@R=2ex{A\ar@/^9pt/[rrrd]^-{k}\ar@/_9pt/[dddr]_-{h}&&&\\
&X'\ar[dd]_-{g'}\ar[rr]^-{f'}&&X\ar[dd]^-{g}\\
\strut\\
&B'\ar[rr]^-{f}&\qquad&B}
\end{equation}
By weak universality of \eqref{ppb}, one obtains a factorisation of $h$ via
$g'$.

(ii) We shall deal only with the case of the weak properties, as the
other is similar.

Consider an object and a pair of arrows in a commutative diagram like
\eqref{justone}. Thus $P_{h}(P_{f}(\alpha))=P_{k}(\top_B)=\top_A$. Since
$g'$ is a comprehending arrow for $P_{f}(\alpha)$, there is an arrow
$l:A\to X'$ such that $h=gl$. The commutativity of the other triangle
follows since \cmp{\alpha} is monic.
\end{proof}

\begin{lem}\label{basce}
Let $P:\ct{B}\op\futo\Ct{Pos}$ be an \ed with comprehensive diagonals. Then
the following hold:
\begin{enumerate}[{\rm(i)}]
\item The diagram
\[\xymatrix{A\ar[r]^-{e}&B\ar@<.7ex>[r]^-{f}\ar@<-.7ex>[r]_-{g}&B'}\]
is a weak equaliser if and only if $e:A\to B$  is a weak comprehending arrow
of $P_{\ple{f,g}}(\delta_B)$.
\item The diagram
\[\xymatrix{A\ar[r]^-{g'}\ar[d]_-{f'}&C\ar[d]^-{f}\\
B\ar[r]^-{g}&D}\]
is a weak pullback if and only if $\ple{f',g'}:A\to B\times C$  is a weak
comprehending arrow of $P_{f\times g}(\delta_D)$.
\end{enumerate}
If one property is strong, then that corresponding to it is also strong.
\end{lem}

\begin{proof}
To see (i), take an arrow $h:X\to B$. Since $P$ is with comprehensive
diagonals, $fh=gh$ if and only if  
\[x\type X\mid \top\vdash f(h(x))\seq{B}g(h(x)).\]
So every $h$ such that $fh=gh$, factors through $e$ if and only if $e$ is a
comprehending arrow of $P_{\ple{f,g}}(\delta_D)$. (ii) is a corollary
of (i). The final part is immediate, since each arrow involved is monic.
\end{proof}

\begin{lem}\label{regularepi}
Let $P:\ct{B}\op\futo\Ct{Pos}$ be a variational \ed. Let
$\rho$ be a $P$-equivalence relation on the object $B$. If $q:B\to B/\rho$
is an effective $P$-quotient of the $P$-equivalence relation $\rho$ on $B$,
and the diagram
\[\xymatrix{A\ar[r]^-{r_2}\ar[d]_-{r_1}&B\ar[d]^-{q}\\
B\ar[r]^-{q}&B/\rho}\]
is a weak pullback,
then $q$ is the coequaliser of $r_1$ and $r_2$.

So, if \ct{B} has pullbacks, and $P$ has stable-under-pullback effective
$P$-quotients of $P$-equivalence relations, then \ct{B} is regular.
\end{lem}

\begin{proof}
By definition of effective $P$-quotient
\[x,x'\type B\mid\Seq{B/\rho}{q(x)}{q(x')}\dashv\vdash\rho(x,x').\]
Therefore $\ple{r_1,r_2}:A\to B\times B$ is a weak comprehending arrow of
the predicate $\rho=P_{q\times q}(\delta_{B/\rho})$. 
Thus $f:B\to C$ coequalises $r_1$ and $r_2$ if and only if
$\rho\leq P_{f\times f}(\delta_{C})$, since $P$ admits full weak
comprehension.

For the final part, note that the kernel pair of the arrow
$f:B\to C$ is a comprehending arrow of the predicate
$P_{f\times f}(\delta_{C})$ on $B\times B$ by \lmx\ref{basce}(ii).
\end{proof}

Finally, we can report about the relationship with the characterisation
results in \cite{CarboniA:regec}.

\begin{prp}
Let \ct{B} be a category with finite products and weak pullbacks.
Then the following hold:
\begin{enumerate}[{\rm(i)}]
\item If $E:\ct{B}\op\futo\Ct{Pos}$ is a sub-\ed of \Wsb{\Ct{B}}, then the
base \qb{E} of the \eqc of $E$ contains, as a full subcategory, the regular
completion $\ct{B}\rgw$ of \ct{B} with respect to weak finite limits.
\item The \eqc of the \ed $\Wsb{\Ct{B}}:\ct{B}\op\futo\Ct{Pos}$ of
variations is the subobject doctrine
$\Sb{\ct{B}\exw}:\ct{B}\exw\op\futo\Ct{Pos}$ of the exact completion
$\ct{B}\exw$ of \ct{B} with respect to weak finite limits.
\end{enumerate}
\end{prp}

\begin{proof}
(i) By the universal property of the regular completion, there is a functor
$F:\ct{B}\rgw\futo\qb{E}$ thanks to \lmx\ref{regularepi}. By
\lmx\ref{basce}(ii), the functor $F$ on an object 
$(f_i:X\to X_{i})_{i\in I}$, $I$ a finite set, in $\ct{B}\rgw$, produces a
coequaliser diagram
$\xymatrix@1{\Yo(K)\ar@<1.3ex>[rr]|-{\Yo(f)}\ar@<-1.3ex>[rr]|-{\Yo(f')}&&
\Yo(X)\ar[r]^-{q}&F(f)}$
in \qb{E}, where
\[\xymatrix{K\ar[rr]^-{f'}\ar[d]_-{f}&&
X\ar[d]^-{\ple{f_i:i\in I}}\\
X\ar[rr]^-{\ple{f_i:i\in I}}&&\prod_{i\in I}X_i}\]
is the kernel pair of $f$ in \ct{B}. The definition of an arrow in
$\ct{B}\rgw$ ensures that $F$ is faithful. The \qd{E}-projectivity of the
objects in the image of $\Yo$ provides fullness.

\noindent(ii) By (i) the \qb{E}-projectives are exactly the regular
projectives. The conclusion follows immediately.
\end{proof}

The \eqc{}s enjoy properties which show they are much closer
to doctrines of subobjects of exact categories than of regular
categories. For instance, there is a descent property for certain
quotients.

\begin{prp}\label{descentdoc}
Let $Q:\ct{C}\op\futo\Ct{Pos}$ be an m-variational \ed with
stable-under-pullback effective quotients of effective descent. If
$q:B\to B/\rho$ is a $Q$-quotient of the $Q$-equivalence relation $\rho$,
then the functor $q^\ast:\comc{\ct{C}}{(B/\rho)}\to\comc{\ct{C}}{B}$ is
monadic.
\end{prp}

\begin{proof}
The proof is an application of the characterisation in
\cite[\thx2.5]{JanelidzeG:facodi}. 
Recall that the monad $M$ on \comc{\ct{C}}{B} is determined
by the adjunction 
\[\xymatrix@C=4em{\comc{\ct{C}}{(B/\rho)}\ar@<1.1ex>[r]^-{q^\ast}&
\comc{\ct{C}}{B}
\ar@<1.1ex>[l]^-{\Sigma_q}\ar@{{}{}}[l]|-{\top}.}\]
Consider the equivalence relation determined by (the two components
of) the comprehending arrow of $\rho$ 
\[\xymatrix{\cmp{\rho}\ar@<.5ex>[r]^-{r_1}\ar@<-.5ex>[r]_-{r_2}&B}\]
which is a kernel pair of $q:B\to B/\rho$.
Hence an $M$-algebra consists of arrows $f$ and $e$ in a commutative
diagram
\[\xymatrix@C=2ex@R=2ex{
r_1^\ast f\isar_-{e}[rd]\ar@/_/[dddr]_-{f_1'}\ar@/^/[drrrr]^-{r_1'}\\
&r_2^\ast f\ar[dd]^-{f_2'}\ar[rrr]_-{r_2'}&&\qquad&A\ar[dd]^-{f}\\
&&&&\\
&\cmp{\rho}\ar@<.5ex>[rrr]^-{r_1}\ar@<-.5ex>[rrr]_-{r_2}&&&B&}\]
where the square involving $f$, $r_i$, $r_i'$ and $f_i'$ is a pullback,
$i=1,2$.
Therefore
\[\xymatrix@C=2ex@R=2ex{
r_1^\ast f\ar[dd]_-{}\ar[rrr]^-{\ple{r_1',r_2'\circ e}}&&\qquad&
A\times A\ar[dd]^-{f\times f}\\
&&&&\\
\cmp{\rho}\ar[rrr]^-{\ple{r_1,r_2}}&&&B\times B&}\]
is a pullback, showing that \ple{r_1',r_2'\circ e} is a comprehending arrow
of the $Q$-equivalence relation $Q_{f\times f}(\rho)$. So it is 
the pullback of the induced map $f':A/Q_{f\times f}(\rho)\to B/\rho$ along $q$,
as, by hypothesis, quotients are stable under pullback.
\end{proof}

\section{Topology on a doctrine}\label{topology}

As we shall see in \S\ref{when}, \eqc's are well suited to the construction
of quasitoposes. To characterize the locally cartesian closed structure of
an \eqc in \S\ref{lcc}, we will make use of the fact that a variational
doctrine can in fact be viewed as being determined by a topology on the
doctrine of weak subobjects of its base.

We therefore define a topology on a doctrine, following the approach of
\cite{HylandJ:trit}. These notions are closely related to those introduced
by Carboni and Mantovani in \cite{CarboniA:eleccs} and by Menni in
\cite{MenniM:clooec}. We leave a precise analysis of the relationship to
future work.

In what follows, we use the notion of morphism between primary doctrines as defined on p. 5 of \cite{MaiettiME:eleqc}.

\begin{dfnt}\label{dtopo}
Let $P:\ct{B}\op\futo\Ct{Pos}$ be a primary doctrine.
An endomorphism of primary doctrines  
\[\xymatrix@C=5em@R=1ex{
\ct{B}\op\ar[rd]^(.4){P}_(.4){}="P"\ar[dd]_{\Id{\ct{B}}\op}&\\
&\Ct{Pos}\\
\ct{B}\op\ar[ru]_(.4){P}^(.4){}="R"&\ntar"P";"R"_{j}}\]
is a \dfn{topology on} $P$ if $j$ is
\begin{description}
\item[increasing:] $x\type B\mid\alpha(x)\vdash(j_B\alpha)(x)$
\item[idempotent:] $x\type B\mid(j_B(j_B\alpha))(x)\vdash(j_B\alpha)(x)$
\end{description}
for every predicate $\alpha$ on an object $B$.

A predicate $\alpha$ is $j$-\dfn{closed} if
$x\type B\mid(j_B\alpha)(x)\vdash\alpha(x)$. 
An object $B$ of $\ct{B}$ is $j$-\dfn{separated} if $\delta_A$ is
$j$-closed.
\end{dfnt}

Since the condition that $j$ is a homomorphism of primary doctrines ensures
that
\[x\type B\mid
(j(\alpha_1\Land_B\alpha_2))(x)\dashv\vdash
(j_B\alpha_1)(x)\Land(j_B\alpha_2)(x)\]
we shall drop the subscript in $j_B$ from now on, and write $j$ as an
operator.

We point out that the term ``topology'' is used in connection with that of
``Lawvere--Tierney topology'' on an elementary topos 
which constitute an obvious instance of \dfx\ref{dtopo},
on the subobject doctrine $\Sb{\ct{E}}:\ct{E}\op\futo\Ct{Pos}$ of the
elementary topos \ct{E}. In that case it is usually called a
\emph{universal closure operator}, see \eg
\cite{CarboniA:eleccs,MenniM:clooec,JohnstoneP:skeeli}. 

That situation is dual to that of interior operators, see
\cite{DagninoF:docmac}. So, in line with that, we shall call
\dfn{closure operator}  an
endomorphism $j:P\ntto P$ of doctrines which is increasing and idempotent,
but not necessarily an endomorphism of primary doctrines.

\begin{rem}\label{addrem}
It is straightforward to check that each logical construct on the primary
doctrine $P:\ct{B}\op\futo\Ct{Pos}$ induces a similar one on the primary
doctrine $P_j$. So, if $P$ is elementary, then so is $P_j$; if $P$ is
existential, so is $P_j$. And similarly for the disjunctive, implicational,
and universal cases.
\end{rem}

\begin{rem}\label{adapt}
Note that in general an adjunction in the 2-category of primary doctrines
requires two pairs of a functor and a natural transformation
\[\xymatrix@C=5em{P\ar@<.1ex>@/^/[r]^-{(L,l)}\ar@{{}{}}[r]|-{\bot}&
P'.\ar@<.1ex>@/^/[l]^-{(R,r)}}\]
In the case of topologies the adjunction has the first components equal to
the identity functor
\[\xymatrix@C=5em{P\ar@<.1ex>@/^/[r]^-{(\Id{\ct{B}},l)}\ar@{{}{}}[r]|-{\bot}&
P_j.\ar@<.1ex>@/^/[l]^-{(\Id{\ct{B}},r)}}\]
It is interesting to note that the general case gives rise to a topology by
taking the natural transformation
\[\xymatrix@C=5em{P\ntar[r]_-{rl}&P(RL)\op\ntar[r]_-{Pu\op}&P}\]
where $u$ is the unit of the adjunction between the functors $L\dashv R$.
\end{rem}

The general fact has interesting consequences in some specific cases.

\begin{prp}\label{abcde}
Let $P:\ct{C}\op\futo\Ct{Pos}$ be an \eed. If \ct{C} has
weak pullbacks, then the following are equivalent:
\begin{enumerate}[{\rm(i)}]
\item $P$ is isomorphic to $\left(\Wsb{\ct{C}}\right)_j$ for some topology
$j$ over $\Wsb{\ct{C}}$.
\item There is an adjunction of primary doctrines $\xymatrix@1@C=2em{
\Wsb{\ct{C}}\ar@<.2ex>@/^/[rr]^-{(\Id{\ct{C}},l)}
\ar@{{}{}}[rr]|-{\bot}&&\,P\ar@<.2ex>@/^/[ll]^-{(\Id{\ct{C}},r)}}$ 
and $lr=\id{P}$.
\item $P$ admits full weak comprehension.
\end{enumerate}
\end{prp}

\begin{proof} \nsc{(i)}{(ii)} is an instance of the general
case.

\noindent \entail{(iii)}{(ii)}
For an object $C$ in \ct{C}, let $r_C(\alpha)\colon=\ec{\cmp{\alpha}}$ when
$\alpha$ is a predicate in $P(C)$. Let
$l_C\ec{f}\colon=\D_{\pr1}(P_{\ple{f,\id{A}}}(\delta_A))$ for $f:A\to C$
in \ct{C}, any representative of a predicate in $\Wsb{}(C)$.
The predicate $l_C\ec{f}$ is determined by the formula
\[\Exists{y\type A}(f(y)\seq{C}z)\]
in the variable $z\type C$.

Note that if 
\[\xymatrix@C=2em@R=2ex{A\ar[rd]_(.4){f}\ar[rr]^-{g}&&A'\ar[ld]^(.4){f'}\\
&C}\]
commutes, then
\[z\type C\mid\Exists{y\type A}(f(y)\seq{C}z)\vdash
\Exists{y\type A'}(f'(y)\seq{C}z).\]
It follows immediately that the
definition does not depend on the choice of the representative and
gives a natural transformation $l:\Wsb{\ct{C}}\ntto P$.
By fullness of weak comprehension
\begin{equation}\label{fcomprfor}
z\type C\mid\Exists{y\type A}(\cmp{\alpha}(y)\seq{C}z)
\dashv\vdash\alpha(z).
\end{equation}
So $l_C(r_C(\alpha))=\alpha$.
Moreover, $f$ factors through \cmp{l_C\ec{f}} since
\[x\type A\mid\top\vdash
\Exists{y\type A}(f(y)\seq{C}f(x)).\]
This provides an adjunction of doctrines
\[\xymatrix@C=2em{
\Wsb{\ct{C}}\ar@<.2ex>@/^/[rr]^-{(\Id{\ct{C}},l)}
\ar@{{}{}}[rr]|-{\bot}&&\,P\ar@<.2ex>@/^/[ll]^-{(\Id{\ct{C}},r)}}\]
where the left adjoint is also left inverse. So $r$ preserves (finite) meets
as a right adjoint.

To see that also $l$ does, recall from
\cite[\prx2.19]{MaiettiME:triech} that a weak pullback can be computed as
\[\xymatrix@C=3em@R=2em{
W\ar[rr]^-{f'}\ar[dd]_-{g'}\ar[rd]^(.6){\cmp{P_{f\times g}(\delta_{C})}}
&\kern5em&B\ar[dd]^-{g}\\
&A\times B\ar[ru]_-{\pr2}\ar[ld]^(.4){\pr1}&\\
A\ar[rr]^-{f}&&C.}\]
So $l_C(\ec{f}\Land\ec{g})$ is determined by the formula
\[\Exists{w\type W}(f(\pr1(\cmp{P_{f\times g}(\delta_{C})}(w)))\seq{C}z)\]
in the variable $z\type C$.
By \eqref{fcomprfor}
\[\begin{array}{@{}r@{{}\dashv\vdash{}}l}
\multicolumn2{@{}l}{z\type C\mid
\Exists{w\type W}(f(\pr1(\cmp{P_{f\times g}(\delta_{C})}(w)))\seq{C}z)
\dashv\vdash}\\[1ex]
\kern10em&
\Exists{x\type A}\Exists{y\type B}(f(x)\seq{C}z\Land f(x)\seq{C}g(y))\\[1ex]
&\Exists{x\type A}(f(x)\seq{C}z)\Land
\Exists{y\type B}(g(y)\seq{C}z)\\[1ex]
\end{array}\]
which shows that $l_C(\ec{f}\Land\ec{g})=l_C\ec{f}\Land l_C\ec{g}$.

\noindent \entail{(i)}{(iii)}
Suppose $\alpha$ in $\Wsb{}(C)$ is such that $\alpha=j\alpha$. Let
$f:A\to C$ be a representative of the class $\alpha$. So $f$ is a
comprehending arrow for $\alpha$ in \Wsb{\ct{C}}.
So $P_{f}(\alpha)=\Wsb{}_{f}(\alpha)=\top_{A}$.
So $P_{g}(\alpha)=\Wsb{}_{g}(\alpha)$ for every $g:B\to C$.
And $\top_B\leq P_{g}(\alpha)$ if and only if
$g$ factors through $f$.
\end{proof}

There is an interesting modification of the previous \prx\ref{abcde} which
involves the relations $\delta$.

\begin{prp}\label{abcde+}
Let $P:\ct{C}\op\futo\Ct{Pos}$ be an \eed. If \ct{C} has
weak pullbacks, then the following are equivalent:
\begin{enumerate}[{\rm(i)}]
\item $P$ is isomorphic to $\left(\Wsb{\ct{C}}\right)_j$ for some topology
$j$ over $\Wsb{\ct{C}}$ for which all objects of \ct{C} are
separated.
\item There is an adjunction of \ed{}s $\xymatrix@1@C=2em{
\Wsb{\ct{C}}\ar@<.2ex>@/^/[rr]^-{(\Id{\ct{C}},l)}
\ar@{{}{}}[rr]|-{\bot}&&\,P\ar@<.2ex>@/^/[ll]^-{(\Id{\ct{C}},r)}}$ 
and $lr=\id{P}$.
\item $P$ is variational.
\end{enumerate}
\end{prp}

\begin{proof}
After \prx\ref{abcde}, it is enough to note that, since \Wsb{\ct{C}} has
comprehensive diagonals, the condition objects of \ct{C} are $j$-separated
is equivalent to $P$ having comprehensive diagonals. In turn that is
equivalent to $(\Id{},r)$ being a homomorphism of \ed{}s.
\end{proof}

The proof of the following propositions is analogous to that of
\prx\ref{abcde+}.

\begin{prp}\label{abcdef}
Let $P:\ct{C}\op\futo\Ct{Pos}$ be an \eed. If \ct{C} has pullbacks,
then the following are equivalent:
\begin{enumerate}[{\rm(i)}]
\item $P$ is isomorphic to $\left(\Sb{\ct{C}}\right)_j$ for some topology 
$j$ over $\Sb{\ct{C}}:\ct{C}\op\futo\Ct{Pos}$.
\item There is $\xymatrix@1{
\Sb{\ct{C}}\ar@<.2ex>@/^/[rr]^-{(\Id{\ct{C}},l)}
\ar@{{}{}}[rr]|-{\bot}&&\,P\ar@<.2ex>@/^/[ll]^-{(\Id{\ct{C}},r)}}$
and $lr=\id{P}$.
\item $P$ has full comprehensions.
\end{enumerate}
\end{prp}

\begin{prp}\label{abcdef+}
Let $P:\ct{C}\op\futo\Ct{Pos}$ be an \eed. If \ct{C} has pullbacks,
then the following are equivalent:
\begin{enumerate}[{\rm(i)}]
\item $P$ is isomorphic to $\left(\Sb{\ct{C}}\right)_j$ for some topology
$j$ over $\Sb{\ct{C}}:\ct{C}\op\futo\Ct{Pos}$ for which all objects of
\ct{C} are separated.
\item There is an adjunction of \ed{}s $\xymatrix@1@C=2em{
\Sb{\ct{C}}\ar@<.2ex>@/^/[rr]^-{(\Id{\ct{C}},l)}
\ar@{{}{}}[rr]|-{\bot}&&\,P\ar@<.2ex>@/^/[ll]^-{(\Id{\ct{C}},r)}}$ 
and $lr=\id{P}$.
\item $P$ is m-variational.
\end{enumerate}
\end{prp}

Since \qd{} is a 2-functor on the 2-category of primary doctrines, it
preserves all the 2-categorical notions which appeared in this section. So,
for instance, if
$\xymatrix@1@C=4em{P\ar@<.1ex>@/^/[r]\ar@{{}{}}[r]|-{\bot}&
R\ar@<.1ex>@/^/[l]}$,
then
$\xymatrix@1@C=4em{\qd{P}\ar@<.1ex>@/^/[r]\ar@{{}{}}[r]|-{\bot}&
\qd{R}\ar@<.1ex>@/^/[l]}$.

And a topology $j$ on a primary doctrine $P$ determines a topology $\qd{j}$
on the \eqc $\qd{P}$. The topology $\qd{j}$ is the restriction of $j$
to the poset of descent data, see also \cite{MenniM:clooec}. 

Since \prx\ref{abcde+} determines that every variational \eed $P$ on a
category $\ct{B}$ with weak pullbacks generates a topology 
on $\Wsb{\ct{B}}$, 2-functoriality together with \rmx\ref{adapt} provides
the following result because \qd{} applies into the 2-category of elementary
doctrines, see \cite{PasqualiF:cofced,EmmeneggerJ:eledac}.
 
\begin{prp}\label{separated}
Suppose $P:\ct{C}\op\futo\Ct{Pos}$ is a variational \eed on a category
$\ct{B}$ with weak pullbacks. Then \qd{P} is the category of
$k$-separated objects for a topology $k$ on
$\Sb{\ct{C}\exw}:\ct{C}\exw\op\futo\Ct{Pos}$.
\end{prp}

\begin{proof}
By \prx\ref{abcde+}, $P$ is obtained as $(\Wsb{\ct{C}})_{j}$ for some
topology $j$ on the doctrine of variations on \ct{C}. Since \qd{} takes
primary doctrines to \ed{}s, the topology \qd{j} determines \qd{P} as the 
doctrine $\qd{\Wsb{\ct{C}}}_{\qd{j}}\equiv(C\exw)_{\qd{j}}$ on the
\qd{j}-separated objects by \prx\ref{abcdef+}.
\end{proof}

\begin{exms}\label{fstex}
\prx\ref{separated} applies to several constructions in the literature,
starting with the examples of realizability toposes, such as the effective
topos \ct{Eff}, see \cite{HylandJ:efft}. In \cite{RobinsonE:colcet}, the
category \ct{Eff} is shown to be the exact completion of the category
\ct{PA} of \dfn{partitioned assemblies}, see \cite{CarboniA:catarp}. The
objects of \ct{PA} are functions $e:S\to\nat$---the name 
\dfn{partitioned assembly} is derived from the observation that $e$
determines an enumerable partition family of the set $S$. An arrow
$f:e\to e'$ is a function $f:S\to S'$ for which there exists a partial
recursive function $r:\nat\pmto\nat$ such that $e'(f(s))=r(e(s))$ for all 
$s\in S$. The category \ct{PA} has finite limits. Since \ct{Eff} is the exact
completion 
of \ct{PA}, the subobject doctrine $\Sb{\ct{Eff}}:\ct{Eff}\op\futo\Ct{Pos}$
is the \eqc of the doctrine of variations
$\Wsb{\ct{PA}}:\ct{PA}\op\futo\Ct{Pos}$. 

The usual factorisation of functions defines a topology $j$ on \Wsb{\ct{PA}}
by mapping a predicate on the partitioned assembly $e:S\to\nat$, \ie an
equivalence class \ec{f:d\to e}, to the predicate \ec{f[d]\emto e} on 
$e:S\to\nat$, determined by a(ny) representative $f$ of the equivalence
class. Every partitioned assembly is $j$-separated, and the doctrine of
$j$-closed predicates is isomorphic to the change of 
base of the  doctrine $\Sb{\ct{Set}}:\ct{Set}\op\futo\Ct{Pos}$ of subsets
along the functor $\dom:\ct{PA}\futo\ct{Set}$. 

That functor is left adjoint to the functor $\nabla:\ct{Set}\futo\ct{PA}$
which sends a set $S$ to the one-partitioned assembly
$[s\mapsto0]:S\to\nat$.
Note that a $\left(\Wsb{\ct{PA}}\right)_j$-equivalence relation 
$R\subset S\times S$ on the partitioned assembly $e$ is the kernel pair of
an arrow in \ct{PA}, \eg $q:e\to\nabla(S/R)$. So the base of the \eqc of
$\left(\Wsb{\ct{PA}}\right)_j:\ct{PA}\op\futo\ct{Set}$ is the regular
completion of the lex category \ct{PA}, which amounts to be the category \ct{Asm} of assemblies and
is equivalent to the quasitopos of $\neg\neg$-separated objects in
\ct{Eff}. Finally by \prx\ref{separated}, the \eqc is the doctrine of strong
subobjects of \ct{Asm}.

Also the category of equilogical spaces, introduced in
\cite{ScottD:newcds,BauerA:equs}, can be recovered from
\prx\ref{separated}. For reasons which will become apparent in the
presentation, it is interesting to see this, following an equivalent
presentation of the category of topological spaces which appears in
\cite{RosoliniG:equsfs}, see also \cite{BauerA:equs}.

Recall from \cite{BirkhoffG:latt} that a complete lattice is \dfn{algebraic}
if every element is the join of the compact elements below it. 
Define an \dfn{algebraic-lattice-partitioned assembly} (an alpa, for short)
as a function $e:S\to\udl{A}$ from a set $S$ into (the underlying set of) an
algebraic lattice $A$. Consider the category \ct{ALPA} whose objects are
alpas; an arrow $f:e\to e'$ is a function $f:S\to S'$ for which there exists
a continuous function $r:A\to A'$ such that $e'(f(a))=r(e(a))$ for all
$a\in\udl{A}$. It is immediate to see that \ct{ALPA} is equivalent to the
category \ct{Top} of topological spaces and continuous functions, thanks to
Scott's Extension Theorem and Embedding Theorem. We shall still write it
\ct{ALPA}.

Consider its doctrine of variations
$\Wsb{\ct{ALPA}}:\ct{ALPA}\op\futo\Ct{Pos}$, and the topology $j$ on it
which maps a predicate on the alpa $e:S\to\udl{A}$, \ie an equivalence class
$\ec{f:d\to e}$, to the predicate \ec{T'\emto S} on $e$, constructed as
follows: fix a representative $f$ for the predicate, consider the image
$ef[T]\subseteq\udl{A}$ of the composite function $e\circ f$, and take the
inverse image of it along $e$, as in the following diagram
\[\xymatrix@C=1ex@R=1ex{
&T'\emar[rrd]\ar[ddd]^(.6){e'}|!{[ld];[rrd]}\hole&\\
T\ar[ddd]_-{d}\ar[rrr]_(.6){f}\onar[rdd]\ar@{-->}[ru]&&&S\ar[ddd]^-{e}\\
\strut&&&\\
&ef[T]\emar[rrd]&\quad\\
\udl{B}\ar[rrr]_-{r}&&&\udl{A}.}\]
The $j$-separated objects are exactly those alpas $e:S\to\udl{A}$ such that
$e$ is monic. These correspond precisely with the T${}_{0}$-spaces via the
categorical equivalence $\ct{ALPA}\equiv\ct{Top}$. 

The category \ct{Equ} of equilogical spaces is equivalent to the
base of the \eqc of the doctrine
$(\Wsb{\ct{ALPA}})_{j}:\ct{ALPA}\op\futo\Ct{Pos}$ of $j$-closed
predicates. As we shall see in \xsx\ref{sndex}, that \eqc is 
(equivalent to) the doctrine of strong subobjects of \ct{Equ}. Obviously, it
is determined as the subdoctrine on the \qd{j}-separated objects of the
exact completion $\ct{ALPA}\exl\equiv\qb{\Wsb{\ct{ALPA}}}$.

It is interesting to note that \ct{Equ} is (equivalent to) the base of the
doctrine $(\Wsb{\ct{ALPA}})_{j}$ restricted to the full subcategory of
\ct{ALPA} on the $j$-separated objects, \ie the T${}_{0}$-spaces via
$\ct{ALPA}\equiv\ct{Top}$, as it follows from \rmx\ref{adapt}. It is easy to
see that, via the categorical equivalence, it is isomorphic to the change of
base of the doctrine $\Sb{\ct{Set}}:\ct{Set}\op\futo\Ct{Pos}$ of subsets
along the forgetful functor $U:\TZ\futo\ct{Set}$.

Another construction which can be analysed through \prx\ref{separated} is
that of localic toposes, \ie Grothendieck toposes obtained as sheaves on a
frame, aka locale when the relevant category is the opposite of that of
frames and homomorpshims, as is the case here. Recall that a \dfn{frame} is
a complete lattice where finite meets distributes over arbitrary joins. 

For such a locale \ct{L} considered as a category, let \Fam{\ct{L}} be the
category of set-indexed \dfn{families} $(b_i)_{i\in I}$ of elements of
\ct{L}. An arrow $f:(b_i)_{i\in  I}\to(b_j')_{j\in J}$ is a function 
$f:I\to J$ such that 
\[b_{i}\leq b_{f(i)}',\quad i\in I.\]
This is precisely the category of \dfn{fuzzy sets} in
\cite{GOGUEN1974513}.
As pointed out in \cite{PittsA:fuzsdn}, that category
is a quasitopos.

The category \Fam{\ct{L}}, as coproduct completion of a category \ct{L}
with finite limits, has itself finite limits. In turn, the exact completion
of that, $\Fam{\ct{L}}\exl$, is the presheaf topos
$[\ct{L}\op,\ct{Set}]$. In other words, the \eqc of the doctrine
$\Wsb{\Fam{\ct{L}}}:\Fam{\ct{L}}\op\futo\Ct{Pos}$ of variations on
\Fam{\ct{L}} is equivalent to the doctrine of subobjects on
$[\ct{L}\op,\ct{Set}]$.

Consider the topology $j$ on \Wsb{\Fam{\ct{L}}} that maps a variation on
$(b_j)_{j\in J}$, \ie an equivalence class 
\ec{f:(a_i)_{i\in I}\to(b_j)_{j\in J}}, to 
\ec{\id{J}:(\bigvee_{f(i)=j}a_i)_{j\in J}\to(b_j)_{j\in J}}.
In this case, the base of the \eqc of the doctrine
$(\Wsb{\Fam{\ct{L}}})_{j}:\Fam{\ct{L}}\op\futo\Ct{Pos}$ 
of $j$-closed predicates is equivalent to the category of separated
presheaves on \ct{L} with respect to the canonical topology on \ct{L}.
\end{exms}

One can see \prx\ref{separated} as a doctrinal counterpart of
\cite[\lmx4.2]{MenniM:clooec}. In fact, most of what Menni develops for
topologies on regular and exact completions can be extended to the primary
doctrines. We leave for a subsequent paper the study of this.

\section{Distributive disjunctive doctrines}\label{fdcsec}

In this section we establish the necessary and sufficient conditions on the
$Q$-projectives under which an \eqc $Q:\ct{C}\op\futo\Ct{Pos}$ is a
distributive category, in the sense 
of \cite{CarboniA:intedc}. Recall that in such a category \ct{C}, any
initial object $0$ is strict, and each injection into a binary coproduct is
monic. In general, these properties of \ct{B} do not generate similar
properties on a doctrine $P$ generating the \eqc $Q$.

For instance, one would expect to find traces of the ``emptyness'' for an
initial object $0$ in the doctrine, in particular that $P(0)$ be a
singleton. But there is no reason for this: just consider a doctrine
$P:\ct{1}\op\futo\Ct{Pos}$. The base \ct{1} is certainly a distributive
category. And the object is initial, but the value of $P$ can be any
meet-semilattice.

We start with a review of how finite coproducts in \ct{C} are related to
finite coproducts of the $Q$-projectives.

\begin{prp}
Let $Q:\ct{C}\op\futo\Ct{Pos}$ be an \eqc, and let 
$\Yo:\proj{Q}\futo\ct{C}$ be the embedding of the full subcategory of the
$Q$-projectives. Then the following are equivalent:
\begin{enumerate}[{\rm(i)}]
\item \proj{Q} has a strict initial object.
\item \ct{C} has a strict initial object.
\end{enumerate}
\end{prp}

\begin{proof}
\entail{(i)}{(ii)} is immediate.

\entail{(ii)}{(i)} There must be a weak initial object in \proj{Q}. But
strictness of the initial object in \ct{C} entails that the weak initial
object is strictly initial.
\end{proof}

In general, given a diagram $J:\ct{K}\futo\proj{Q}$, any
colimit with vertex $C$ of the diagram $\Yo\circ J:\ct{K}\futo\ct{C}$
in \ct{C} induces a weak colimit with vertex $C$ of the original diagram $J$
in \proj{Q}, since the image of the embedding
$\Yo:\proj{Q}\futo\ct{C}$ is a cartesian projective cover of the 
\eqc $Q:\ct{C}\op\futo\Ct{Pos}$.
But, a finite coproduct of $Q$-projective objects is
$Q$-projective---in fact, any colimit of $Q$-projectives is. Thus the
vertex is in \proj{Q}.

\begin{rem}
For any  \eqc
$Q:\ct{C}\op\futo\Ct{Pos}$ generated by any \ed $E:\ct{B}\op\futo\Ct{Pos}$,
 \thx\ref{PPO} ensures that a vertex of a colimit
of $Q$-projectives is a retract of an object of the form
$(B,\delta_B)$. Therefore, the Karoubi envelope \K{\ct{B}} of \ct{B} has
finite coproducts, and the \ed extended to it by the 2-functor \K{}, would
still produce the same \eqc $Q$.
\end{rem}

\begin{dfnt}\label{defdpd}
Let $P:\ct{B}\op\futo\Ct{Pos}$ be a primary doctrine. We say that $P$ is
\dfn{distributive} when
\begin{enumerate}[{\rm(i)}]
\item for any initial object $0$, the poset $P(0)$ is the terminal poset; 
\item for every coproduct diagram
$\xymatrix@1@C=3em{B_1\ar[r]|-{\inj1}&B&B_2\ar[l]|-{\inj2}}$
the diagram
\[\xymatrix@C=4em{P(B_1)&
\ar[l]_-{P(\inj1)}P(B)\ar[r]^-{P(\inj2)}&P(B_2)}\]
is a product of posets.
\end{enumerate}
\end{dfnt}

\begin{exm}
If \ct{C} is an extensive category, then the \ed{}s
$\Sb{\ct{C}}:\ct{C}\op\futo\Ct{Pos}$ and
$\Wsb{\ct{C}}:\ct{C}\op\futo\Ct{Pos}$ are distributive, as one can see
directly from their characterisation in \cite[\prx2.2]{CarboniA:intedc}.
\end{exm}

\begin{rem}\label{drul}
It is useful to list the logical rules which correspond to each condition
for a distributive doctrine. Condition (i) corresponds to the rule
\[\Frac{}{x\type 0\mid\top\vdash\bot}\]
while condition (ii) corresponds to the rule
\[\Frac{x_1\type B_1\mid\phi(\inj1(x_1))\vdash\psi(\inj1(x_1))
\qquad x_2\type B_2\mid\phi(\inj2(x_2))\vdash\psi(\inj2(x_2))}
{x\type B_1+B_2\mid\phi(x)\vdash\psi(x)}\]
\end{rem}

\begin{dfnt}
In a disjunctive \ed $D:\ct{B}\op\futo\Ct{Pos}$
we say that an object $O$ satisfies Property \dis when
\begin{description}
\item[\dis:] $x\type O\mid\top\vdash\bot$.
\end{description}
A pair $\xymatrix@1@C=3em{B_1\ar[r]|-{f_1}&B&B_2\ar[l]|-{f_2}}$ satisfies
Property \dtc when
\begin{description}
\item[\dtc:] $x_1\type B_1,x_2\type B_2\mid
f_1(x_1)\seq{B}f_2(x_2)\vdash\bot$.
\end{description}
In case the doctrine $D$ is also existential, we say that the above pair  satisfies
Property \exh when
\begin{description}
\item[\exh:] $x\type B\mid\top\vdash
(\Exists{x_1\type B_1}x\seq{B}f_1(x_1))\Lor
(\Exists{x_2\type B_2}x\seq{B}f_2(x_2))$.
\end{description}
\end{dfnt}

In a distributive category, injections into a binary coproduct are monic,
see \cite{CarboniA:intedc}. In general, for an \ed
$E:\ct{B}\op\futo\Ct{Pos}$, arrows $f:A\to B$ which are monic in the base
category \ct{B} need not satisfy the property as stated in the logical
shorthand:
\begin{equation}\label{intmono}
x,x'\type A\mid f(x)\seq{B}f(x')\vdash x\seq{A}x'.
\end{equation}
which is stronger than monic, also in the distributive case. To distinguish
the two situations, we say that an arrow $f:A\to B$ that satisfies
\eqref{intmono} is $E$-\dfn{monic}.

Thanks to comprehensive diagonals, an $E$-monic arrow is monic in
\ct{B} ---the converse implication requires additional hypotheses, as stated in
 cor 4.8 of \cite{MaiettiME:quofcm}.

\begin{prp}\label{mono}
Let $E:\ct{B}\op\futo\Ct{Pos}$ is an \ed with comprehensive diagonals. Let
$m:X\to Y$ be an arrow in \ct{B}. If $m$ is $E$-monic, then it is
monic. Moreover, in case $E$ is variational, if $m$ is monic, then it is
$E$-monic.
\end{prp}

\begin{proof}
The condition that the arrow $m:X\to Y$ is monic is equivalent to
the condition that the commutative diagram
\[\xymatrix@C=4em@R=2.5em{X\ar[r]^-{m}\ar[d]_-{\ple{\id{X},\id{X}}}&
Y\ar[d]^-{\ple{\id{Y},\id{Y}}}\\
X\times X\ar[r]^-{m\times m}&Y\times Y}\]
is a pullback.

When $E$ is variational, that condition is equivalent to the condition that
$m:X\to Y$ is $E$-monic.  
\end{proof}

We now look at the interplay between \dfx\ref{defdpd} and 
Properties \dtc and \exh. It is convenient to indicate as
\dfn{disjunctive distributive} (\dd for short) those primary doctrines
$D:\ct{B}\op\futo\Ct{Pos}$ which are disjunctive and distributive, and on a
distributive category \ct{B} as a base.

\begin{prp}\label{injmon}
Let $D:\ct{B}\op\futo\Ct{Pos}$ be a \dd\ed.
Suppose $\xymatrix@1@C=3em{B_1\ar[r]|-{\inj1}&B&B_2\ar[l]|-{\inj2}}$ is a
coproduct diagram. Then the following hold:
\begin{enumerate}
\item each injection is $D$-monic;
\item the pair of injections satisfies Property \dtc.
\end{enumerate}
\end{prp}

\begin{proof}
Since the isomorphism in \dfx\ref{defdpd}(ii) is constructed from functions
which preserve the structure, it transforms a meet in a pair of
meets, and a left adjoint in a pair of left adjoints. So in the case of
$B\times B\isto (B_1\times B_1)+(B_1\times B_2)+
(B_2\times B_1)+(B_2\times B_2)$, reindexing along the four injections
transform $\delta_B$ in the quadruple 
\ple{\delta_{B_1},\bot_{B_1\times B_2},\bot_{B_2\times B_1},\delta_{B_2}}. 
In other words,
$\delta_{B_1}=D_{\inj1\times\inj1}(\delta_B)$,
$\bot_{B_1}=D_{\inj1\times\inj2}(\delta_B)$,
$\bot_{B_2}=D_{\inj2\times\inj1}(\delta_B)$ and
$\delta_{B_2}=D_{\inj2\times\inj2}(\delta_B)$.
\end{proof}

\begin{prp}\label{prev}
Let $D:\ct{B}\op\futo\Ct{Pos}$ be a \dd\ed. Then
\begin{enumerate}[{\rm(i)}]
\item the poset $D(O)$ is terminal if and only if
an object $O$ satisfies Property \dis;
\end{enumerate}
Moreover, if $D$ is existential, then for a pair
$\xymatrix@1@C=3em{B_1\ar[r]|-{f_1}&B&B_2\ar[l]|-{f_2}}$ 
of $D$-monos which satisfies Property \dtc, the following holds
\begin{enumerate}[{\rm(i)}]
\item[{\rm(ii)}]
the diagram of posets
$\xymatrix@C=3.5em{D(B_1)&\ar[l]_-{D_{f_1}}D(B)\ar[r]^-{D_{f_2}}&D(B_2)}$
is a product if and only if
the pair satisfies Property \exh.
\end{enumerate}
\end{prp}

\begin{proof}
(i) Immediate.

\noindent(ii) First note that, if Property \dtc holds, the function
\begin{equation}\label{sepsec}
\ple{D_{f_1},D_{f_2}}:D(B)\to D(B_1)\times D(B_2)
\end{equation}
has a section $\mathsf{sec}$ which send a pair \ple{\gamma',\gamma''} to the
predicate $\gamma$ given by the formula 
\[\Exists{x_1\type B_1}(\gamma'(x_1)\Land x\seq{B}f_1(x_1))\Lor
\Exists{x_2\type B_2}(\gamma''(x_2)\Land x\seq{B}f_2(x_2))\]
in the variable $x\type B$. Indeed, Property \dtc ensures that the second
summand in $\os{D_{f_1}(\gamma)}(y)$ is $\bot$, \ie
\[y\type B_1\mid 
\Exists{x_2\type B_2}(\gamma''(x_2)\Land f_1(y)\seq{B}f_2(x_2))\vdash\bot\]
while the first summand in $\os{D_{f_1}(\gamma)}(y)$ is
\[y\type B_1\mid 
\Exists{x_1\type B_1}(\gamma'(x_1)\Land f_1(y)\seq{B}f_1(x_1))\dashv\vdash
\gamma'(y)\]
by the hypothesis that $f_1$ is $D$-monic. Similarly one proves
that $D_{f_2}(\gamma)=\gamma''$.

Now, consider predicates $\alpha$  on $B$ and write
$\alpha_i\colon=D_{f_{i}}(\alpha)$  Note that
\[x\type B\mid\alpha(x)\Land(\Exists{x_1\type B_1}x\seq{B}f_1(x_1))
\dashv\vdash
\Exists{x_1\type B_1}(\alpha_1(x_1)\Land x\seq{B}f_1(x_1))\]
So Property \exh yields that that \eqref{sepsec}
is one-one with inverse $\mathsf{sec}$.

Conversely, assume that  the function in \eqref{sepsec} is one-one. Then its
inverse is the section $\mathsf{sec}$ computed above. This yields
immediately that the pair of $f_1$ and $f_2$ satisfies Property \exh.
\end{proof}

\begin{cor}
Let $D:\ct{B}\op\futo\Ct{Pos}$ be a disjunctive existential \ed. Then $D$ is
\dd if and only if $D$ is such that
\begin{enumerate}[{\rm(i)}]
\item an initial object $0$ satisfies Property \dis;
\item for every coproduct diagram
$\xymatrix@1@C=3em{B_1\ar[r]|-{\inj1}&B&B_2\ar[l]|-{\inj2}}$,
each injection is $D$-monic, and the pair of them satisfies Properties \dtc
and \exh.
\end{enumerate}
\end{cor}

With an argument similar to that of \prx\ref{injmon}, one characterises
$D$-equivalence relations on a coproduct.

\begin{prp}
Let $D:\ct{B}\op\futo\Ct{Pos}$ be a \dd\ed.
Let
$\xymatrix@1@C=3em{B_1\ar[r]|-{\inj1}&B&B_2\ar[l]|-{\inj2}}$
be a coproduct diagram in \ct{B}. If $\rho_1$ is a $D$-equivalence relation
on $B_1$ and $\rho_2$ is a $D$-equivalence relation on $B_2$, then there is
a $D$-equivalence relation on $B$ such that
\begin{enumerate}[{\rm(i)}]
\item $D_{\inj{i}\times\inj1}(\rho)=\rho_i$ for $i=1,2$;
\item $D_{\inj{i}\times\inj{j}}(\rho)=\bot_{B_i}$ for $i\neq j$.
\end{enumerate}
\end{prp}

\begin{proof}
The four conditions determine a relation on $B$ which is easily seen to be
symmetric and transitive with an argument similar to that of the proof in
\prx\ref{injmon}. To see that it is also reflexive, rewrite the condition of
reflexivity of $\rho$ as $\delta_B\leq\rho$ and consider its four
reindexings along the injections
\[\xymatrix@C=3em@R=1ex{B_1\times B_1\ar[rd]&&B_1\times B_2\ar[dl]\\
&B\times B\\
B_2\times B_2\ar[ur]&&B_2\times B_1\ar[ul]
}\]
\end{proof}

The previous results can be summarised in the following theorem.

\begin{thm}\label{chardist}
Let $Q:\ct{C}\op\futo\Ct{Pos}$ be an \eqc, and let 
$\Yo:\proj{Q}\futo\ct{C}$ be the embedding of the full subcategory of the
$Q$-projectives. Then the following are equivalent:
\begin{enumerate}[{\rm(i)}]
\item $Q\Yo\op$ is \dd.
\item $Q$ is \dd.
\end{enumerate}
\end{thm}

\section{Local cartesian closure}\label{lcc}

In this section we consider when the base of an \eqc
$Q:\ct{C}\op\futo\Ct{Pos}$ is locally cartesian closed (abbreviated \lcc),
determining conditions on the restriction 
$Q':\proj{Q}\op\futo\Ct{Pos}$ to the full subcategory of the
$Q$-projectives.

There is already such a characterisation in \cite{MaiettiME:loccce}, where
an explicit construction of the exponentials in each slice category of the
base \qb{E} is provided. Here, however, we present a different proof which
shows how the descent argument employed in \cite{EmmeneggerJ:ontlc} to
determine when an exact completion  $\ct{C}\exw$ is locally cartesian
closed can also be applied in the framework of doctrines. A novel feature of
is that the proof relies solely on the concept of cartesian closure. 

Since we are interested in \eqc{}s whose base is \lcc,
the first observation is that \ct{C} must be endowed with pullbacks. In
turn, this is mirrored in the condition that the base \proj{Q} of the \ed
$Q'$ must have weak pullbacks, with an argument totally similar to that in
the opening of \S\ref{fdcsec}.

We introduce some technical notions which will be used in the proof of the
characterisation in \crx\ref{lccmain}.

Let \ct{C} be a category with finite products and let $J:\ct{D}\to\ct{C}$
be an inclusion of a subcategory in it. Recall that an object $B$ is
\dfn{weakly exponentiable relative to} $J$ if the functor
\[B\times J(-):\ct{D}\futo\ct{C}\]
is a \dfn{weak left adjoint}, in the sense of \cite{KainenP:weaaf}, \ie for
every object $Y$ in \ct{C} there are an object $W$ in \ct{D} and an arrow 
\[\xymatrix{B\times J(W)\ar[r]^-{\ev}& Y}\]
in \ct{C} such that for every $D$ in \ct{D} and every arrow
$f:B\times J(D)\to Y$ there are an arrow $f':D\to W$ and a commutative
diagram
\[\xymatrix{B\times J(D)
\ar[d]_-{\id{B}\times J\left(f'\right)}\ar[rd]^-{f}&\\
B\times J(W)\ar[r]^-{\ev}& Y.}\]
We may call the object $J(W)$ a \dfn{relative weak exponential} of $Y$ and
$B$.

The condition of weak left adjoint provides a family of surjective functions
\[\xymatrix{\ct{D}(D,W)\suar[r]&\ct{C}(B\times J(D),Y)}\]
which is natural in $D$.

We are interested in the case when all the objects $J(A)$, in the image of the
functor $J$, are weakly exponentiable relative to $J$.

\begin{dfnt}\label{rwc}
Let $J:\ct{D}\to\ct{C}$ be an inclusion in a category \ct{C} with finite
products. We say that \dfn{\ciwe{$J$}} if,
for all objects $A$ in \ct{D}, the object $J(A)$ is weakly exponentiable
relative to $J$. In other words, for all objects $A$ in \ct{D} and all
objects $Y$ in \ct{C}, there is a relative weak exponential of $Y$ and
$J(A)$ of the form $J(W)$.
\end{dfnt}

\begin{rem}\label{bwcc}
Consider a category \ct{C} with finite products. An object $A$ is weakly
exponentiable in the sense of \cite{RosoliniG:loccce} if (and only if) it is
weakly exponentiable relative to \Id{\ct{C}}. So \ct{C} is weakly cartesian
closed if and only if \ciwe{\Id{\ct{C}}}.
\end{rem}

A first test for the notion of weak relative exponentiability is to
characterise when the base of an \eqc is cartesian closed. The strategy of
proof is similar to those of \cite[\thx2.5]{RosoliniG:loccce} and
\cite[\lmx2.12]{EmmeneggerJ:ontlc}. The result is also instrumental for the
characterisation of the \eqc which are \lcc which we
obtain with \thx\ref{fstform}.

\begin{prp}\label{cartcloch}
Let $Q:\ct{C}\op\futo\Ct{Pos}$ be a variational universal \ed with
comprehensive diagonals and quotients of effective descent. Let
\ct{B} be a full subcategory of \ct{C}. If \ct{B} is a projective cover of
$Q$, then the following are equivalent:
\begin{enumerate}[{\rm(i)}]
\item \ct{C} is cartesian closed.
\item \ciwe{The embedding $J:\ct{B}\emto\ct{C}$}.
\end{enumerate}
\end{prp}

\begin{proof}
Note that the condition in (ii) requires a natural surjection
\begin{equation}\label{basexp}
\xymatrix@1@C=2.5em{\ct{C}(J(A),J(W))\isar[r]
\ar@/_10pt/[rr]_-{e_A}&
\ct{B}(A,W)\suar[r]&
\ct{C}(B/\delta_B\times J(A),Y/\sigma)}
\end{equation}
where we denote an arbitrary object in \ct{C} by an appropriate quotient of
a $Q$-projective $J(Y)$ with a $Q$-equivalence relation $\sigma$ simply as $Y/\sigma$, even in the
case of a trivial quotient by a relation $\delta$ to stress the different
roles of the objects.

Moreover, given two arrows $f,g:A\to J(W)$, the condition
$e_A(J(f))=e_A(J(g))$ is equivalent to the condition 
\[x\type B,y\type A\mid \top\vdash\sigma(\ev(x,f(y)),\ev(x,g(y))).\]
which in turn is equivalent to $q\circ J(f)=q\circ J(g):J(A)\to W/\omega_W$ 
in \ct{C}, where $q:J(W)\onto W/\omega_W$ is the quotient in \ct{C} with the
$Q$-equivalence relation  
\[\omega_W\colon=\Forall{x\type B}\sigma(\ev(x,w),\ev(x,w))\]
thanks to the assumption that $Q$ is universal.
In other words, the quotient of the kernel pair of the surjection $e_A$
induces a bijection
\[\xymatrix@C=2.5em@R=2.5em{
\ct{C}(J(A),J(W))\isar[r]\ar[rd]_-{e_A}\suar[d]&
\ct{B}(A,J(W))\suar[d]\\
\ct{C}(J(A),W/\omega_W)\isar[r]&
\ct{C}(B/\delta_B\times J(A),Y/\sigma)}\]
displaying $J(W)/\omega_W$ as a representation of the universal problem which
determines the exponential $(Y/\sigma)^{(B/\delta_B)}$ in \ct{C} thanks to
the hypotheses on \ct{B}. 

\entail{(i)}{(ii)} Immediate from the previous observations.

\entail{(ii)}{(i)} 
By \lmx\ref{regularepi}, for an arbitrary object $B/\rho$, there is a
coequaliser 
\[\xymatrix@C=2em{R/\delta_R
\ar@<.5ex>[rr]^-{[r_1]}\ar@<-.5ex>[rr]_-{[r_2]}&&
B/\delta_B\onar[r]&B/\rho}\]
So there are natural bijections, given by \eqref{basexp}, in a commutative
diagram of transformations, natural in $A$ 
\[\xymatrix{\ct{C}(J(A),W/\omega_W)\isar[r]
\ar@<.8ex>[d]^-{\ec{r_2'}\circ\blank}\ar@<-.8ex>[d]_-{\ec{r_1'}\circ\blank}&
\ct{C}(B/\delta_B\times J(A),Y/\sigma)
\ar@<.8ex>[d]^-{\blank\circ\ec{r_2\times\id{j(A)}}}
\ar@<-.8ex>[d]_-{\blank\circ\ec{r_1\times\id{j(A)}}}\\
\ct{C}(J(A),Z/\omega_Z)\isar[r]&
\ct{C}(R/\delta_R\times J(A),Y/\sigma)}\]
Since \ct{C} has pullbacks, consider the equaliser in \ct{C}
\[\xymatrix@C=2em{X/\xi\ar[r]^-{\ec{h}}&W/\omega_W
\ar@<.5ex>[rr]^-{\ec{r_1'}}\ar@<-.5ex>[rr]_-{\ec{r_2'}}&&Z/\omega_Z}\]
by naturality one obtains
\[\xymatrix@C=3em@R=3em{\ct{C}(J(A),X/\xi)\isar[r]
\moar[d]^-{\ec{h}\circ\blank}&
\ct{C}(B/\rho\times J(A),Y/\sigma)\moar[d]\\
\ct{C}(J(A),W/\omega_W)\isar[r]
\ar@<.8ex>[d]^-{\ec{r_2'}\circ\blank}\ar@<-.8ex>[d]_-{\ec{r_1'}\circ\blank}&
\ct{C}(B/\delta_B\times J(A),Y/\sigma)
\ar@<.8ex>[d]^-{\blank\circ\ec{r_2\times\id{j(A)}}}
\ar@<-.8ex>[d]_-{\blank\circ\ec{r_1\times\id{j(A)}}}\\
\ct{C}(J(A),Z/\omega_Z)\isar[r]&
\ct{C}(R/\delta_R\times J(A), Y/\sigma)}\]
The conclusion follows from the hypotheses on \ct{B}.
\end{proof}

\begin{rem}
The construction is very much in line with that of
\cite[\lmx5.1]{RobinsonE:colcet} which was used to represent functors on
an exact completion, see also \cite[\lmx3.7]{MenniM:chalec} and
\thx\ref{wpow}.
\end{rem}

\begin{rem}\label{firstrem}
Observe that in the proof of \prx\ref{cartcloch} the construction of
the equaliser  
\[\xymatrix@C=2em{X/\xi\ar[r]^-{\ec{h}}&W/\omega_W
\ar@<.5ex>[rr]^-{\ec{r_1'}}\ar@<-.5ex>[rr]_-{\ec{r_2'}}&&Z/\omega_Z}\]
provides a  weak comprehending arrow $X\to W$ for the predicate
determined by the formula
\[\Forall{x,x'\type B}(\rho(x,x')\Implies\sigma(\ev(x,w),\ev(x',w)))\]
which requires that $Q'$ is also implicational. Then, we 
take a quotient by the reindexing on $X$ of the relation $\omega_W$.
\end{rem}

\begin{rem}
To understand how the argument of \prx\ref{cartcloch} is related to that of
\cite[\thx2.14]{EmmeneggerJ:ontlc} consider that in the case of the doctrine
of variations $\Wsb{}:\ct{C}\op\futo\Ct{Pos}$ the request that
\ciwe{$J:\ct{C}\futo\qb{\Wsb{}}$} corresponds to that of extensional
exponentials \cite[\dfx2.4]{EmmeneggerJ:ontlc}.
\end{rem}

\prx\ref{cartcloch} is an extension of \cite[\thx1]{RosickyJ:carcec},
and we are interested in extending it further to characterise local
cartesian closure for an \eqc. In order to motivate what follows, we show
how a general argument allows to reduce the general case of a fibre
\comc{\ct{C}}{B/\rho} to that of the fibre over a $Q$-projective
\comc{\ct{C}}{B/\delta_B}.

\begin{prp}\label{Barr}
Let \ct{C} be a cartesian closed category with pullbacks. Let $q:X\to D$
be an arrow of effective descent.
If \comc{\ct{C}}{X} is cartesian closed, then also \comc{\ct{C}}{D}
is cartesian closed.
\end{prp}

\begin{proof} 
Let $f:A\to D$ be an object in
\comc{\ct{C}}{D}. Consider a pullback
\[\xymatrix{A'\ar[r]^-{q'}\ar[d]_-{f'}&A\ar[d]^-{f}\\
X\ar[r]^-{q}&D}\]
The assumption that \comc{\ct{C}}{X} is cartesian closed, gives an
adjunction
\[\xymatrix{\comc{\ct{C}}{X}\ar@<1.2ex>[rr]^-{-\times_X f'}
\ar@{{}{}}[rr]|>{}|-*=0[@]{\bot}&&
\comc{\ct{C}}{X}\ar@<1.2ex>[ll]^-{\Exp_X(f')}.}\] 
Consider the diagram with three pairs of adjoints
\[\xymatrix@C=2em@R=3ex{\comc{\ct{C}}{D}
\ar[dd]_-{-\times_{D}f}^(.7){}="a"
\ar@<1.1ex>[rr]^-{q^\ast}\ar@{{}{}}[rr]|>{}|-*=0[@]{\top}&\kern5em&
\comc{\ct{C}}{X}\ar@<1.1ex>[ll]^-{\postc{q}{}}
\ar@<-1.1ex>[dd]_-{(-\times_X f')}\ar@{{}{}}[dd]|>{}|-*=0[@]{\top}\\
\\
\comc{\ct{C}}{D}\ar@<-1.1ex>[rr]_-{q^\ast}
\ar@{{}{}}[rr]|>{}|-*=0[@]{\bot}&
&\comc{\ct{C}}{X}\ar@<-1.1ex>[ll]_-{\postc{q}{}}_(.72){}="b"
\ar@<-1.1ex>[uu]_-{\Exp_X(f')}\isar@/^/"a";"b"}\]
where the natural iso 
$(\blank\times_{D}f)\circ\postc{q}{}\ntto
\postc{q}\circ(\blank\times_X f')$ is obtained from pullback pasting. By \ref{descentdoc}
 the functor $q^\ast$ is monadic. So, to apply
\cite[\thx1,~\S3.7]{BarrM:toptt},we just need to check that
$(-\times_{D}f):\comc{\ct{C}}{D}\futo\comc{\ct{C}}{D}$
preserves coequalisers of $q^\ast$-contractible pairs: this follows from the
natural iso $q^\ast\circ(-\times_{D}f)\ntto
(\blank\times_X f')\circ q^\ast$ obtained again from pullback pasting.
\end{proof}

As suggested, \prx\ref{Barr} reduces the question of when an \eqc is locally
cartesian closed when just the slice 
category over a $Q$-projective is cartesian closed. We shall address this in
line with the approaches in
\cite{RosoliniG:loccce,EmmeneggerJ:ontlc} by describing the slice category 
\comc{\qb{D}}{(B,\delta_B)} so as to apply \prx\ref{cartcloch}.

\begin{rem}\label{expdesc}
Given a variational
\ed $E:\ct{B}\op\futo\Ct{Pos}$ on a category \ct{B}, let
$Q:\ct{C}\op\futo\Ct{Pos}$ be its \eqc, \ie $\ct{C}=\qb{E}$ and
$Q=\qd{E}$. Write $\Yo:\ct{B}\futo\ct{C}$ for the full embedding of the
$Q$-projectives from the base of the generating doctrine, so that its image
gives a cartesian projective cover of $Q$.

Consider the slice category \comc{\ct{C}}{\Yo(B)} on such a $Q$-projective,
and the doctrine $Q':\comc{\ct{C}}{\Yo(B)}\op\futo\Ct{Pos}$ induced on the
slice category by change of base along the domain functor
$\dom:\comc{\ct{C}}{\Yo(B)}\futo\ct{C}$.

Then the functor $\comc{\Yo}{B}:\comc{\ct{B}}{B}\futo\comc{\ct{C}}{\Yo(B)}$
induced by $\Yo$ is full and faithful, and its image
satisfies the hypotheses for the full subcategory in \prx\ref{cartcloch}.

Indeed, to see that in \comc{\ct{C}}{\Yo(B)} an arbitrary object 
$\ec{f}:(A,\rho)\to(B,\delta_B)$ is the codomain of a $Q$-quotient, one
can take the projective $\Yo(f)$ and consider the $Q'$-quotient 
\[\xymatrix@C=2em@R=2ex{(A,\delta_A)\ar[rd]_(.4){\Yo(f)}
\ar[rr]^-{\ec{\id{A}}}&&(A,\rho)\ar[ld]^(.4){\ec{f}}\\
&(B,\delta_B).&}\] 
\end{rem}

\begin{rem}
It is possible to draw a comparison with the notion of ``slicewise weakly
cartesian closed'' in \cite[\dfx3.5]{MaiettiME:loccce} for a
variational, implicational and universal \ed $E:\ct{B}\op\futo\Ct{Pos}$ with
comprehensive diagonals, which, in terms of \dfx\ref{rwc}, requires that
\ciwe{$\comc{\Yo}{B}:\comc{\ct{B}}{B}\futo\comc{\qb{E}}{\Yo(B)}$}.
As already pointed out in \rmx\ref{firstrem}, the assumptions that the \ed
is variational and implicational ensure that the base \qb{E} of the \eqc has
pullbacks. We found it more stringent to require just that condition when
needed, see \eg the following \thx\ref{fstform}.
\end{rem}

Thanks to \rmx\ref{expdesc} and \prx\ref{cartcloch}, we obtain a
characterisation of the universal, implicational \eqc{}s which are locally
cartesian closed.

\begin{thm}\label{fstform}
Let $Q:\ct{C}\op\futo\Ct{Pos}$ be an m-variational, universal,
implicational \eqc.
Let $\Yo:\proj{Q}\futo\ct{C}$ be the embedding of the full subcategory on
the $Q$-projectives. If \ct{C} has pullbacks, then the following are
equivalent:
\begin{enumerate}[{\rm(i)}]
\item \ct{C} is \lcc.
\item
\ciwe{$\comc{\Yo}{B}:\comc{\proj{Q}}{B}\futo\comc{\ct{C}}{\Yo(B)}$}, for
every $Q$-projective $B$.
\end{enumerate}
\end{thm}

\begin{proof}
\entail{(i)}{(ii)} Straightforward.
  
\entail{(ii)}{(i)}
We show that a fibre \comc{\ct{C}}{\Yo(B)} on a $Q$-projective is cartesian
closed. Then the case of the general fibre follows from condition (ii)
by \prx\ref{Barr}.

The proof for the case of the fibre on a $Q$-projective is an application of
\prx\ref{cartcloch}. 
Write \ct{B} the full subcategory \proj{Q} on the $Q$-projectives, and write
$E$ the restriction of the \ed $Q$ to it. Note that, since $Q$ has a
cartesian projective cover, \ct{B} has weak pullbacks.

Thanks to \rmx\ref{expdesc} we know that the image of the functor
$\comc{\Yo}{B}:\comc{\proj{Q}}{B}\futo\comc{\ct{C}}{\Yo(B)}$ satisfies
the hypotheses on the full subcategory in \prx\ref{cartcloch}. So, in order
to apply \prx\ref{cartcloch}, we need only to check that the \ed
$Q':\comc{\ct{C}}{\Yo(B)}\op\futo\Ct{Pos}$
is universal. A product in \comc{\ct{C}}{\Yo(B)} is a pullback diagram
\[\xymatrix@R=1ex@C=1.5em{
&&(A,\rho)\ar[lld]_-{\ec{\pr1}}\ar[rrd]^-{\ec{\pr2}}\ar[ddd]&&\\
(A_1,\rho_1)\ar[rrdd]_-{\ec{f_1}}&&&&(A_2,\rho_2)\ar[lldd]^-{\ec{f_2}}\\
\strut&&&&\\
&&(B,\delta_B).&&}\]
So $\rho$ is the predicate determined by the formula
\[\rho_1(\pr1(x),\pr1(x'))\Land\rho_2(\pr2(x),\pr2(x'))\]
in the variables $x,x'\type A$.

The value of the right adjoint to
$Q'{\ec{\pr1}}:Q'(A_1,\rho_1)\to Q'(A,\rho)$ applied to the predicate
$\alpha$ in $Q'(A,\rho)$ is given by the formula
\[\Forall{x\type A}\left(\rho_1(\pr1(x),x_1)\Implies\alpha(x)\right)\]
in the variable $x_1\type A$.

To see the adjunction consider that, for $\beta$ a predicate on
$(A_1,\rho_1)$, it is
\[x_1\type A_1\mid\beta(x_1)\vdash
\Forall{x\type A}\left(\rho_1(\pr1(x),x_1)\Implies\alpha(x)\right)\]
if and only if
\[x\type A\mid\beta(\pr1(x))\vdash\alpha(x)\]
because $\beta$ is descent datum for $\rho_1$.
\end{proof}

\begin{cor}\label{lccmain} 
Let $D:\ct{B}\op\futo\Ct{Pos}$ be a variational \ed.
Let $\qd{D}:\qb{D}\op\futo\Ct{Pos}$ be its \eqc. Then the following are
equivalent:
\begin{enumerate}[{\rm(i)}]
\item \qb{D} is \lcc.
\item $D$ is universal and implicational, and for
every $Q$-projective $B$, the functor
\ciwe{$\comc{\Yo}{B}:\comc{\ct{B}}{B}\futo\comc{\qb{D}}{\Yo(B)}$}.
\end{enumerate}
\end{cor}

\begin{proof}
By \prx\ref{abcde+}, the doctrine is obtained from a topology $j$ on
the doctrine $\Wsb{}:\ct{B}\op\futo\Ct{Pos}$. Condition (i) ensures that
$\Wsb{}$ is universal and implicational by \rmx\ref{addrem}; hence, so is
$D$. Condition (ii) does so explicitly. Thus, in either case \qd{D} does by
\rmx\ref{eqc-impl}. The conclusion follows from \thx\ref{fstform}.
\end{proof}

Recall that, given a category \ct{C} with finite products and weak
pullbacks, the doctrine of subobjects
$\Sb{\ct{C}\exw}:\ct{C}\exw\op\futo\Ct{Pos}$ of 
an exact completion $\ct{C}\exw$ is an instance of \eqc. It is obtained as
the \eqc of the doctrine of variations $\qd{{\Wsb{}}}:\ct{C}\futo\Ct{Pos}$;
so, in particular, $\qb{\Wsb{}}\equiv\ct{C}\exw$ and the
\qd{{\Wsb{}}}-projectives are exactly the regular projectives in $\ct{C}\exw$.
Thus, as a byproduct of \crx\ref{lccmain}, we recover the characterisation
in \cite{RosoliniG:loccce} and \cite{EmmeneggerJ:ontlc} in the case \ct{C}
has finite products.

\begin{cor}\label{lccccor}
Suppose \ct{C} a category with finite products and weak pullbacks. Then
$\ct{C}\exw$ is \lcc if and only if for every \qd{{\Wsb{}}}-projective
object $B$, the functor
\ciwe{$\comc{\Yo}{B}:\comc{\ct{C}}{B}\futo\comc{\ct{C\exw}}{\Yo(B)}$}.
\end{cor}

\crx\ref{lccmain} provides an alternative proof to the characterisation of local cartesian closure established in
\cite{MaiettiME:loccce}, by just relying on cartesian
closedness and the use of a projective cover with a descent
argument. It also provides a proof to the
case of doctrines with only weak finite limits, known as ``biased doctrines" in
 \cite{CioffoC:biaedq}, by enforcing the characterization of
locally cartesian closed elementary quotient completions for
 universal biased 
doctrines established in  \cite{cioffothesis}.

\section{Generic predicate}

In this section we show how the \eqc of a suitable doctrine $P$ inherits a
generic predicate in the sense of \cite{PittsA:tritir}, whose definition we
recall. For a primary doctrine $P:\ct{B}\op\futo\Ct{Pos}$, a 
\dfn{generic predicate} is a predicate \gp over an object \ogp such that
for every predicate $\alpha$ over $B$ there is an arrow
$\chi:B\to\ogp$ such that $P_{\chi}(\gp)=\alpha$. 

Similar to the situation for relative weak exponentiability in \S\ref{lcc},
the condition of a generic predicate provides a 
family of surjective functions
\begin{equation}\label{surgp}
\xymatrix{\ct{B}(D,\ogp)\suar[r]&\udl{P(D)}}
\end{equation}
which is natural in $D$, into the underlying set of the poset $P(D)$.

Beyond those which appear in \cite{PittsA:tritir}, 
an example of generic predicate is the weak proof classifier of
\cite{MenniM:chalec} which is exactly a generic predicate for a doctrine of
the form $\Wsb{}:\ct{C}\op\futo\Ct{Pos}$ when \ct{C} has weak pullbacks.

A generic predicate is a \dfn{predicate classifier} if the existing arrow
$\chi$ is also unique, \ie the natural surjection in \eqref{surgp} is also
one-one. The paradigmatic example is that of the subobject classifier which is
exactly a predicate classifier for a doctrine of the form
$\Sb{\ct{C}}:\ct{C}\op\futo\Ct{Pos}$ when \ct{C} has finite limits.

A parametrised version of generic predicate goes under the name of
``weak powerobject'', for the obvious connection with the standard powerset
in the doctrine of $\Sb{\ct{Set}}:\ct{Set}\op\futo\Ct{Pos}$, and with the
(strong) 
powerobjects in an elementary topos. A \dfn{weak powerobject} of an object
$B$ in $\ct{B}$ is an object \pow{B} together with a predicate $\mbs{B}$
over the product $\pow{B}\times B$ 
such that for every relation $\rho$ from $D$ to $B$ there is an arrow
$r:D\to\pow{B}$ such that $P_{r\times\id{B}}(\mbs{B})=\rho$. 
Like for \eqref{surgp} the condition ensures a family of natural surjections
\begin{equation}\label{surwpo}
\xymatrix{\ct{B}(D,\pow{B})\suar[r]&\udl{P(D\times B)}.}
\end{equation}
Clearly, a generic predicate is a weak powerobject of a terminal
object. Similarly, for a weakly exponentiable object $B$, a weak
exponential of \gp and $B$ produces a weak powerobject of $B$.

Examples come mainly from triposes as in
\cite{HylandJ:trit,PittsA:tritir,OostenJ:reaahe,OostenJ:reaait}, which are 
exactly first order doctrines with weak powerobjects.

A \dfn{powerobject} is a weak powerobject such that the natural surjection
in \eqref{surwpo} is also one-one.

\begin{rem}\label{logop}
The weak representability in \eqref{surgp} allows to describe the logical
structure on the doctrine $P$ directly with operations on \ogp.
Let $P:\ct{B}\op\futo\Ct{Pos}$ be a primary doctrine, and \gp generic
predicate in it. Then there is an arrow $\cng:\ogp\times\ogp\to\ogp$ which
represents conjunction of predicates in the following sense:

Given any two predicates $\alpha_1$ and $\alpha_2$ on the object $B$, if
$\chi_i:B\to\ogp$, $i=1,2$ be such that $\alpha_i=P_{\chi_i}(\gp)$, then
$\alpha_1\Land_B\alpha_2=P_{\cng\circ\ple{\chi_1,\chi_2}}(\gp)$. In logical
notation, it is
\[x\type B\mid\alpha_1(x)\Land\alpha_2(x)\dashv\vdash
\gp(\chi_1(x)\cng\chi_2(x))\]
where the binary operation $\cng$ on \ogp is written infix.

Similar arguments apply for the other connectives and operators. So, for the
true predicate $\top$ there is an arrow $\tru:1\to\ogp$ with $\top_1=P_{\tru}(\gp)$ and
$\top_B=P_{\tru\circ\tot{B}}(\gp)$ for any object $B$.

In case $P$ is implicational, there is an operation
$\eqv:\ogp\times\ogp\to\ogp$ such that, for any two predicates $\alpha_1$
and $\alpha_2$ on the object $B$, 
\begin{equation}\label{intequiv}
x\type B\mid\alpha_1(x)\Iff\alpha_2(x)\dashv\vdash
\gp(\chi_1(x)\eqv\chi_2(x)).
\end{equation}
The right-hand side displays a $P$-equivalence relation on \ogp which
will be relevant in what follows: we shall write \pgp for it. It is
determined by the formula $\gp(w\eqv w')$ in the variables
$w,w'\type\ogp$, and strenghtens \eqref{surgp} to the following quotient of
natural transformations 
\begin{equation}\label{strurgp}
\vcenter{\xymatrix@C=.1ex@R=2em{\set{\ple{\xi_1,\xi_2,\xi_1\eqv\xi_2}\in
\ct{B}(D,\ogp)^3}{\top_D\leq \gp(\xi_1(x)\eqv\xi_2(x))}
\ar@<-1ex>[d]_-{\pr1}\ar@<1ex>[d]^-{\pr2}\\
\ct{B}(D,\ogp)\onar[r]&\udl{P(D)}}}
\end{equation}
which is expressed in terms of the structure of the doctrine.
\end{rem}

But the interaction of a generic predicate with the structure of the
doctrine can provide useful properties, such as in the following case. The
result will become useful in \S\ref{when}.

\begin{prp}\label{gendd}
Let $P:\ct{B}\op\futo\Ct{Pos}$ be an implicational doctrine with a
generic predicate \gp on the object \ogp. If \ct{B} has finite coproducts,
then $P$ is distributive.
\end{prp}

\begin{proof}
Let $0$ be an initial object. By definition of generic predicate, there is a
natural surjection 
\[\xymatrix{\ct{B}(0,\ogp)\suar[r]&\udl{P(0)}}\]
as an instance of \eqref{surgp}. Thus the poset $P(0)$ is terminal.

Consider now a coproduct diagram
$\xymatrix@1@C=3em{B_1\ar[r]|-{\inj1}&B&B_2\ar[l]|-{\inj2}}$. So 
the diagram
\[\xymatrix@1@C=3em{0\ar[r]^-{\tot{B_2}}\ar[d]_-{\tot{B_1}}
&B_2\ar[d]^-{\inj2}\\
B_1\ar[r]^-{\inj1}&B}\]
is a pushout. Hence in the diagram
\[\xymatrix@C=.2ex@R=2ex{&\ct{B}(B,\ogp)\suar[rrrrr]
\ar[ddl]_-{\blank\circ\inj1}
\ar[dddrr]^-{\blank\circ\inj2}|!{[ldd];[rrrrdd]}\hole&&&&&
\udl{P(B)}\ar[ddl]_-{P(\inj1)}\ar[dddrr]^-{P(\inj2)}\\
&&&&&&&&\\
\ct{B}(B_1,\ogp)\suar[rrrrr]\ar[rrddd]_-{\blank\circ\tot{B_1}}&&&&&
\udl{P(B_1)}\ar[rrddd]_-{P(\tot{B_1})}&&&\\
&&&\ct{B}(B_2,\ogp)\suar[rrrrr]|!{[rru];[rrrrdd]}\hole
\ar[ldd]^-{\blank\circ\tot{B_2}}
&&&&&\udl{P(B_2)}\ar[ldd]^-{P(\tot{B_2})}\\
&&&&&&&&\\
&&\ct{B}(0,\ogp)\suar[rrrrr]&&&&&
\udl{P(0)}&&
}\]
the left-hand square is a pullback. So by \eqref{strurgp} also the
right-hand square is a pullback.
\end{proof}

Note though that one should not expect that algebraic properties of the
operations hold formally since lack of uniqueness prevents the possibility
to derive identities on arrows. But in case the generic predicate is a
predicate classifier, then the algebraic identities involving the logical
operations are inherited by the corresponding operations on \ogp.

Like in the case for weak exponentiability, the passage to the \eqc
strengthens the condition of generic predicate.

\begin{thm}\label{tclassifier}
Let $P:\ct{B}\op\futo\Ct{Pos}$ be an implicational doctrine. Then the
following are equivalent:
\begin{enumerate}
\item \qd{P} has a predicate classifier.
\item $P$ has a generic predicate.
\end{enumerate}
\end{thm}

\begin{proof}
\entail{(i)}{(ii)} 
Immediate by \thx\ref{PPO}.

\entail{(ii)}{(i)} 
Suppose \gp is a generic predicate on \ogp in $P$. Consider the object 
$(\ogp,\pgp)$ in \qb{P}. Consider a predicate $\alpha$ on $(B,\rho)$, let
$\chi:B\to\ogp$ be an arrow such that $\alpha=P_{\chi}(\gp)$. Since
$P$ is implicational, note that the condition that $\alpha$ is descent datum
can be written as
\[x,x'\type B\mid \rho(x,x')\vdash\alpha(x)\Iff\alpha(x').\]
By \eqref{intequiv}, this ensures that $\chi$ preserves relations:
\[x,x'\type B\mid \rho(x,x')\vdash\pgp(\chi(x),\chi(x'))\]
so that $\ec{\chi}:(B,\alpha)\to(\ogp,\pgp)$.
Suppose that also $\phi:B\to\ogp$ is such
that $\alpha=P_{\phi}(\gp)$, so
\[x\type B\mid\gp(\chi(x))\dashv\vdash\gp(\phi(x)).\]
Since $P$ is implicational, this can be written as
\[x\type B\mid\top\vdash\gp(\chi(x))\Iff\gp(\phi(x)).\]
By \eqref{intequiv}, this is equivalent to
$x\type B\mid\top\vdash\gp(\chi(x)\eqv\phi(x))$, \ie
$P_{\ple{\chi,\phi}}(\pgp)=\top$. Therefore
$\ec{\chi}=\ec{\psi}:(B,\delta_B)\to(\ogp,\pgp)$ in \qb{P}.
This proves that the predicate \gp on $(\ogp,\pgp)$ classifies in \qd{P}
since every predicate in \qb{P} is descent datum on an object
$(B,\delta_B)$.
\end{proof}

Further properties on the generic predicate produce strengthening of
\thx\ref{tclassifier}.

\begin{prp}\label{strongmonoclassifier}
Let $P:\ct{B}\op\futo\Ct{Pos}$ be a variational implicational doctrine and
let \gp be a generic predicate on the object \ogp. If $\cmp{\gp}:X\to\ogp$
is a weak comprehending arrow for \gp, then 
$\ec{\tru}:(1,\top_{1})\moto(\ogp,\pgp)$ is a classifier of the
comprehending arrows in \qb{P}.
\end{prp}

\begin{proof}
First note a weak comprehending arrow of $\top_1$ is $\id{1}$. Recall that
$\tru :1\to\ogp$ satisfies $\top_1=P_{\tru}(\gp)$ and hence $\tru$ factors
through $\cmp{\gp}$ by the property of comprehension. Furthermore, by
\prx\ref{basce}(ii) the weak pullback of $\cmp{\gp}$ along $\tru$ is a
comprehending arrow of $\top_1$ which is  $\id{1}$. Hence $\cmp{\gp}$ can be
factor through $\tru$. This implies that the monic map
$\ec{\tru}:(1,\top_1)\moto(\ogp,\pgp)$ is isomorphic to $\cmp{\gp}:(
X,\top_X)\moto(\ogp,\pgp)$. Now observe that the  arrows classified by
$\ec{\tru}:(1,\top_{1})\moto(\ogp,\pgp)$ are those which appear in an
equaliser
\[\xymatrix@R=1ex@C=2em{
&&&\quad& (1,\top_{1})\moar[drr]^-{\ec{\tru}}
&\quad\\
(X,\sigma)\moar[rr]_-{}&&
(B,\rho)\ar[rrrr]_-{}\ar[rru]^-{\ec{!_B}}&&&&(\ogp,\pgp)}\]
by \prx\ref{basce}(ii) since \qd{P} has
comprehensive diagonals.
\end{proof}

A situation similar to that of a generic predicate can be developed
in case the doctrine $P$ has weak powerobjects. The full strength of
first order operators is required, though.

\begin{thm}\label{wpow}
Let $P:\ct{B}\op\futo\Ct{Pos}$ be a variational
implicational universal doctrine. Then the following are equivalent:
\begin{enumerate}
\item \qd{P} has powerobjects.
\item $P$ has weak powerobjects.
\end{enumerate}
\end{thm}

\begin{proof}
\entail{(i)}{(ii)} Immediate.

\entail{(ii)}{(i)} 
In line with the proof of \prx\ref{cartcloch}, note that the
hypothesis provides a natural surjection 
\[\xymatrix{\qb{P}(\Yo(\blank),(\pow{B},\delta_{\pow{B}}))\isar[r]
\ntar@{->>}[rrd]_{e}&
\ct{B}(\blank,\pow{B})\ntar@{->>}[r]&\udl{P(\blank\times B)}\isar[d]\\
&&\udl{\qd{P}(\Yo(\blank)\times(B,\delta_B))}}\]
for each object of the form $(B,\delta_B)$ in \qb{P}.
For two arrows $f,g:D\to\pow{B}$, the condition that $e_D(f)=e_D(g)$
is equivalent to
$P_{f\times\id{B}}(\mbs{B})=P_{g\times\id{B}}(\mbs{B})$ which can be
written in logical form as
\[y\type D\mid\top\vdash\Forall{x\type B}(f(y)\rmbs{B}x\Iff
g(y)\rmbs{B}x)\]
where we write the relation \mbs{B} infix.
The relation $\varsigma$ determined by the formula 
$\Forall{x\type B}(w\rmbs{B}x\Iff w\rmbs{B}x)$ in the variable
$w\type\pow{B}$ is clearly a \qd{P}-equivalence relation, and \mbs{B}
is a descent datum for $\varsigma\times\delta_B$. Hence it gives a
powerobject of $(B,\delta_B)$.

The case of an arbitrary object $(B,\rho)$ is dealt with similarly 
to the proof of \prx\ref{cartcloch}: the object appears in a
coequaliser
\[\xymatrix@C=2em{(R,\delta_R)
\ar@<.5ex>[rr]^-{[r_1]}\ar@<-.5ex>[rr]_-{[r_2]}&&
(B,\delta_B)\onar[r]&(B,\rho)}\]
So one obtains a diagram of natural transformations
\[\xymatrix@C=4em@R=3.5em{\qb{P}(\Yo(\blank),(X,\theta))\isar[r]
\ar[d]^-{[h]\circ\blank}&
\udl{\qd{P}(\Yo(\blank)\times(B,\rho))}\moar[d]\\
\udl{\qb{P}(\Yo(\blank),(\pow{B},\varsigma_B))}\isar[r]
\ar@<.8ex>[d]^-{[r_2']\circ\blank}\ar@<-.8ex>[d]_-{[r_1']\circ\blank}&
\udl{\qd{P}(\Yo(\blank)\times(B,\delta_B))}
\ar@<.8ex>[d]^-{\udl{\qd{P}(\ec{\id{}\times r_2})}}
\ar@<-.8ex>[d]_-{\udl{\qd{P}(\ec{\id{}\times r_1})}}\\
\qb{P}(\Yo(\blank),(\pow{R},\varsigma_R))\isar[r]&
\udl{\qd{P}(\Yo(\blank)\times(R,\delta_R))}}\]
where the equaliser
\[\xymatrix@C=2em{(X,\theta)\ar[r]^-{\ec{h}}&(\pow{B},\varsigma_B)
\ar@<.5ex>[rr]^-{\ec{r_1'}}\ar@<-.5ex>[rr]_-{\ec{r_2'}}&&
(\pow{R},\varsigma_R)}\]
provides the universal solution for \udl{\qd{P}(\blank\times(B,\rho))}.
\end{proof}

The added strength to the logic of a doctrine brought by the presence of
powerobjects resembles the standards of elementary set theory. For instance,
it is possible to repeat the argument of Tarksi's Fixpoint Theorem.

\begin{lem}\label{traclo}
Let $P:\ct{B}\op\futo\Ct{Pos}$ be an implicational universal doctrine.
Let $B$ be an object in \ct{B}. Then the following hold:
\begin{enumerate}[{\rm(i)}]
\item A $($strong$)$ powerobject \pow{B} bears an internal structure of
complete lattice $(\pow{B},\subseteq,\glb)$.
\item There is an internal operator 
$\ler:\pow{B\times B}\to\pow{B\times B}$ that satisfies the following
formula
\[\begin{array}{@{}l@{}l@{}}
r\type\pow{B\times B}\mid\top\vdash{}&{}\ier(\ler(r))\Land
r\subseteq\ler(r)\Land\\[1ex]
&\Land\Forall{e\type\pow{B\times B}}((\ler(e)\Land r\subseteq e)\Implies
\ler(r)\subseteq e)
\end{array}\]
where $\ier(e)$ is the formula which states that $e\type\pow{B\times B}$
is an $($internal$)$ equivalence relation, \ie $\ier(e)$ is the conjunction
of the three formulas 
\begin{description}
\item[$\stndop{rfl}(e):=$] $\Forall{x\type B}(e\rmbs{B}\ple{x,x})$
\item[$\stndop{sym}(e):=$] $\Forall{x,x'\type B}
(e\rmbs{B}\ple{x,x'}\Implies e\rmbs{B}\ple{x',x})$
\item[$\stndop{trn}(e):=$] $\Forall{x,x',x''\type B}
((e\rmbs{B}\ple{x,x'}\Land e\rmbs{B}\ple{x',x''})\Implies
e\rmbs{B}\ple{x,x''})$
\end{description}
\end{enumerate}
\end{lem}

\begin{proof}
It is all a very simple exercise in the use of the logic of the doctrine. We
just sketch the basic points.

\noindent (i) The order relation $\subseteq$ is the predicate in
$P(\pow{B}\times\pow{B})$ determined by the formula
\[\Forall{x\type B}(w\rmbs{B}x\Implies w'\rmbs{B}x)\]
in the variables $w,w'\type\pow{B}$. It is immediate to check that the
relation satisfies reflexivity, antisymmetry and transitivity.

The arrow $\glb:\pow{\pow{B}}\to\pow{B}$ for the greatest lower bound is
the one classifying the predicate
\[\Forall{w\type\pow{B}}(W\rmbs{\pow{B}}w\Implies w\rmbs{B}x)\]
in the variable $W\type\pow{\pow{B}}$ and $x\type B$.

\noindent(ii) Consider the arrow 
$h:\pow{B\times B}\to\pow{\pow{B\times B}}$ which
classifies the relation on $\pow{B\times B}\times\pow{B\times B}$
determined by the formula $\ier(e)\Land r\subseteq e$ in the variables
$e,r\type\pow{B\times B}$. Take the composition
\[\xymatrix@C=3em{\pow{B\times B}\ar@/_10pt/[rr]_-{\ler}\ar[r]^-{h}&
\pow{\pow{B\times B}}\ar[r]^-{\glb}&\pow{B\times B}.}\]
Use the logic to check that 
$r\type\pow{B\times B}\mid\top\vdash\ier(\ler(r))$ and condition (ii).
\end{proof}

\begin{prp}\label{coeqer}
Let $P:\ct{B}\op\futo\Ct{Pos}$ be an implicational, universal \eed with
comprehensive diagonals, powerbojects, and quotients of $P$-equivalence
relations. Then every parallel pair of arrows
\[\xymatrix@C=3em{
A\ar@<1.1ex>[r]^-{f}\ar@<-1.1ex>[r]_-{g}&B}\]
in \ct{B} has a coequaliser.
\end{prp}

\begin{proof}
Consider the relation $\rho$ determined by the formula 
\[\Exists{y\type A}(f(y)\seq{B}x\Land g(y)\seq{B}x')\]
in the variables $x,x'\type B$. Let $\rho^{\#}:1\to\pow{B\times B}$ be the
arrow which classifies $\rho$. Let $\rho'$ be the relation classifed by
$\ler\circ\rho^{\#}:1\to\pow{B\times B}$. By \lmx\ref{traclo}(ii), $\rho'$
is a $P$-equivalence relation and $\rho\leq_{B\times B}\rho'$. Moreover, if
the $P$-equivalence relation $\sigma$ is such that 
$\rho\leq_{B\times B}\sigma$, then $\rho'\leq_{B\times B}\sigma$.

Consider the quotient $q:B\to B/\rho'$. By definition, 
$\rho'\leq_{B\times B}P_{q\times q}(\delta_{A/\rho'})$. So
\[y\type A\mid\top\vdash q(f(y))\seq{B/\rho'}q(g(y)).\]
Since $P$ has comprehensive diagonals, $qf=qg$ in \ct{B}.

Consider an arrow $h:B/\rho'\to C$ such that $hf=hg$. So
\[x,x'\type B,y\type A\mid f(y)\seq{B}x\Land g(y)\seq{B}x'
\vdash h(x)\seq{C}h(x')\]
which states that $\rho\leq_{B\times B}P_{h\times h}(\delta_C)$. Since
$P_{h\times h}(\delta_C)$ is a $P$-equivalence relation, 
$\rho'\leq_{B\times B}P_{h\times h}(\delta_{C})$. Hence, there is a unique
arrow $k:A/\rho'\to C$ such that $hq=k$, by the universal property of the
quotient $q:B\to B/\rho'$.
\end{proof}

\section{When an e.q.c. is a quasitopos}\label{when}

A \dfn{quasitopos} is a category \ct{A} which
\begin{enumerate}[1]
\item has finite limits;
\item has finite colimits;
\item is locally cartesian closed;
\item has a classifier of \dfn{strong} monos, \ie monos which are orthogonal
to epimorphisms.
\end{enumerate}
A basic reference for quasitoposes is \cite{WylerO:lectq}. 

Since the class \Stg{\ct{A}} of strong monos of the quasitopos \ct{A} is
closed under identities and pullbacks, there is an \ed
$\Stg{\ct{A}}:\ct{A}\op\futo\Ct{Pos}$. Also, the quasitopos \ct{A} has
effective quotients of strong equivalence relations, so they are
precisely the quotients of the \Stg{\ct{A}}-equivalence relations.
We have seen that \eqc{}s generalise the exact completion of
\cite{CarboniA:freecl,CarboniA:regec}. As in general a quasitopos is not
exact, if there are doctrines
of strong monos on quasitoposes which are \eqc, then  the construction of \eqc would definitely prove a productive
extension of the exact completion.

To approach the problem, consider an \fo\eqc $Q:\ct{C}\op\futo\Ct{Pos}$. Let
$\Yo:\ct{B}\futo\ct{C}$ be the embedding of the full subcategory \ct{B} of
the $Q$-projectives, write $P$ for the composition
$Q\Yo\op:\ct{B}\op\futo\ct{Pos}$. We would like to determine which
properties of the \ed $P:\ct{B}\op\futo\ct{Pos}$ are required to make sure
that \ct{C} is a quasitopos and $Q$ is the doctrine of its strong
subobjects. 

To start with, note that
\begin{enumerate}[{\rm(i)}]
\item $P$ is a universal and implicational \eed.
\end{enumerate}
It is, in fact, \fo, but for large part of the present argument condition
(i) will suffice.

By \thx\ref{tclassifier}, condition 4.  for a quasitopos requires that
\begin{enumerate}[{\rm(i)}]
\item[(ii)] $P$ have a generic predicate.
\end{enumerate}

Furthermore, thanks to \thx\ref{comprequo} since $Q$ is m-variational, then
\begin{enumerate}[{\rm(i)}]
\item[(iii)] $P$ is variational.
\end{enumerate}
 From this it
follows also that \ct{B} has weak equalisers, and weak pullbacks. Hence $Q$
is the m-variational doctrine of equalisers in \ct{C}.

Since \ct{C} has pullbacks, by \crx\ref{lccmain}, condition 3. in the
definition of quasitopos requires that
\begin{enumerate}[{\rm(i)}]
\item[(iv)] for each object $B$ in \ct{B}, the functor
\ciwe{$\comc{\Yo}{B}:\comc{\ct{B}}{B}\futo\comc{\ct{C}}{\Yo(B)}$}.
\end{enumerate}
Furthermore, since $P$ is a variational \eed, \prx\ref{abcde+} yields that $P$ is
obtained from a topology on $\Wsb{}:\ct{B}\op\futo\Ct{Pos}$.

Condition 2. remains to be considered: already (i)--(iv)
ensures that $Q$ verifies the hypotheses of \prx\ref{coeqer}. So \ct{C} has
coequalisers. But, by \thx\ref{chardist}, finite coproducts (a quasitopos
is, in fact, a distributive category) require that
\begin{enumerate}[{\rm(i)}]
\item[(v)] $P$ is \dd.
\end{enumerate}

The next theorem summarises the previous discussion. Recall from \cite{MaiettiME:relqcv}  that if $P$ is an m-variational \ed, then comprehending arrows and $P$-surjective arrows form a factorisation system; moreover, if $P$ has a generic predicate then $P$-surjections and epimorphisms coincide, making the class of comprehending arrows coincide with that of strong monomorphisms.

\begin{thm}\label{quasitop}
Let $Q:\ct{C}\op\futo\Ct{Pos}$ be an \eqc, and let $\Yo:\ct{B}\futo\ct{C}$
be the embedding of the full subcategory of the $Q$-projectives, and let
$P:\ct{B}\op\futo\Ct{Pos}$ be the \ed $Q\Yo\op$.
Then the following are equivalent:
\begin{enumerate}[{\rm(a)}]
\item \ct{C} is a quasitopos and the doctrine $Q:\ct{C}\futo\Ct{Pos}$
is the m-variational doctrine of strong subobjects in \ct{C}.
\item $P:\ct{B}\op\futo\Ct{Pos}$ satisfies the following conditions:
\begin{enumerate}[(i)]
\item $P$ is \fo\dd;
\item $P$ has a generic predicate arrow in $P$;
\item $P$ is variational;
\item for each object $B$ in \ct{B}, the functor
\ciwe{$\comc{\Yo}{B}:\comc{\ct{B}}{B}\futo\comc{\ct{C}}{\Yo(B)}$}.
\end{enumerate}
\end{enumerate}
\end{thm}

\begin{proof}
One just need to add to the previous discussion the remark that, in a
quasitopos, the collection of equalisers coincides with that of strong
monos.
\end{proof}

Given the above characterization,  we call \dfn{quasitopos-representing indexed partial order} (quipo, for short)
a \ed $P:\ct{B}\op\futo\Ct{Pos}$ satisfying the above conditions (i),(ii),(iii),(iv).

\begin{rem}
A quipo is a particular tripos. Indeed, by definition a quipo
$P:\ct{B}\op\futo\Ct{Pos}$ is \fo Since \ct{B} is weakly cartesian
closed and $P$ has a generic predicate, by \thx\ref{wpow} $P$ has weak
powerobjects, \ie $P$ satisfies the Comprehension Axiom of 
\cite{PittsA:tritir}.

In fact, a tripos $P:\ct{B}\op\futo\Ct{Pos}$ is a quipo exactly when
$P$ is  a variational \dd and, for each object $B$ in \ct{B}, the functor 
\ciwe{$\comc{\Yo}{B}:\comc{\ct{B}}{B}\futo\comc{\qb{P}}{\Yo(B)}$}.
\end{rem}

\begin{exm}
An example of quipo  is the doctrine
$\left(\Wsb{\ct{PA}}\right)_j:\ct{PA}\op\futo\ct{Set}$ of \xsx\ref{fstex}
since its \eqc is the doctrine of strong subobjects of the category of
assemblies \ct{Asm}.

In particular, the doctrine $\Wsb{\ct{PA}}:\ct{PA}\op\futo\Ct{Pos}$ on \ct{PA} is also a quipo
but its elementary quotient completion is even a realizablity topos  \cite{RobinsonE:colcet}.
\end{exm}

\thx\ref{quasitop} can be extended to
consider arithmetic quasitoposes, \ie with a natural number object.

\begin{cor}\label{main2}
Let $P:\ct{B}\op\futo\Ct{Pos}$ be a quipo. Then the following are
equivalent: 
\begin{enumerate}[{\rm(i)}]
\item The category \ct{C} has a natural number object.
\item The \eqc $\qd{P}:\qb{P}\op\futo\Ct{Pos}$ is an arithmetic quasitopos.
\end{enumerate}
\end{cor}

\begin{proof}
It is easy to see that weak cartesian closure ensures that a natural
number object in \ct{C} is necessarily parametric. Hence the
conclusion follows from \cite[\lmx3.5]{MaiettiME:eleqcc}.
\end{proof}

Observe that we can use \thx\ref{quasitop} to produce new quasitoposes: the
doctrine $\Stg{\ct{A}}:\ct{A}\op\futo\Ct{Set}$ of strong subobjects on a
quasitopos \ct{A} is a quipo, as one easily verifies. Of course, taking
\qb{\Stg{\ct{A}}} would produce a new quasitopos---and a new doctrine
\qd{\Stg{\ct{A}}}. This is a very relevant difference from the construction
\ecb{P} for a tripos $P$, as we shall see in the following.

\subsection{Quasitoposes as categories of partial equivalence relations}
As already observed, a tripos is not, in general, a quipo, since it lacks
full comprehension. However, as shown in \cite{MaiettiME:eleqc}, full
comprehension can be freely added to a doctrine as well as comprehensive
diagonals. In this section, we characterize those triposes $P$ with full
comprehension that are also quipos, and hence whose elementary  quotient
completion yields a quasitopos. The base category of the  elementary
quotient completion construction applied to a doctrine obtained by freely
adding full comprehension to $P$ is known as the category of partial
equivalence relations of $P$, also defined in \cite[\S5.1]{FreyJ:triqat}, as
first observed in \cite[p.~17]{MaiettiME:eleqc}.

We briefly  recall how to complete a primary doctrine
$P:\ct{B}\op\futo\Ct{Pos}$ to a doctrine $P\cmpr:\ct{G}\op\futo\Ct{Pos}$
which admits (strong) comprehension as described in
\cite{MaiettiME:quofcm}.

The category \ct{G} consists of objects $(B,\beta)$ where $\beta$ is a
predicate on $B$, and arrows $f:(B,\beta)\to(B',\beta')$ are $f:B\to B'$
in \ct{B} such that $\beta\leq_B P_f(\beta')$. The functor $p$ is the
first projection on objects. It is easy to see that the category \ct{G} has
products and the functor $p$ preserves them, see
\cite{MaiettiME:quofcm}. Furthermore, $P\cmpr$ maps each $(A,\alpha)$ to
$P\cmpr(A,\alpha)=\{\phi\in P(A)\mid \phi\le\alpha\}$ and each
$f:(A,\alpha)\to(B,\beta)$ to the function $P\cmpr(f):P\cmpr(B,\beta)\to
P\cmpr(A,\alpha)$ determined by the assignment $\psi\mapsto
P_f(\psi)\wedge\alpha$. For $\phi$ in $P\cmpr(A,\alpha)$ it is
$\cmp{\phi}=\id{A}:(A,\phi)\to (A,\alpha)$.

One then checks that, in case $P$ is elementary, also $P\cmpr$ is elementary
with 
$\delta_{(B,\alpha)}\colon=(P_{p_1}\alpha\Land P_{p_2}\alpha)\Land\delta_{B}$. 

To force comprehensive diagonals requires quotienting the homsets of the
base category \ct{B} with respect to the equivalence relation on arrows 
$f,g:(B,\alpha)\to(B',\alpha')$ when
\[x\type B\mid\alpha(x)\vdash f(x)\seq{B'}g(x)\]
as described in \cite[\S5]{MaiettiME:eleqc}.
The quotient functor $Q:\ct{G}\futo\ct{G}\ext$ preserves finite products,
and it is easy to see that the doctrine $P$ factors through $Q$ to give an
\ed $P\extcmpr:\ct{G}\ext\op\futo\Ct{Pos}$.

The results in \cite{MaiettiME:eleqc} about the free constructions just
introduced can be summarised stating that $P\cmpr$ completes a primary
doctrine $P$ with respect to primary doctrines admitting full comprehension. On
the other hand, $P\extcmpr$ completes the \ed $P$ to admit full comprehension
forcing at the same time comprehensive diagonals.

In both cases, it is also possible to determine the generating doctrine, up
to splitting of idempotents. In a primary doctrine
$P:\ct{B}\op\futo\Ct{Pos}$, an object $K$ in \ct{B} is \dfn{$P$-injective}
if in every diagram with a comprehending arrow
\begin{equation}\label{tinje}
\xymatrix@C=3em@R=2ex{X\ar[r]^-{\cmp{\alpha}}\ar[rd]_-{f}&Y\\
&K}
\end{equation}
there is a filler $g:Y\to K$.

A doctrine $P:\ct{B}\op\futo\Ct{Pos}$ has \dfn{enough $P$-injectives} if
every object $B$ in \ct{B} appears as the domain $\cmp{\alpha}:B\to K$ of a
comprehending arrow of $P$.

The notion of $P$-injectivity allows to recover the generating doctrine
of $P\extcmpr$ as follows (see also \cite[\thx5.13]{DagninoPDuality}).

\begin{thm}\label{pinj}
Let $E:\ct{C}\futo\Ct{Pos}$ be an \ed. Then the following are equivalent:
\begin{enumerate}[{\rm(i)}]
\item $E$ is equivalent to $P\cmpr$ for some \ed $P$.
\item $E$ satisfies the following properties:
\begin{enumerate}[(a)]
\item $E$ has full comprehension;
\item $E$ has enough $E$-injectives.
\end{enumerate}
\end{enumerate}
\end{thm}

\begin{proof}
\entail{(i)}{(ii)}
Condition (a) follows from the results in 
\cite[Theorems~4.8 and 5.7]{MaiettiME:eleqc}. It follows from the definition
of the category $\ct{G}\ext$ that in a diagram like \eqref{tinje}
\[\xymatrix@C=3em@R=2ex{(X,\alpha)\ar[r]^-{\ec{\id{X}}}\ar[rd]_-{\ec{f}}&
(X,\beta)\\
 &(K,\top_K).}\]
a filler can be chosen as $\ec{f}:(X,\alpha)\to(K,\top_K)$. So objects
$(K,\top_K)$ are $P\cmpr$-injective. Since every object $(B,\alpha)$
there is an arrow $\ec{\id{B}}:(B,\alpha)\to(B,\top_B)$, the doctrine
$P\extcmpr$ satisfies also condition (b).

\entail{(ii)}{(i)}
Let \ct{A} be the full subcategory in \ct{C} on the $E$-injectives of
$E$. Let $J:\ct{A}\futo\ct{C}$ the embedding and
$P=E\circ J\op:\ct{A}\op\futo\Ct{Pos}$ the \ed restricted to \ct{A}. 

By the universal property of $P\cmpr$ in \cite[\thx5.7]{MaiettiME:eleqc},
there is a homomorphism of \ed admitting comprehension
$(\overline{J},\bar{\jmath})=\overline{(J,\id{J})}:P\cmpr \to E$. 
 We intend to show 
that it is an equivalence of doctrines. To that goal, consider that $\overline{J}$
sends an object $(K,\alpha)$ to $C$ appearing in a comprehending arrow 
$\cmp{\alpha}:C\moto K$ of  $\alpha$.
Now, since any
object $C$ in \ct{C} appears in a comprehending arrow 
$\cmp{\alpha}:C\moto K$ of some predicate $\alpha$ on the $E$-injective
object $K$, then $\bar{J}$ is essentially surjective on objects. Using the
fact that every diagram 
\[\xymatrix@C=3em@R=3em{C\moar[r]^-{\cmp{\alpha}}\ar[d]_-{f}&K\\
C'\moar[r]^-{\cmp{\alpha'}}&K'}\]
has a filler, it is easy to show that the functor $\bar{J}$ is full and
faithful by full comprehension, as well as $\bar{\jmath}$ is an
isomorphism.
\end{proof}

From \thx\ref{pinj} it is immediate to  deduce  this other characterization:
\begin{thm}\label{pinje}
Let $E:\ct{C}\futo\Ct{Pos}$ be an \ed with comprehensive
diagonals. Then the following are equivalent:
\begin{enumerate}[{\rm(i)}]
\item $E$ is equivalent to $P\extcmpr$ for some \ed $P$.
\item $E$ satisfies the following properties:
\begin{enumerate}[(a)]
\item $E$ is m-variational;
\item $E$ has enough $E$-injectives.
\end{enumerate}
\end{enumerate}
\end{thm}

Before we embark on the proof of \thx\ref{quipeqc}, we introduce a
simplified notation for the constituents of the category \qb{P\extcmpr}, which amounts to be equivalent
also to \qb{P\cmpr}.  An
object of that category consists of a 
$P\extcmpr$-equivalence relation $\rho$ on an object $(B,\beta)$. By
reflexivity, $\beta=P_{\ple{\id{B},\id{B}}}(\rho)$. So there is no need to
transcribe $\beta$ in the object $((B,\beta),\rho)$ as long as one
remembers that $\rho$ need not satisfy reflexivity on $B$, and only symmetry
and transitivity are satisfied. So we write simply $(B,\rho)$---and
possibly call such an object a $P$-\dfn{partial equivalence relation}
($P$-per, for short). Also, for an arrow
$\ec{f}:((B,\beta),\rho)\to((B',\beta'),\rho')$ it is sufficient to verify
that
\[x\type B\mid\rho(x,x')\vdash\rho'(f(x),f(x')),\]
with a similar condition to express that $\ec{f}=\ec{g}$ in
\qb{P\extcmpr}. Again it is enough to use $\rho$ and leave the predicate
$\beta$ implicit: we write $\ec{f}:(B,\rho)\to(B',\rho')$.

The base category is often denoted with a reference to partial equivalence
relations; we shall maintain that by writing it as \perb{P}. But we shall
keep the notation \perd{P} for the doctrine on it. We shall refer to
$\perd{P}:\perb{P}\op\futo\Ct{Pos}$ as a \dfn{doctrine of pers}.

To denote the action of the embedding
$\Yo:\ct{G}\futo\perb{P}$, we write the object $\Yo(B,\beta)$ as
$(B,\delta_{\beta})$ where
$\delta_{\beta}\colon=(P_{\pr1}(\beta)\Land P_{\pr2}(\beta)) \Land\delta$, which gives in logical
notation
\[x_1,x_2\type B\mid x_1\seq{\beta}x_2\dashv\vdash(\ \beta(x_1) \Land \beta(x_2)\ ) \Land  x_1\seq{B}x_2.\]

\begin{exm}
There are several examples of the construction of \perd{P}, even though
often just the base category \perb{P} is presented.

The one that gives the name to the construction itself is that of the
category of partial equivalence relations (pers, for short) on a combinatory
algebra. Excellent references for (partial) combinatory algebras are
\cite[Chapter~1]{OostenJ:reaait} and \cite[Chapter~2]{ZoetouthJ:commrt}.

The first instance of pers appears in the fundamental paper 
\cite{ScottD:dattl}. That paper introduces \emph{restricted} equivalence
relations---the adjective \emph{restricted} would later become
\emph{partial} in \cite{HylandJ:efft} and the latter term stuck---on the
combinatory algebra on the  ``graph model'' $\pw\nat$ on the natural
numbers.
\end{exm}

\begin{rem}
Note that the construction we denoted as \perb{P} is denoted
$\textbf{\itshape Q}(P)$ in \cite[\dfx5.1]{FreyJ:triqat} for $P$ a tripos.
\end{rem}

As a direct consequence of \thx\ref{pinje}, it is possible to derive a
characterisation of the doctrines of the 
form $\perd{P}:\perb{P}\op\futo\Ct{Pos}$, extending the results in
\cite{MaiettiME:eleqc} about \eqc{}s.

\begin{thm}\label{quapers}
Let $Q:\ct{C}\futo\Ct{Pos}$ be an \ed. Then the following are equivalent:
\begin{enumerate}[{\rm(i)}]
\item $Q:\ct{C}\op\futo\Ct{Pos}$ is a doctrine of pers.
\item $Q:\ct{C}\op\futo\Ct{Pos}$ satisfies the following properties:
\begin{enumerate}[(a)]
\item $Q$ is m-variational;
\item $Q$ has quotients of effective descent;
\item the full subcategory \ct{A} on the $Q$-projectives is closed under
  finite products;
\item for every object $C$ there is a diagram in \ct{C}
\[\xymatrix@C=3em@R=3em{B\onar[r]^-{q}\moar[d]_-{c}&C\\
K&}\]
where $q$ is $Q$-quotient from a $Q$-projective $B$ and $c$ is a
comprehending arrow into a \restrict{Q}{\ct{A}}-injective $K$ in \ct{A}.
\end{enumerate}
\end{enumerate}
\end{thm}

\begin{proof}
\entail{(i)}{(ii)}
Conditions (a) and (b) follow from
\cite[\thx5.10]{MaiettiME:eleqc}. Condition (c) is ensured by
\thx\ref{PPO}. To see that objects of the form 
$(K,\delta_K)$ are \restrict{Q}{\ct{A}}-injective, consider a diagram like
\eqref{tinje} in the full subcategory \ct{A} on the $Q$-projectives
\[\xymatrix@C=1ex@R=1ex{
(X,\delta_{\alpha})\ar@/_/[rrd]_-{\ec{d}}&&&&&&
(X,\delta_{\beta})\ar@/^/[lld]^-{\ec{e}}\\
&&(X,\rho)\ar@/_/[llu]_-{\ec{d'}}\ar@(u,ur)[];[]^(.7){\id{}}
\ar[rr]^-{\ec{\id{X}}}\ar[rrdd]_-{\ec{f}}
&\kern6em&
(X,\sigma)\ar@/^/[rru]^-{\ec{e'}}\ar@(u,ul)[];[]_(.7){\id{}}
&&\\
\strut&&&&&\\
&&&&(K,\delta_K).}\]
taking into account that $Q$-projectives are retracts of
objects of the form $(A,\delta_{\gamma})$ for $\gamma$ a predicate on $A$ in
\restrict{Q}{\ct{A}}.
A filler can be chosen as $\ec{fde}:(X,\delta_{\beta})\to(K,\top_K)$.
So condition (d) follows from \thx\ref{pinje}.

\entail{(ii)}{(i)} It follows from \thx\ref{PPO} and  \thx\ref{pinje}.
Let $J:\ct{B}\futo\ct{C}$ the embedding of the full subcategory on the
\restrict{Q}{\ct{A}}-injectives. Since \ct{A} is closed under finite
products, \ct{B} is also closed under finite products.
Let $P=Q\circ J\op:\ct{B}\op\futo\Ct{Pos}$ the \ed restricted to \ct{B}. It
has comprehensive diagonals.
By the universal property of $P\extcmpr$, there is a homomorphism of \ed
admitting comprehension $P\extcmpr\to Q$ extending $(J,\id{J})$. In turn,
this extends to a homomorphism $\qd{P\extcmpr}\to Q$ of \ed admitting
comprehension and quotients of effective descent, by the universal property
of \qd{P\extcmpr}. Since condition (d) ensures that arrows in \ct{C} can be
described by arrows on the appropriate \restrict{Q}{\ct{A}}-injectives, the
conclusion follows easily.
\end{proof}

\begin{thm}\label{quipeqc}
Let $P:\ct{B}\op\futo\Ct{Pos}$ be a tripos. Then the following are
equivalent:
\begin{enumerate}[{\rm(i)}]
\item $P\extcmpr:\ct{G}\ext\op\futo\Ct{Pos}$ is a quipo.
\item \ct{B} is weakly cartesian closed with finite distributive coproducts.
\end{enumerate}
\end{thm}

\begin{proof}
\entail{(i)}{(ii)} Weak cartesian closure follows easily from the definition
of quipo using \rmx\ref{bwcc}. It is also easy to check that the coproduct
of objects $(B_1,\top)$ and $(B_2,\top)$ is $(B_1+B_2,\top)$.

\entail{(ii)}{(i)}
We first show that the doctrine $P$ is \dd, so to extend that property to
$P\extcmpr$.
By \prx\ref{gendd} the doctrine $P$ is distributive.

Each logical operator on $P$ extends to the corresponding one on
$P\extcmpr$. 

A pullback of two arrows $\ec{f}:(A,\alpha)\to(C,\gamma)$ and
$g:(B,\beta)\to(C,\gamma)$ is
\[\xymatrix@C=4em@R=2.5em{
(A\times B,\psi)\ar[r]^-{\ec{\pr2}}\ar[d]_-{\ec{\pr1}}&
(B,\beta)\ar[d]^-{\ec{g}}\\
(A,\alpha)\ar[r]^-{\ec{f}}&(C,\gamma)}\]
where the predicate $\psi$ is determined by the formula
$\alpha(x)\Land\beta(y)\Land f(x)\seq{C}g(y)$ in the variables $x\type A$
and $y\type B$.

The generic predicate for $P\extcmpr$ is essentially that of $P$: it is the
predicate \gp on the object $(\ogp,\top_{\ogp})$. Given a predicate
$\alpha$ on $(B,\beta)$, consider an arrow $\chi:B\to\ogp$ in \ct{B} such
that $P_{\chi}(\gp)=\alpha$. Then
$\ec{\chi}:(B,\beta)\to(\ogp,\top_{\ogp})$.

The relevant part of the proof is for relative (weak) exponentiation: write
$\Yo:\ct{G}\ext\futo\perb{P}$ for the embedding into the \eqc of
$P\extcmpr$.

Given an object $(B,\beta)$ of $\ct{G}\ext$, and  two objects  
$\ec{f}:(A,\alpha)\to(B,\beta)$ and $\ec{h}:(C,\sigma)\to(B,\delta_{\beta})$ in
\comc{\ct{G}\ext}{\Yo(B,\beta)}, one must find an
object $\ec{w}:(W,\omega)\to(B,\beta)$ in \comc{\ct{G}\ext}{B} together with a
natural surjection
\[\xymatrix@1@C=2.5em{\comc{\ct{G}}{(B,\beta)}(\blank,\ec{w})
\suar[r]&
\comc{\perb{P}}{(B,\delta_{\beta})}
(\Yo\ec{f}\times_{(B,\delta_{\beta})}\Yo(\blank),\ec{h}).}\]
So take an arbitrary object $(X,\xi)$ in \comc{\ct{G}}{(B,\beta)} and a
commutative diagram in \comc{\perb{P}}{(B,\delta_{\beta})}
\[\xymatrix@R=.1ex@C=.1ex{&\ &
(A\times X,\psi)\ar@/^/[rrrrd]^-{\ec{k}}
\ar[lldd]_-{\ec{\pr1}}\ar[rrdd]^-{\ec{\pr2}}&\ &&&&\\
&&&&&\qquad&(C,\sigma)\ar@/^/[llllddd]^-{\ec{h}}\\
(A,\delta_{\alpha})\ar[rrdd]^-{\ec{f}}&&&&
(X,\delta_{\xi})\ar[lldd]_-{\ec{g}}&&\\
\strut&&&&&&&\\
&&(B,\delta_{\beta})&&&&&}\]
where $\psi$, as before to describe pullbacks, is the predicate determined
by the formula $\alpha(a)\Land\xi(x)\Land f(a)\seq{B}g(x)$ in the variables
$a\type A$ and $x\type X$. Consider $W\colon=C^A$ and the predicate $\omega$
determined by the formula
\[\Forall{a\type A}(\alpha(a)\Land f(a)\seq{B}h(\ev(a,w))
\]
in the variable $w\type W$.
\end{proof}

\begin{dfnt}\label{defhtri}
An \dfn{h-tripos} is an \fo\ed $P:\ct{B}\op\futo\Ct{Pos}$ with a generic
predicate, on a weakly cartesian closed category \ct{B} with finite
coproducts.
\end{dfnt}

\begin{rem}
To justify the name, note that \dfx\ref{defhtri} can be stated equivalently
by requesting a tripos on a weakly cartesian closed category with finite
coproducts.

The prefix ``h-'', which indicates that the base category of the tripos is
weakly cartesian closed, is a reference to both the original
\emph{hyper\/}doctrines as in \cite{LawvereF:equhcs} and the 
\emph{higher types} in the base.

We stress again that the notion of tripos we use, is the one
which appears in \cite{PittsA:tritir,OostenJ:reaait}. So, like that, it
clashes with the original definition of tripos as in \cite{HylandJ:trit}.
\end{rem}

A direct consequence of \thx\ref{quipeqc} is the following remarkable
result.

\begin{thm}\label{triqua}
Let $P:\ct{B}\op\futo\Ct{Pos}$ be a tripos. Then the following are
equivalent:
\begin{enumerate}[{\rm(i)}]
\item \perb{P} is a quasitopos.
\item $P\extcmpr:\ct{G}\ext\op\futo\Ct{Pos}$ is a quipo.
\item $P$ is an h-tripos.
\end{enumerate}
\end{thm}

\begin{exms}\label{sndex}
\thx\ref{triqua} applies to the following examples in \xsx\ref{fstex}:
\begin{itemize}
\item The change of base of the doctrine
$\Sb{\ct{Set}}:\ct{Set}\op\futo\Ct{Pos}$ of subsets 
along the forgetful functor $U:\TZ\futo\ct{Set}$ is an h-tripos.
Note that the category \TZ is only weakly cartesian closed.
\item The doctrine of variations
$\Wsb{\Fam{\ct{L}}}:\Fam{\ct{L}}\op\futo\Ct{Pos}$ on \Fam{\ct{L}} is
$P\extcmpr$ for $P$ the tripos $\ct{L}^{(\blank)}$ of \cite{HylandJ:trit}
for a locale $\ct{L}$.
\end{itemize}
\end{exms}

\section{Comparison with the tripos-to-q-topos construction}

The reader familiar with the construction of the elementary topos \ecb{P}
represented by a tripos $P$ has probably recognised a strong similarity
between that and \perb{P}. In fact, the topos \ecb{P} is the exact
completion of the quasitopos \perb{P} regarded as a regular category
described in \cite{FreydP:cata,CarboniA:regec}, thanks to
\cite[\thx4.7]{MaiettiME:triech}. 

For the benefit of those readers who are not so familiar with the
construction of the topos \ecb{P}, we recall it here, framing it in a way
that is useful for our purposes, using the construction of the exact
completion of an \eed of \cite{MaiettiME:uniec}.

So, fix an \eed $D:\ct{B}\op\futo\Ct{Pos}$, and consider the category \ecb{D}
which consists of the following data:
\begin{description}
\item[Objects] are pairs $(B,\rho)$ where $\rho$ is a $D$-per on $B$.
\item[An arrow] $\phi:(B,\rho)\to(B',\rho')$ is a relation from $B$ to $B'$
  which satisfies the following conditions:
\[\begin{array}{@{}r@{{}\mid{}}l}
x\type B,x'\type B'&\phi(x,x')\vdash\rho(x,x)\Land\rho'(x',x')\\[1ex]
x_1,x_2\type B,x_1',x_2'\type B'&
\rho(x_1,x_2)\Land\phi(x_2,x_2')\Land\rho'(x_2',x_1')\vdash
\phi(x_1,x_1')\\[1ex]
x\type B,x_1',x_2'\type B'&
\phi(x,x_1')\Land\phi(x,x_2')\vdash\rho'(x_1',x_2')\\[1ex]
x\type B&\rho(x,x)\vdash\Exists{x'\type B'}\phi(x,x')
\end{array}\]
\item[Composition] is relational composition:
\begin{equation}\label{relc}
(\psi\circ\phi)(x,x'')\colon=
\Exists{x'\type B'}(\phi(x,x')\Land\psi(x',x'').
\end{equation}
\end{description}

In \cite{MaiettiME:uniec} it is shown that \ecb{D} is an exact category,
with its universal property with respect to \eed{}s. In particular, it is
reasonably straightforward how to extend the \eed $D$ along the embedding
$\ct{B}\futo\ecb{D}$ and check that it gives the same \eed as the
subobject doctrine $\Sb{\ecb{D}}:\ecb{D}\op\futo\Ct{Pos}$, see \cite[\prx3.5]{MaiettiME:uniec}. So, in
particular, the construction \ecb{} is idempotent.

The connection with the construction in \thx\ref{triqua} is more evident if
one notices that the doctrine $D\extcmpr$ is part of the construction of
\ecb{D}.
More specifically, following \cite[\dfx4.1]{MaiettiME:relqcv} we say that a
relation $\phi$ from $B$ to $B'$ in a \eed $D$
is \dfn{functional} if it satisfies the following conditions:
\begin{description}
\item[\dfn{total}:]
$x\type B,x_1',x_2'\type B'\mid
\phi(x,x_1')\Land\phi(x,x_2')\vdash x_1'\seq{B'}x_2'$
\item[\dfn{single-valued}:]
$x\type B\mid\top\vdash\Exists{x'\type B'}\phi(x,x')$
\end{description}
Composition of functional relations is again relational composition as in
\eqref{relc}.
Write \fb{D} for the category of functional relations on the objects of
\ct{B}. It is immediate to check that \fb{D} has finite products, computed
as in \ct{B}.

As usual in situations such as these, since $B\isto B\times1$, one can
consider predicates on $B$ as relations from $B$ to $1$. Thus relational
composition as in \eqref{relc} gives an action of a functional relation on a
predicate, producing a doctrine $\fd{D}:\fb{D}\op\futo\Ct{Pos}$.
As explaned in \cite[\rmx4.6]{MaiettiME:relqcv}, the doctrine \fd{D} is an
extension of $D$ in the following sense:

There is a functor $\kf{D}:\ct{B}\futo\fb{D}$ which is the identity
on objects and maps an arrow $f:X\to Y$ to its \dfn{graph}, \ie the
functional relation determined by the formula $f(x)\seq{Y}y$ in the
variables $x\type X$ and $y\type Y$. The functor $\kf{D}$ is
faithful if and only if $D$ has comprehensive diagonals.

It is easy to see that change of base along $\kf{D}$ applied to \fd{D}
gives the doctrine $D$ back:
\[\xymatrix@C=3em@R=2ex{\ct{B}\op\ar[dd]_-{\kf{D}\op}
\ar[rrd]^-{D}_-{}="a"&&\\
&&\Ct{Pos}\\
\fb{D}\ar[rru]_-{\fd{D}}^-{}="b"&\ar@{2-2-}"a";"b"&}\]

It is the universal homomorphism into \eed{}s which satisfy the
Rule of Unique Choice, as we show in \prx\ref{RUCequiv} after recalling the
necessary information.

\begin{rem}
The construction \fb{} is determined implicitly in
\cite[\thx4.18]{FreyJ:triqat}. It shows the basic relationship which
involves the tripos-to-topos construction \ecb{} and the tripos-to-q-topos
construction \qb{}, discarding the information about the doctrines. In terms
of the \eqc construction, that can be written as follows: given an \eed
$D:\ct{B}\op\futo\Ct{Pos}$, there is an equivalence of \eed{}s
\[\ecb{D}\equiv\fb{\perd{D}},\]
see \cite[\thx 4.7]{MaiettiME:triech}.
\end{rem}

\begin{rem}
A word of warning before continuing. The locution ``category of pers'' is a
fairly common misnomer: it names a category by mentioning only its
objects. It is used in the literature to refer to both the base category
\perb{P} of the doctrine \perd{P}, and the category \ecb{P}. In the present
paper, we use the term ``category of pers'' in accordance with
\cite{ScottD:dattl,HylandJ:efft}, thus in discordance with 
\cite{PittsA:tritir,OostenJ:reaait}.
\end{rem}

In preparation for \prx\ref{RUCequiv}, recall from
\cite{MaiettiME:relqcv} that an \eed $D:\ct{B}\op\futo\Ct{Pos}$ satisfies
the \dfn{Rule of Unique Choice} \ruc when the following property holds:
\begin{quote}
For every functional
relation $\phi$ from $A$ to $B$ there is an arrow $f:A\to B$ such that 
$a\type A\mid\top\vdash\phi(a,f(a))$.
\end{quote}
Note that functionality of $\phi$ ensures that 
$a\type A,b\type B\mid\phi(a,b)\dashv\vdash f(a)\seq{B}b$. Again, uniqueness
of the arrows would be equivalent to the fact that $D$ has comprehensive
diagonals.

\begin{rem}
The property \ruc is weaker than the \dfn{Rule of Choice} \rc
which, always for an \eed $D:\ct{B}\op\futo\Ct{Pos}$, is the following
requirement:
\begin{quote}
For every total relation $\phi$
from $A$ to $B$ there is an arrow $f:A\to B$ such that 
$a\type A\mid\top\vdash\phi(a,f(a))$. 
\end{quote}
But there is a strong relationship between \ruc and \rc, which is relevant
for the present work: for a variational \eed $D:\ct{B}\op\futo\Ct{Pos}$, the
\eqc $\qd{D}:\qb{D}\op\futo\Ct{Pos}$ satisfies \ruc if and only if $D$
satisfies \rc, as stated in \cite[\prx4.11]{MaiettiME:relqcv}.
\end{rem}

We present a characterisation which may contain folklore results, see
\cite[\thx4.9.4]{JacobsB:catltt}, \cite[\S7]{CarboniA:somfcr}, and
\cite[\S2.154--66]{FreydP:cata} for different presentations of similar 
results (see also \cite[\thx4.3]{DagninoPasquali2025}). Before that, we
adapt the notion of Higgs-complete object of 
\cite[\dfx4.9]{MenniM:clooec} to the present situation, but we shall then
pursue a different line of study.

\begin{dfnt}\label{hcom}
Let $D:\ct{B}\op\futo\Ct{Pos}$ be an \eed. An object $B$ of \ct{B} is
\dfn{Higgs-complete} if, for every functional relation $\phi$ from $A$ to
$B$ there is an arrow $f:A\to B$ such that
\[x\type A,y\type B\mid\phi(x,y)\dashv\vdash(f(x)\seq{B}y).\]
\end{dfnt}

Write \crb{D} for the full subcategory of \ct{B} on the Higgs-complete
objects and $\crd{D}:\crb{D}\op\futo\Ct{Pos}$ for the restriction 
$\restrict{D}{\crb{D}}:\crb{D}\op\futo\Ct{Pos}$.
The following property is straightforward.

\begin{prp}\label{hcruc}
If $D:\ct{B}\op\futo\Ct{Pos}$ is an \eed with comprehensive diagonals, then 
$\crd{D}:\crb{D}\op\futo\Ct{Pos}$ is an \eed and satisfies
\ruc.
\end{prp}

\begin{proof}
One only needs to check that a functional relation from $A$ to 
$B_1\times B_2$ is the meet of two functional relations, one from $A$ to 
$B_1$, the other from $A$ to $B_2$. So a product of two Higgs-complete
objects is also Higgs-complete.
\end{proof}

\begin{prp}\label{RUCequiv}
Let $D:\ct{B}\op\futo\Ct{Pos}$ be an \eed with comprehensive diagonals. The
following are equivalent:
\begin{enumerate}[{\rm(i)}]
\item The \eed $D$ satisfies \ruc.
\item The embedding $\crb{D}\emto\ct{B}$ is an equivalence.
\item The functor $\kf{D}:\ct{B}\futo\fb{D}$ is an equivalence.
\end{enumerate}
\end{prp}

\begin{proof}
The three categories involved are the bases of the following doctrines
\[\xymatrix@C=3em@R=.5ex{\crb{D}\op\strut
\emar[dd]\ar[rrdd]^-{\crd{D}}_(.3){{}}="x"&&\\
&\kern4em&\strut\\
\ct{B}\op
\emar[dd]_-{\kf{D}\op}\ar[rr]^-{D}_(.3){{}}="a"^(.3){}="aa"&&\Ct{Pos}\\
&&\ar@{2-2-}"x";"aa"\strut\\
\fb{D}\op
\ar[rruu]_-{\fd{D}}^(.3){{}}="b"&\ar@{2-2-}"a";"b"&}\]
where also \kf{D} is marked as an embedding thanks to the hypothesis of
comprehensive diagonals for \ct{B}.

The two categories \crb{D} and \fb{D} achieve \ruc somehow for
different reasons: one cuts down the objects of \ct{B} to those which
satisfy \ruc for functional relations out of them; the other extends the
arrows of \ct{B} to make sure that there are enough of them to satisfy
\ruc. So the equivalences \nsc{\nsc{(ii)}{(i)}}{(iii)} are obvious.
\end{proof}

\begin{thm}\label{singt}
Let $D:\ct{B}\op\futo\Ct{Pos}$ be an \eed with comprehensive diagonals. The
following are equivalent:
\begin{enumerate}[{\rm(i)}]
\item The functor \kf{D} restricted to Higgs-complete objects
$\restrict{\kf{D}}{\crb{D}}:\crb{D}\futo\fb{D}$
is an equivalence.
\item The embedding of \eed{}s $\crb{D}\emto\ct{B}$
has a product-preserving left adjoint $S:\ct{B}\futo\crb{D}$ and each
component $u_B:B\to S(B)$ of the unit of the adjunction is such that
$D_{u_B}:D(S(B))\isto D(B)$. 
\end{enumerate}
\end{thm}

\begin{proof}
Note that in general, the functor
$\restrict{\kf{D}}{\crb{D}}:\crb{D}\futo\fb{D}$
is faithful because diagonals 
are comprehensive, and full by \dfx\ref{hcom}. So (i)
is equivalent to the request that the full embedding is essentially
surjective on objects.

\entail{(i)}{(ii)} By hypothesis, for every object $B$ in \ct{B},
there is a Higgs-complete objects \sng{B} together with an iso
$\Upsilon:\sng{B}\to B$ in \fb{D}.
For any Higgs-complete object $C$, there is a chain of natural
bijections
\[\xymatrix@C=1ex{
\ct{B}(B,C)\isar[rr]&\quad&\fb{D}(B,C)\isar[rr]_-{\blank\circ\Upsilon}&
\qquad&\fb{D}(\sng{B},C)\isar[rr]&\quad&\crb{D}(\sng{B},C)}\]
proving that the embedding has a left adjoint which preserves finite
products. It proves also that the unit of the reflection is an arrow
$u_B:B\to S_B$ whose graph is $\Upsilon$.

So the final condition
follows from the commutative diagram of posets
\[\xymatrix@C=1ex@R=1em{D(S_B)\ar@{2-2-}[d]\ar[rr]^-{D(u_B)}&&D(B)\\
\fd{D}(S_B)\ar[rr]_-{\fd{D}({\Upsilon})}&\kern5em&
\fd{D}(B).\ar@{2-2-}[u]}\]

\entail{(ii)}{(i)}
Clearly $D_{u}:\crd{D}S\ntto D$ is a natural isomorphism. So it
preserves both the elementary and the existential
structures. Preservation of the elementary structure states that
\[x'x'\type B\mid u_B(x)\seq{S(B)}u_B(x')\vdash x\seq{B}x',\]
\ie $u_b$ is $D$-monic. Preservation of the existential structures
ensures that
\[y\type S(B)\mid\Exists{x\type B}(y\seq{S(B)}u_B(x))\dashv\vdash
\Exists{x'\type S(B)}(y\seq{S(B)}x').\]

Hence $u_b$ is $D$-surjective. It follows that the graph of $u_B$ is
an iso in \fb{D}, thus proving that \restrict{\kf{D}}{\crb{D}} is an
equivalence.
\end{proof}

In case the doctrine $D:\ct{B}\op\futo\Ct{Pos}$ is m-variational, the list
of equivalent conditions in \prx\ref{RUCequiv} include two interesting
others, as we show in \prx\ref{RUCiso}. In order to do that, it is useful to
connect the notion of Higgs-complete object to that of coarse object as in
\cite[\dfx4.15]{FreyJ:triqat}, which we rephrase for a general \eed
$D:\ct{B}\op\futo\Ct{Pos}$. An object $B$ is \dfn{coarse} if, for every
arrow $f:A\to A'$ which is $D$-monic and $D$-surjective, the function
\[\blank\circ f:\ct{B}(A',B)\to\ct{B}(A,B)\]
is onto. Note that it is certainly one-to-one since $f$ is epic.

\begin{prp}\label{RUCiso}
Let $D:\ct{B}\op\futo\Ct{Pos}$ be an \eed with comprehensive diagonals. 
Let $B$ be an
object in \ct{B}. If $B$ is Higgs-complete, then $B$ is coarse.

Moreover, if $D$ is m-variational, then also the converse holds: if $B$ is
coarse, then $B$ is Higgs-complete.
\end{prp}

\begin{proof}
Assume $B$ is Higgs-complete.
Let $f:A\to A'$ be an arrow which is both $D$-monic and $D$-surjective, \ie
\[\begin{array}{@{}l}
x,x'\type A,y\type A'\mid y\seq{A'}f(x)\Land y\seq{A'}f(x')\vdash
x\seq{A}x'\\[1ex]
y\type A'\mid\top\vdash\Exists{x\type A}(y\seq{A'}f(x)).
\end{array}\]
For $g:A\to B$, consider the relation $\phi$ from $A'$ to $B$ determined by
the formula $\Exists{x\type A}(y\seq{A'}f(x)\Land z=g(x))$ in the variables
$y\type A'$ and $z\type B$. It is easy to check that the $\phi$ is
functional. Since $B$ is Higgs-complete, there is an arrow $h:A'\to B$ such
that $h\circ f=g$.

Suppose now that $D$ is m-variationaL. Assume $B$ is coarse.
Consider a functional relation $\phi$ from $A$ to $B$, and let
$\cmp{\phi}:X\to A\times B$ be a comprehending arrow for $\phi$. Since $D$
is m-variational
\begin{equation}\label{coahc}
x\type A,y\type B\mid\phi(x,y)\dashv\vdash
\Exists{w\type X}(\cmp{\phi}(w)\seq{A\times B}\ple{x,y}).
\end{equation}
Since $\phi$ is functional, the arrow $\pr1\cmp{\phi}:X\to A$ is $D$-monic
and $D$-surjective. Since $B$ is coarse, there is an arrow $f:A\to B$ such
that $f\pr1\cmp{\phi}=\pr2\cmp{\phi}$. From \eqref{coahc} it follows that
$\phi$ is the graph of $f$.
\end{proof}

There are a few direct consequences from the previous results.

\begin{cor}
Let $D:\ct{B}\op\futo\Ct{Pos}$ be an m-variational \eed. The following
are equivalent:
\begin{enumerate}[{\rm(i)}]
\item The \eed $D$ satisfies \ruc.
\item Every arrow which is both $D$-monic and $D$-surjective is iso.
\item The embedding $\cmp{}:D\emto\Sb{\ct{B}}$ into the doctrine of
subobjects of \ct{B} is an equivalence.
\end{enumerate}
If any of the above conditions holds, then \ct{B} is a regular category.
\end{cor}

\begin{proof}
\nsc{(i)}{(ii)} follows directly from \prx\ref{RUCiso}

\entail{(ii)}{(iii)}
Given a mono $m:A\moto B$, consider the predicate $\theta$ on $B$
determined by the formula $\Exists{a\type A}(m(a)\seq{B}b)$ in the variable
$b\type B$, and let $\cmp{\theta}:X\to B$ be its comprehending arrow. Since
$D_{m}(\theta)=\top$, there is a commutative diagram 
\[\xymatrix@C=4em@R=2em{A\moar[rd]_-{m}\ar@{-->}[r]^-{f}
&X\moar[d]^-{\cmp{\theta}}\\
&B}\]
Since $D$ is m-variational, $f$ is $D$-monic by \prx\ref{mono} and
\[b\type B\mid\Exists{x\type X}(\cmp{\theta}(x)\seq{B}b)\dashv\vdash
\Exists{a\type A}(m(a)\seq{B}b)\]
by the definition of $\theta$. So $f$ is $D$-surjective, hence it is
iso by (iii).

\entail{(iii)}{(i)} follows from \cite[\thx4.4.4]{JacobsB:catltt}.
\end{proof}

\begin{cor}\label{corquasitopr}
Let $P:\ct{B}\op\futo\Ct{Pos}$ is a quipo. 
\begin{enumerate}[{\rm(i)}]
\item $P$ satisfies \rc.
\item \qb{P} is a topos.
\end{enumerate}
\end{cor}

\begin{exms}
A notable example of the above theorem is the doctrine
of variations $\Wsb{\ct{PA}}:\ct{PA}\op\futo\Ct{Pos}$ on \ct{PA} whose
 elementary quotient completion is the realizability topos of
\cite{HylandJ:trit}.
\end{exms}

For the next results, recall from \cite{MaiettiME:triech} that an \eed
$P:\ct{C}\op\futo\Ct{Pos}$ is \dfn{equipped with \epso operators} if,
for every relation $\alpha$ from $B$ to $C$, there exists an arrow 
$\heps_\alpha:B\to C$ such that 
\[x\type B\mid
\Exists{y\type C}\alpha(x,y)\dashv\vdash\alpha(x,\heps_\alpha(x)).\]

\begin{cor}\label{cortriqua}
Let $P:\ct{B}\op\futo\Ct{Pos}$ be an h-tripos. Then the following are
equivalent:
\begin{enumerate}[{\rm(i)}]
\item $P$ is equipped with \epso operators.
\item \perb{P} is a topos.
\end{enumerate}
\end{cor}

\begin{proof}
Follows from \crx\ref{corquasitopr} by \cite[\thx5.15]{MaiettiME:triech}.
\end{proof}

\begin{exms}
A tripos of the form $\ct{L}^{(\blank)}$, for \ct{L} a locale, is
equipped with \epso operators if and only if it is a finite chain.
\end{exms}
 
\section{Conclusions}

We have presented a construction of a quasitopos from a suitable tripos,
either a quipo or an h-tripos, using the machinery of elementary quotient
completion developed in \cite{MaiettiME:quofcm,MaiettiME:eleqc}. 

The results extend three fundamental characterisations of exact completions:
Carboni and Vitale's characterisation in terms of projectives
\cite{CarboniA:regec}; Carboni and Rosolini's characterisation of locally
cartesian closed exact completions of weakly lex categories
\cite{RosoliniG:loccce}, as subsequently revised by Emmenegger, for which
we provide an alternative  proof to that given in \cite{MaiettiME:loccce},
relying solely on cartesian closedness; and Menni's characterisation of
toposes as exact completions \cite{MenniM:chalec}.

A remarkable example of a quasitopos arising as an elementary quotient
completion in type theory is provided by the setoid model \`a la Bishop over
Coquand–Huet–Paulin's Calculus of Inductive Constructions (\CTT); see
\cite{CoquandT:calc,CP89}. More precisely, the setoid model over \CTT
coincides with the elementary quotient completion \qb{F^\CTT}  of the doctrine 
$F^\CTT$ of propositions over \CTT, as described in \cite{MaiettiME:quofcm}.

The significance of this example lies in the fact that \CTT is the logical
foundation underlying the kernels of both proof assistants Rocq and
Lean. However, whereas Rocq handles quotients exclusively through setoids,
Lean provides built-in quotient types.

The results presented here pave the way for using the machinery of
elementary quotient completion to relate Rocq and Lean through their
categorical semantics.

This connection could be further developed and extended to other predicative
theories by generalising the construction of quasitoposes to accommodate
quotient models, such as those arising in \cite{m09}.

\section{Acknowledgements}

The paper contains results presented by the authors at several
international meetings in the past years, in particular at the Logic
Colloquium 2016 in Leeds and the International Category Theory Conference
``CT~2017'' in Vancouver, and during the Trimester Program on
``Types, Sets and Constructions'' at HIM in Bonn during summer of 2018. We
would like to thank the organisers of these events who gave us the
opportunity to present our results.


\begin{thebibliography}{10}

\bibitem{AdamekJ:algt}
J.~Ad\'{a}mek, J.~Rosick\'{y}, and E.~M. Vitale.
\newblock {\em Algebraic theories}, volume 184 of {\em Cambridge Tracts in
  Mathematics}.
\newblock Cambridge University Press, Cambridge, 2011.

\bibitem{Awodey2025}
S.~Awodey and J.~Emmenegger.
\newblock Toward the effective 2-topos.
\newblock {\em Mathematical Structures in Computer Science}, 35, 2025.

\bibitem{BarrM:toptt}
M.~Barr and C.~Wells.
\newblock {\em Toposes, {T}riples, and {T}heories}.
\newblock Springer-Verlag, 1984.

\bibitem{BauerA:equs}
A.~Bauer, L.~Birkedal, and D.S. Scott.
\newblock Equilogical spaces.
\newblock {\em Theoretical Computer Science}, 315(1):35--59, 2004.

\bibitem{vandenBerg20183137}
{B. van den} Berg and I.~Moerdijk.
\newblock Exact completion of path categories and algebraic set theory: Part
  {I}: Exact completion of path categories.
\newblock {\em Journal of Pure and Applied Algebra}, 222(10):3137 -- 3181,
  2018.

\bibitem{BirkhoffG:latt}
G.~Birkhoff.
\newblock {\em Lattice {T}heory}.
\newblock A.M.S. Colloquium Publications, 1940.

\bibitem{CarboniA:somfcr}
A.~Carboni.
\newblock Some free constructions in realizability and proof theory.
\newblock {\em J. Pure Appl. Algebra}, 103:117--148, 1995.

\bibitem{CarboniA:freecl}
A.~Carboni and R.~{Celia Magno}.
\newblock The free exact category on a left exact one.
\newblock {\em J. Aust. Math. Soc.}, 33(A):295--301, 1982.

\bibitem{CarboniA:catarp}
A.~Carboni, P.J. Freyd, and A.~Scedrov.
\newblock A categorical approach to realizability and polymorphic types.
\newblock In M.~Main, A.~Melton, M.~Mislove, and D.~Schmidt, editors, {\em
  Mathematical {F}oundations of {P}rogramming {L}anguage {S}emantics}, volume
  298 of {\em Lectures Notes in Comput. Sci.}, pages 23--42, New Orleans, 1988.
  Springer-Verlag.

\bibitem{CarboniA:intedc}
A.~Carboni, S.~Lack, and R.F.C. Walters.
\newblock Introduction to extensive and distributive categories.
\newblock {\em J. Pure Appl. Algebra}, 84:145--158, 1993.

\bibitem{CarboniA:eleccs}
A.~Carboni and S.~Mantovani.
\newblock An elementary characterization of categories of separated objects.
\newblock {\em J. Pure Appl. Algebra}, 89(1-2):63--92, 1993.

\bibitem{RosoliniG:loccce}
A.~Carboni and G.~Rosolini.
\newblock Locally {C}artesian closed exact completions.
\newblock {\em J. Pure Appl. Algebra}, 154(1-3):103--116, 2000.

\bibitem{CarboniA:regec}
A.~Carboni and E.M. Vitale.
\newblock Regular and exact completions.
\newblock {\em J. Pure Appl. Algebra}, 125:79--117, 1998.

\bibitem{cioffothesis}
C.~{Cioffo, Jr.}
\newblock {\em Homotopy setoids and generalized quotient completions}.
\newblock {Ph.D.} {T}hesis, University of Milan, 2022.

\bibitem{CioffoC:biaedq}
C.~{Cioffo, Jr.}
\newblock Biased elementary doctrines and quotient completions.
\newblock {\em J. Pure Appl. Algebra}, 229(6), 2025.

\bibitem{CoquandT:calc}
Th. Coquand and G.~Huet.
\newblock The {C}alculus of {C}onstructions.
\newblock {\em Inform. and Comput.}, 76:95--120, 1988.

\bibitem{CP89}
Th. {Coquand} and C.~{Paulin-Mohring}.
\newblock Inductively defined types.
\newblock In P.~{Martin-L{\"o}f} and G.~{Mints}, editors, {\em Proceedings of
  the International Conference on Computer Logic (Colog '88)}, volume 417 of
  {\em Lecture Notes in Computer Science}, pages 50--66, Berlin, Germany, 1990.
  Springer.

\bibitem{DagninoPasquali2025}
F.~Dagnino and F.~Pasquali.
\newblock Cauchy-completions and the rule of unique choice in relational
  doctrines.
\newblock {\em Theory and Applications of Categories}, 43(9):243--280, 2025.

\bibitem{dagnino2026projectivecoversdoctrinesalgebras}
F.~Dagnino and F.~Pasquali.
\newblock Projective covers, doctrines of algebras and the relational quotient
  completion.
\newblock Available as {\texttt{arXiv:2608.01915}}, 2026.

\bibitem{DagninoPDuality}
F.~Dagnino and F.~Pasquali.
\newblock Quotients-comprehensions duality in relational doctrines.
\newblock Available as {\texttt{arXiv:2608.15199}}, 2026.

\bibitem{DagninoP26}
F.~Dagnino and F.~Pasquali.
\newblock The relational quotient completion.
\newblock {\em Annals of Pure and Applied Logic}, 177(6):103728, 2026.

\bibitem{DagninoF:docmac}
F.~Dagnino and G.~Rosolini.
\newblock Doctrines, modalities and comonads.
\newblock {\em Math. Structures Comput. Sci.}, 31(7):769--798, 2021.

\bibitem{EmmeneggerJ:ontlc}
J.~Emmenegger.
\newblock On the local cartesian closure of exact completions.
\newblock {\em J. Pure Appl. Algebra}, 224(11):106414, 25, 2020.

\bibitem{EmmeneggerJ:eledac}
J.~Emmenegger, F.~Pasquali, and G.~Rosolini.
\newblock Elementary doctrines as coalgebras.
\newblock {\em J. Pure Appl. Algebra}, 224(12):106445, 16, 2020.

\bibitem{FreyJ:triqat}
J.~Frey.
\newblock Triposes, q-toposes and toposes.
\newblock {\em Ann. Pure Appl. Logic}, 166(2):232--259, 2015.

\bibitem{Frey2023}
J.~Frey.
\newblock Categories of partial equivalence relations as localizations.
\newblock {\em Journal of Pure and Applied Algebra}, 227(8), 2023.

\bibitem{FreydP:catcf}
P.J. Freyd and G.M. Kelly.
\newblock Categories of continuous functors {I}.
\newblock {\em J. Pure Appl. Algebra}, 2:169--191, 1972.

\bibitem{FreydP:cata}
P.J. Freyd and A.~Scedrov.
\newblock {\em Categories {A}llegories}.
\newblock North Holland, 1991.

\bibitem{GOGUEN1974513}
J.A. Goguen.
\newblock Concept representation in natural and artificial languages: Axioms,
  extensions and applications for fuzzy sets.
\newblock {\em International Journal of Man-Machine Studies}, 6(5):513--561,
  1974.

\bibitem{Gran20141771}
M.~Gran and Z.~Janelidze.
\newblock Star-regularity and regular completions.
\newblock {\em Journal of Pure and Applied Algebra}, 218(10):1771 -- 1782,
  2014.

\bibitem{GrandisM:weasec}
M.~Grandis.
\newblock Weak subobjects and the epi-monic completion of a category.
\newblock {\em J. Pure Appl. Algebra}, 154(1-3):193--212, 2000.

\bibitem{Hofstra2004129}
P.J.W. Hofstra.
\newblock Relative completions.
\newblock {\em Journal of Pure and Applied Algebra}, 192(1-3):129 -- 148, 2004.

\bibitem{HylandJ:efft}
J.M.E. Hyland.
\newblock The effective topos.
\newblock In A.S. Troelstra and D.~{van Dalen}, editors, {\em The {L}.{E}.{J}.
  {B}rouwer {C}entenary {S}ymposium}, pages 165--216. North Holland, 1982.

\bibitem{HylandJ:trit}
J.M.E. Hyland, P.T. Johnstone, and A.M. Pitts.
\newblock Tripos {T}heory.
\newblock {\em Math. Proc. Camb. Phil. Soc.}, 88:205--232, 1980.

\bibitem{JacobsB:catltt}
B.~Jacobs.
\newblock {\em Categorical {Logic and Type Theory}}, volume 141 of {\em Studies
  in Logic and the foundations of mathematics}.
\newblock North Holland, 1999.

\bibitem{JanelidzeG:facodi}
G.~Janelidze and W.~Tholen.
\newblock Facets of descent. {I}.
\newblock {\em Appl. Categ. Structures}, 2(3):245--281, 1994.

\bibitem{JohnstoneP:skeeli}
P.T. Johnstone.
\newblock {\em Sketches of an elephant: a topos theory compendium. {V}ol. 1},
  volume~43 of {\em Oxford Logic Guides}.
\newblock The Clarendon Press Oxford University Press, 2002.

\bibitem{Johnstone200187}
P.T. Johnstone, J.~Power, T.~Tsujishita, H.~Watanabe, and J.~Worrell.
\newblock On the structure of categories of coalgebras.
\newblock {\em Theoretical Computer Science}, 260(1-2):87 -- 117, 2001.

\bibitem{KainenP:weaaf}
P.C. Kainen.
\newblock Weak adjoint functors.
\newblock {\em Math. Z.}, 122:1--9, 1971.

\bibitem{Lack2001203}
S.~Lack and E.M. Vitale.
\newblock When do completion processes give rise to extensive categories?
\newblock {\em Journal of Pure and Applied Algebra}, 159(2-3):203 -- 230, 2001.

\bibitem{LawvereF:adjif}
F.W. Lawvere.
\newblock Adjointness in foundations.
\newblock {\em Dialectica}, 23:281--296, 1969.

\bibitem{LawvereF:equhcs}
F.W. Lawvere.
\newblock Equality in hyperdoctrines and comprehension schema as an adjoint
  functor.
\newblock In A.~Heller, editor, {\em Proc. {N}ew {Y}ork {S}ymposium on
  {A}pplication of {C}ategorical {A}lgebra}, pages 1--14. Amer.{M}ath.{S}oc.,
  1970.

\bibitem{LawvereF:adjiab}
F.W. Lawvere.
\newblock Adjoints in and among bicategories.
\newblock In P.~{Aglian\`o} A.~Ursini, editor, {\em Proceedings of the 1994
  Siena Conference in Memory of Roberto Magari. Logic and Algebra}, pages
  181--189. Marcel Dekker, Inc., 1996.

\bibitem{quarit}
M.~E. Maietti and D.~Trotta.
\newblock Quotients, pure existential completions and arithmetic universes.
\newblock {\em Theory Appl. Categ.}, 2023.

\bibitem{m09}
M.E. {Maietti}.
\newblock A minimalist two-level foundation for constructive mathematics.
\newblock {\em Ann. Pure Appl. Logic}, 160(3):319--354, 2009.

\bibitem{trip}
M.E. Maietti and S.~Maschio.
\newblock A predicative variant of a realizability tripos for the {M}inimalist
  {F}oundation.
\newblock {\em IfCoLog Journal of Logics and their Applications},
  3(4):595--667, 2016.

\bibitem{MaiettiME:triech}
M.E. Maietti, F.~Pasquali, and G.~Rosolini.
\newblock Triposes, exact completions, and {H}ilbert's
  {$\varepsilon$}-operator.
\newblock {\em Tbilisi Math. J.}, 10(3):141--166, December 2017.

\bibitem{MaiettiME:eleqcc}
M.E. Maietti, F.~Pasquali, and G.~Rosolini.
\newblock {Elementary Quotient Completions, Church's Thesis, and Partitioned
  Assemblies}.
\newblock {\em {Logical Methods in Computer Science}}, {Volume 15, Issue 2},
  2019.

\bibitem{MaiettiME:loccce}
M.E. Maietti, F.~Pasquali, and G.~Rosolini.
\newblock Local {C}artesian closure of elementary quotient completions.
\newblock {\em Theory and Applications of Categories}, 43(15):520--554, 2025.

\bibitem{MaiettiME:eleqc}
M.E. Maietti and G.~Rosolini.
\newblock Elementary quotient completion.
\newblock {\em Theory Appl. Categ.}, 27:445--463, 2013.

\bibitem{MaiettiME:quofcm}
M.E. Maietti and G.~Rosolini.
\newblock Quotient completion for the foundation of constructive mathematics.
\newblock {\em Log. Univers.}, 7(3):371--402, 2013.

\bibitem{MaiettiME:uniec}
M.E. Maietti and G.~Rosolini.
\newblock Unifying exact completions.
\newblock {\em Appl. Categ. Structures}, 23:43--52, 2015.

\bibitem{MaiettiME:relqcv}
M.E. {Maietti} and G.~{Rosolini}.
\newblock Relating quotient completions via categorical logic.
\newblock In Dieter Probst and Peter~Schuster (eds.), editors, {\em Concepts of
  Proof in Mathematics, Philosophy, and Computer Science}, pages 229--250. De
  Gruyter, 2016.

\bibitem{MenniM:clooec}
M.~{M}enni.
\newblock Closure operators in exact completions.
\newblock {\em Theory Appl. Categ.}, 8:522–540, 2001.

\bibitem{MenniM:chalec}
M.~{M}enni.
\newblock A characterization of the left exact categories whose exact
  completions are toposes.
\newblock {\em J. Pure Appl. Algebra}, 177(3):287--301, 2003.

\bibitem{OostenJ:reaahe}
{J. van} Oosten.
\newblock Realizability: a historical essay.
\newblock {\em Math.Struct.Comp.Sci.}, 12:239--263, 2002.

\bibitem{OostenJ:reaait}
{J. van} Oosten.
\newblock {\em Realizability: {An Introduction to its Categorical S}ide},
  volume 152 of {\em Studies in Logic and the Foundations of Mathematics}.
\newblock North Holland, 2008.

\bibitem{PareR:onac}
R.~Par\'{e}.
\newblock On absolute colimits.
\newblock {\em J. Algebra}, 19:80--95, 1971.

\bibitem{PasqualiF:cofced}
F.~Pasquali.
\newblock A co-free construction for elementary doctrines.
\newblock {\em Appl. Categ. Structures}, 23(1):29--41, 2015.

\bibitem{PittsA:fuzsdn}
A.M. Pitts.
\newblock Fuzzy sets do not form a topos.
\newblock {\em Fuzzy Sets and Systems}, 8(1):101--104, 1982.

\bibitem{PittsAM:catl}
A.M. Pitts.
\newblock Categorical logic.
\newblock In S.~Abramsky, D.M. Gabbay, and T.S.E. Maibaum, editors, {\em
  Handbook of Logic in Computer Science, Volume 5. Algebraic and Logical
  Structures}, chapter~2, pages 39--128. Oxford University Press, 2000.

\bibitem{PittsA:tritir}
A.M. Pitts.
\newblock Tripos theory in retrospect.
\newblock {\em Math. Structures Comput. Sci.}, 12(3):265--279, 2002.

\bibitem{RosoliniG:colcat}
E.P. Robinson and G.~Rosolini.
\newblock Colimit completions and the effective topos.
\newblock {\em J. Symb. Logic}, 55:678--699, 1990.

\bibitem{RobinsonE:colcet}
E.P. Robinson and G.~Rosolini.
\newblock Colimit completions and the effective topos.
\newblock {\em J. Symb. Logic}, 55:678--699, 1990.

\bibitem{RosickyJ:carcec}
J.~Rosick{\'y}.
\newblock Cartesian closed exact completions.
\newblock {\em J. Pure Appl. Algebra}, 142(3):261--270, 1999.

\bibitem{RosoliniG:equsfs}
G.~Rosolini.
\newblock Equilogical spaces and filter spaces.
\newblock {\em Rend. Circ. Mat. Palermo (2) Suppl.}, 64:157--175, 2000.

\bibitem{ScottD:dattl}
D.S. Scott.
\newblock Data types as lattices.
\newblock {\em SIAM J. Comput.}, 5(3):522--587, 1976.

\bibitem{ScottD:newcds}
D.S. Scott.
\newblock A new category? {D}omains, spaces and equivalence relations.
\newblock Available at {\texttt{http://www.cs.cmu.edu/Groups/LTC/}}, 1996.

\bibitem{WylerO:lectq}
O.~Wyler.
\newblock {\em Lecture Notes on Topoi and Quasitopoi}.
\newblock World Scientific Publishing Co. Inc., 1991.

\bibitem{ZoetouthJ:commrt}
J.~Zoetouth.
\newblock {\em Computability Models and Realizability Toposes}.
\newblock PhD thesis, University of Utrecht, 2022.

\end{thebibliography}
\end{document}